%% file: MathComp201116Revis2.tex
\documentclass{amsart}
\usepackage[normalem]{ulem}
\usepackage{amssymb,epsfig,mathrsfs,mathpazo}
\usepackage{color}
\usepackage{epsfig,bm}
\usepackage[multiple]{footmisc}
\usepackage{accents} 
\usepackage{mathtools} 

\newtheorem{theorem}{Theorem}[section]
\newtheorem{lemma}[theorem]{Lemma}
\newtheorem{proposition}[theorem]{Proposition}

\newtheorem{corollary}[theorem]{Corollary}

\theoremstyle{definition}

\theoremstyle{remark}
\newtheorem{remark}[theorem]{Remark}

\numberwithin{equation}{section}

\newcommand{\eps}{\varepsilon}

\newcommand{\R}{\mathbb R}

\newcommand{\N}{\mathbb N}

\newcommand{\cL}{\mathcal L}
\newcommand{\Lis}{\cL\mathrm{is}}

\newcommand{\identity}{\mathrm{Id}}

\DeclareMathOperator{\ran}{ran}

\DeclareMathOperator{\diam}{diam}
\DeclareMathOperator*{\argmin}{argmin}

\DeclareMathOperator{\divv}{div}

\DeclareMathOperator{\Span}{span}

\newcommand{\nrm}{| \! | \! |}
\newcommand{\nrmm}{| \! | \! | \! |}

\newcommand{\new}[1]{{\color{black}{#1}}}

\newcommand{\wnew}[1]{{\color{black}{#1}}}

\newcommand{\be}{\begin{equation}}
\newcommand{\ee}{\end{equation}}

\newcommand{\1}{\mathbb 1}

\newcommand{\tria}{{\mathcal T}}
\newcommand{\cP}{{\mathcal P}}

\newcommand{\oumlaut}{{\"o}}

\newcommand{\cons}{{\rm cons}}
\newcommand{\ap}{{\rm approx}}
\newcommand{\osc}{{\rm osc}}

\makeatletter
\@namedef{subjclassname@2020}{%
  \textup{2020} Mathematics Subject Classification}
\makeatother

\title{Accuracy controlled data assimilation for parabolic problems}

\date{\today}

\author{Wolfgang Dahmen}
\address{Mathematics Department, University of South Carolina, Columbia SC 29208}
\email{wolfgang.anton.dahmen@googlemail.com}
\author{Rob Stevenson}
\author{Jan Westerdiep}
\address{Korteweg-de Vries (KdV) Institute for Mathematics, University of Amsterdam, P.O. Box 94248, 1090 GE Amsterdam, The Netherlands.}
\email{r.p.stevenson@uva.nl, j.h.westerdiep@uva.nl}

\thanks{This research has been supported in part by NSF Grants DMS ID 1720297 and DMS ID 2012469, by the SmartState and Williams-Hedberg Foundation, and by the Netherlands Organization for Scientific Research (NWO) under contract.~no.~613.001.652}



\begin{document}
\begin{abstract}
This paper is concerned with the recovery of (approximate) solutions to parabolic problems from incomplete and possibly
inconsistent
observational data, given on a time-space cylinder that is a strict subset of the computational domain under consideration.
Unlike previous approaches to this and related problems our starting point is a {\em regularized least squares} formulation
in a continuous {\em infinite-dimensional} setting that is based on stable variational {\em time-space} formulations of the parabolic PDE.
This allows us to derive a priori as well as a posteriori error bounds for the recovered states with respect to a certain reference solution. In these bounds the regularization parameter is
disentangled from the underlying discretization. An important ingredient for the derivation of a posteriori bounds is the construction of suitable {\em Fortin operators}
which allow us to control oscillation errors stemming from the discretization of dual norms.
Moreover, the variational framework
allows us to contrive preconditioners for the discrete problems  whose application can be performed in linear time, and for which the condition numbers  of the
preconditioned systems are uniformly proportional to that of the regularized continuous problem.
 In particular,  we provide suitable
stopping criteria for the iterative solvers based on the a posteriori error bounds. The presented numerical experiments
quantify the theoretical findings and demonstrate the performance
of the numerical scheme in relation with the underlying discretization and regularization.
\end{abstract}

\maketitle

\maketitle

\section{Introduction}\label{sec:1}
\subsection{Background}\label{ssec:1.1}
Ever-increasing computational resources encourage considering more and more complex mathematical models
for simulating or predicting physical/technological processes. However, striving for increasing quantifiable accuracy
such models typically exhibit significant bias or are incomplete in that important model data or accurate constitutive laws
are missing. It is all too natural to gather complementary information from {\em data} provided by also ever-improving sensor capabilities.
Such a process of fusing models and data is often referred to as {\em data assimilation} which seems to originate from climatology
\cite{58.9,169.07}. In this context, streaming data are used to stabilize otherwise chaotic dynamical systems for prediction purposes,
typically with the aid of (statistical) filtering techniques.
While this is still an expanding and vibrant research area \cite{199.6}, the notion of data assimilation is by now understood in a
wider sense referring to efforts of improving quantifiable information extraction from synthesizing model-based and data driven approaches.
Incompleteness of underlying models or model deficiencies could come in different forms. For instance, one could lack model data such as
initial conditions, or the model involves uncalibrated coefficients  represented e.g.~by a parameter-dependent family of coefficients.

In this paper we focus on such a problem scenario where the physical law takes the form of a parabolic partial differential equation (PDE),
in the simplest case just the heat equation in combination with a known source term. We then assume that the state of interest, a (near-)solution
to this PDE, can be observed on some restricted {\em time-space cylinder} while its initial conditions are unknown.
We are then interested in recovering the partially observed state on the whole time-space domain from the given information.

This problem is known to be \wnew{(mildly)} ill-posed. This or related problems have been treated in numerous articles.
In particular, the recent work in \cite{35.925,35.926} proposes a finite element method with  built-in {\em mesh-dependent}
regularization terms \wnew{has been a primary motivation for the present paper.
Moreover, similar concepts for an analogous data-assimilation problem associated
with the wave equation have been applied in \cite{MR4221321}.} Considering first a semi-discretization in \cite{35.925}, the main results for a fully
discrete scheme in
\cite{35.926} provide a priori estimates for the recovered state on a domain that excludes a small region around the
location of initial data.

The results obtained in the present paper,
\wnew{although similar in nature, are \wnew{instead} based on a different approach and} exhibit a few noteworthy distinctions \wnew{explained  below}. \wnew{In fact,} our starting point is the formulation of a {\em regularized} estimation problem in terms of a {\em least squares} functional
in an {\em infinite-dimensional} function space setting.
\wnew{We postpone for a moment the particular role of the regularization parameter in the present context and remark first that our approach resembles a number of  other
prior studies of   ill-posed operator equations, that are also based on a {\em Tikhonov regularization} in terms of similar mixed variational formulations,
see e.g.~\cite{MR3461700,MR3808155,MR3885757,MR3079321,
MR3680382}. \wnew{These contributions are typically formulated in  more general setting
(see e.g.~\cite{MR3808155}), covering also problems
 that exhibit a stronger level of ill-posedness  such as
the Cauchy problem for second order elliptic equations, or
the backward heat equation. Although a direct comparison with these works is therefore
difficult, there are noteworthy relevant conceptual links as well as
distinctions that we will briefly comment on next.}

For instance, in \cite{MR3808155}, one arrives at a similar mixed formulation
as in the present paper exploiting then, however, just coercivity where,
for a regularizing term $\eps\|\cdot\|$,
the coercivity constant decreases proportionally to the regularization parameter.
The results in \cite{MR3461700} for related numerical schemes indeed confirm
a corresponding adverse dependence of error bounds on the regularization parameter.
Moreover, these bounds are obtained only under {\em additional regularity} assumptions.
In contrast, our approach is based on numerically realizing {\em inf-sup stability}  needed to handle dual
  norms,  resulting in the present context in $\eps$-independent error estimates without any a priori additional regularity assumptions.

This hints at the perhaps main principal distinction from the above prior related work.
Our guiding theme is that, how well one can solve the inverse problem,
depends on the
 condition of the corresponding
forward problem (already on an infinite-dimensional level), in the present context
a parabolic initial-boundary value problem. Specifically, this requires
identifying first a suitable pair $X, Y$ of (infinite-dimensional) trial- and test-spaces, for which   the forward operator   takes $X$ onto the {\em dual} $Y'$ of $Y$, without imposing any ``excess regularity'' assumptions on the solution
beyond membership to $X$. This is tantamount to a {\em stable variational formulation} of the forward problem in the sense of the Babuska-Necas Theorem.
We briefly refer to this as ``natural mapping properties''.
Drawing on the work in  \cite{11,243.85,249.99}, the present approach is solely
based on such natural mapping properties. As a consequence, the basic error analysis is independent of any data-consistency assumptions or of
the regularity of solutions in the case of consistent data, which
in general never occur in practice.

In summary, the guiding ``general hope'' is that, just exploiting natural mapping properties rather than assuming any excess regularity, should ``help'' minimizing the necessary amount of regularization in an inverse problem.
This in turn, is intimately related to the central motivation of this paper,
namely the development  of {\em efficient}
and {\em certifiable} numerical methods that should not rely on
unverifiable assumptions.
%
In a nutshell, for the particular problem type at hand,
significant consequences of a stable variational formulation
of the forward problem are:
 the proposed numerical solvers exhibit
 a favorable quantifiable performance to be commented on further below;  regardless of data consistency and without imposing any regularity assumptions we derive sharp a priori  error bounds that do not degrade when the
regularization parameter tends to zero;
there is no need for tuning parameters  {inside} any mesh-dependent stabilization terms; we can derive computable {\em a posteriori error bounds} that
are valid without any excess regularity assumptions, for arbitrary (inconsistent) data, and, in the present particular inverse problem, are independent of the regularization parameter. 

However, it should be emphasized that our
  ``general hope'' could so far be realized only
for the current rather mildly ill-posed problem class.
 The following remarks \wnew{elaborate a bit more on} some of the related
 aspects.}

\wnew{
(i) Respecting natural mapping properties reveals, in particular, that
 a unique minimizer of
the objective functional exists for {\em any arbitrarily small regularization parameter and even for a vanishing regularization parameter}. In fact, a least squares
formulation by itself turns out to be already a sufficient regularization. However, the {\em condition number} of corresponding
discrete systems may increase   with decreasing  regularization parameter. Our numerical experiments will shed some light on
this interdependence. We  use this insight to develop \wnew{efficient} preconditioners for the discrete problems.
\wnew{In fact, w}ithin the limitations of the infinite-dimensional formulation the solvers will be seen to exhibit  {a quasi-optimal performance for any fixed regularization parameter. Even for the mildly-ill posed problem under
consideration this seems to be unprecedented in the literature.
 In that sense, the primary role of a non-vanishing regularization parameter for us is to facilitate a rigorous performance
analysis of the iterative solver in favor of its quantitative  improvement.}

(ii) A stable infinite dimensional variational formulation is also an essential prerequisite for 
deriving rigorous {\em a posteriori} \wnew{regularity-free} - meaning they
are valid without any excess regularity assumptions -
error bounds for the recovered states. \wnew{As shown later,} such bounds can be used, in particular,
to identify suitable stopping criteria for
iterative solvers.
Finally, we demonstrate some practical consequences of regularity-free computable a posteriori bounds in Section \ref{sec:numexp}. We indicate
their use for estimating data consistency errors as well as for choosing the regularization parameter in a way that accuracy of the results is not compromized in any essential way while enhancing solver performance.


(iii) 
Respecting natural
mapping properties, allows us to ``disentangle'' discretization and
regularization by studying first the intrinsic necessary ``strength'' of the regularization in the infinite-dimensional setting.
Moreover, it turns out that additional regularization beyond the least squares formulation is not necessary on the infinite-dimensional level, persists to remain true for the proposed
inf-sup stable discretizations.  Choosing nevertheless a positive
regularization parameter in favor of a better and rigorously founded solver
performance, still requires a balanced choice so as to warrant optimal
achievable accuracy of the state estimate. Our formulation reveals that
the relevant balance criterion is then the achievable approximation accuracy of
the trial space. Only sufficiently high regularity, typically hard to check
in practice,  allows one to express this quantity in terms of a uniform
mesh-size. Our approach
 will be seen to offer
\wnew{more flexibility and potentially} different choices of
regularization parameters than those stemming from the a priori fixed mesh-dependent approach  in \cite{35.925,35.926} or \cite{MR3079321}.
This concerns, for instance,   adaptively refined meshes or higher order discretizations.

A perhaps more subtle further consequence of exploiting natural
mapping properties are
somewhat stronger {\em a priori estimates} than those obtained  in previous works.

Of course, the robustness of our results with respect to the regularization parameter reflects the mild degree of ill-posedness of the data-assimilation
problem under consideration. This cannot be expected to carry over to less
stable problems in exactly the same fashion. We claim though that important elements will persist to hold, for instance, for conditionally stable problems.
In particular, non-vanishing regularization parameters will then be essential
and regularity-free a posteriori bounds will be  {all} the more important for
arriving at properly balanced choices in the spirit of the strategy indicated
in Section \ref{sec:numexp}.
A detailed discussion is beyond the scope of this paper and is therefore deferred to forthcoming work.
}


\subsection{Layout}
\label{ssec:1.3}
In Section \ref{S2} we present a stable weak time-space formulation of a parabolic model problem and introduce
the data assimilation problem considered in this work.
Based on these findings we  propose in Section \ref{Sregularized} a regularized {\em least squares formulation} of the
state estimation task. This formulation permits model as well as data errors
as the recovered states are neither required to
satisfy the parabolic equation exactly nor to match the data. We then derive {\em a priori} as well as {\em a posteriori} error estimates for the infinite-dimensional minimizer
as well as for the minimizer over a finite dimensional trial space revealing the basic interplay between model inconsistencies, data errors,
and regularization strength.

Since the ``ideal'' infinite-dimensional objective functional involves a dual-norm a practical numerical method needs to
handle this term.  We show that a proper discretization of the dual norm is tantamount to identifying a stable
{\em Fortin operator}. For the given formulation of the parabolic problem this turns out to impose theoretical limitations on
discretizations based on a standard second order variational formulation. Therefore we consider in Section \ref{SFOSLS} an equivalent first order system least squares formulation.
Section \ref{sec:Fortin} is devoted to the construction of Fortin operators for both settings. Moreover, we present in Sections \ref{Sprac1} and
\ref{Sprac2}
effective preconditioners for the iterative solution of the arising discrete problems along with suitable
stopping criteria.
Section \ref{sec:numexp} is devoted to numerical experiments that quantify the theoretical findings and illustrate the
performance of the numerical schemes, \wnew{in particular, depending on the
choice of the regularization parameter which, in principal, could be chose
as zero}.
 We conclude in Section \ref{sec:outlook} with a brief discussion of several ramifications of the results, \wnew{including the application of a posteriori bounds
 for estimating data-consistency errors}.

\subsection{Notations}\label{ssec:1.2}
In this work, by $C \lesssim D$ we will mean that $C$ can be bounded by a multiple of $D$, independently of parameters which C and D may depend on.
Exceptions are given by the parameter\new{s} $\eta$ \new{and $\omega$} in the Carleman estimate \eqref{carleman}, the polynomial degrees of various finite element spaces, and the dimension $d$ of the spatial domain $\Omega$.
Obviously, $C \gtrsim D$ is defined as $D \lesssim C$, and $C\eqsim D$ as $C\lesssim D$ and $C \gtrsim D$.

For normed linear spaces $E$ and $F$, by $\cL(E,F)$ we will denote the normed linear space of bounded linear mappings $E \rightarrow F$,
and by $\Lis(E,F)$ its subset of boundedly invertible linear mappings $E \rightarrow F$.
We write $E \hookrightarrow F$ to denote that $E$ is continuously embedded into $F$.
For simplicity only, we exclusively consider linear spaces over the scalar field $\R$.



\section{Problem Formulation and Preview}  \label{S2}
For a given  domain $\Omega \subset \R^d$ and time-horizon $T>0$, let
$ I:=[0,T]$. Let  $a(t;\cdot,\cdot)$ denote a bilinear form on $H^1_0(\Omega) \times H^1_0(\Omega)$ such that for
any $\theta,\zeta \in H^1_0(\Omega)$, the function $t \mapsto a(t;\theta,\zeta)$ is measurable on $I$. Moreover, we assume
  that for almost all $t \in I$, $a(t;\cdot,\cdot)\colon H^1_0(\Omega) \times H^1_0(\Omega)\to \R$ is {\em bounded} and {\em coercive}, i.e.
\begin{alignat}{3} \label{1}
|a(t;\theta,\zeta)|
& \lesssim  \|\theta\|_{H^1(\Omega)} \|\zeta\|_{H^1(\Omega)} \quad &&(\theta,\zeta \in H^1_0(\Omega)), 
\\ \label{2}
 a(t;\theta,\theta)  &\gtrsim \|\theta\|_{H^1(\Omega)}^2 \quad
&&(\theta \in {H^1_0(\Omega)}), 
\end{alignat}
hold with constants independent of $t\in I$. By Lax Milgram's Theorem, $A(t)$, defined by $(A(t)\theta)(\zeta):=a(t;\theta,\zeta)$,
($\theta,\,\zeta\in H^1_0(\Omega)$), belongs to $\Lis(H^1_0(\Omega),H^{-1}(\Omega))$.

Before discussing the parabolic data assimilation problem, we recall some facts about a time-space variational formulation of the \emph{parabolic initial value problem} -- the corresponding forward problem --
of the form
\be \label{initial}
\left\{
\begin{array}{rl}
\frac{d z}{d t}(t) +A(t) z(t)&\!\!\!= h(t) \quad(t \in I\text{ a.e.}),\\
\gamma_0 z &\!\!\!= z_0,
\end{array}
\right.
\ee
\new{with trace map $\gamma_t\colon z \mapsto z(t)$.}
With the spaces
$$
X:=L_2(I;H^1_0(\Omega)) \cap H^1(I;H^{-1}(\Omega)),\quad Y:=L_2(I;H^1_0(\Omega)),
$$
the operator $B$ defined by
$$
(Bw)(v):=\int_I {\textstyle \frac{d w}{dt}}(t)(v(t))  +
a(t;w(t),v(t)) dt,
$$
belongs to $\cL(X,Y')$.
Recall also that
\be \label{embedding}
X \hookrightarrow C(I;L_2(\Omega))
\ee
with the latter space being equipped with the norm on $L_\infty(I;L_2(\Omega)).$
In particular, this implies that \new{$\gamma_t \in \cL(X,L_2(\Omega))$}, with a norm that is uniformly bounded in $t \in I$.
The resulting weak formulation of \eqref{initial} reads as
$$
Bz = h,\quad \gamma_0 z= z_0,
$$
and it is known (e.g.~see \cite[Ch.XVIII, \S3]{63}, \cite[Ch.~IV, \S26]{314.9}, \cite{247.15})
to be well-posed in the sense that
\be \label{boundedinvertible}
 \left[\begin{array}{@{}c@{}} B \\ \gamma_0\end{array} \right]\in \Lis(X,Y' \times L_2(\Omega)).
\ee

Turning to the data assimilation problem,
suppose in what follows that $\omega \subset \Omega$ is a fixed non-empty sub-domain (possibly much smaller than $\Omega$)
and that we are given data $f \in L_2(I\times \omega)$ as well as $g\in Y'$.
The {\em data-assimilation} problem considered in this paper
 is to seek a state $u\in X$, that approximately satisfies
 $B u= g$, while also closely agreeing with $f$
in $L_2 ( I\times \omega)$, see \cite{35.925,35.926}. To make this precise, ideally one would like to solve
 \begin{equation} \label{11}
\left\{
\begin{array}{rl}
\frac{d u}{d t}(t) +A(t) u(t)&\!\!\!= g(t) \quad(t \in I\text{ a.e.}),\\
u(t)|_{\omega}&\!\!\!= f(t)\quad(t \in I).
\end{array}
\right.
\end{equation}
However, in general such data ${(g,f)}$ may be {\em inconsistent}, i.e., \eqref{11} has no solution and is therefore
 ill-posed. To put this formally, denoting by  $\Gamma_\omega$  the restriction of a function on $I \times \Omega$ to a function on $I \times \omega$, we have $\Gamma_\omega \in \cL(X,L_2(I\times \omega))$,  i.e.,
 $\Gamma_\omega$ is bounded on $X$. However, the range of the operator
$$
B_\omega :=  \left[\begin{array}{@{}c@{}} B \\ \Gamma_\omega\end{array} \right]\in \cL(X,Y' \times L_2(I\times \omega))
$$
induced by \eqref{11},
is a strict subset of $Y' \times L_2(I\times \omega)$.

Before addressing this issue, it is instructive to understand the case of a {\em consistent} pair $(g,f)\in \ran B_\omega$, i.e.,  {when} there exists a $u\in X$ such that $(g,f)=(Bu,\Gamma_\omega u)$. 
\begin{remark}
\label{rem:unique}
Any data consistent pair $(g,f)\in \ran B_\omega$ determines a unique state $u\in X$ satisfying \eqref{11}.
\end{remark}
 That this is indeed the case
 can be derived from the following crucial tool that has been
employed in prior related studies such as  \cite{35.925,35.926} and will be heavily used
in what follows as well.
\new{For $\eta \in (0,T)$ let
 $$
 X_\eta:=L_2([\eta,T];H^1_0(\Omega)) \cap H^1([\eta,T];H^{-1}(\Omega)).
 $$
\emph{Fixing} both $\eta$ and a subdomain $\omega \subset \Omega$},
 a version of the so-called \emph{Carleman Estimate} says in the present terms
\be \label{carleman}
\|w\|_{X_\eta} \lesssim \| {\Gamma_\omega} w\|_{L_2(I \times \omega)}+\|Bw\|_{Y'} \quad(w \in X).
\ee
\new{
\begin{remark}
\label{rem:carlass}
The validity of \eqref{carleman}
has been established in \cite[Thm.~2]{35.925} for the \emph{heat operator} (i.e., $a(t;\theta,\zeta)=\int_\Omega \nabla \theta(t) \cdot  {\nabla} \zeta(t)\,d{\bf x}$) and $\Omega \subset \R^d$ being a \emph{convex polytope}. It holds in greater generality though. For instance, the argument
in the proof of \cite[Lemma 7]{35.925} still works when $\Omega$ is star-shaped w.r.t.~an $x_0 \in \Omega$
and any open $\omega \subset \Omega$ that contains $x_0$.
In what follows up to this point we will tacitly assume at this point
suitable problem specifications that guarantee the validity of \eqref{carleman} without further mentioning.
\end{remark}
}

Returning to the uniqueness of $u$ given consistent  data ${(g,f)}$, suppose there exist two solutions, then their difference $e \in X$ satisfies
$\|{\Gamma_\omega} e\|_{L_2(I \times \omega)}+\|Be\|_{Y'}=0$, meaning in view of \eqref{carleman}  that $\|e\|_{X_\eta}=0$,
and so thanks to $X_\eta \hookrightarrow C([\eta,T];L_2(\Omega))$, that $e(t)=0$ for $t \in [\eta,T]$. From $X \hookrightarrow C([0,T];L_2(\Omega))$, and the fact that $\eta>0$ is arbitrary, it follows that $e=0$.

However, the nature of the Carleman Estimate indicates that one cannot stably recover the trace $\gamma_0 u$ which would then together with $g$ stably recover $u$. In fact, one may convince oneself that significantly different initial data (far) outside $\omega$ may give
rise to homogeneous solutions of \eqref{initial} that hardly differ
on $I\times \omega$. Thus, even for a state $u\in X$ from  (nearly) consistent data ${(g,f) \in Y' \times L_2(I\times \omega)}$
we cannot expect to find an accurate numerical approximation to $u$ on the whole  {time-space} cylinder $I \times \Omega$. Moreover, any perturbation of the data may land outside $\ran B_\omega$.

In practice, neither will the  data/measurements $(g,f)\in Y'\times L_2(I\times \omega)$ be exact, nor will the observed state behind $f$
satisfy the model -- here a parabolic PDE --  exactly. Thus, in general a pair of data $(g,f)\in Y'\times L_2(I\times \omega)$
allows one to recover any hypothetical source $u\in X$ only within some {\em uncertainty}. A central theme in this article
is to quantify this uncertainty (theoretically and numerically)
by properly exploiting the information provided by the PDE model, and the data. 
While any such assimilation attempt rests on the basic hypothesis  that
the  data ${(g,f)}$ are ``close'' in ${Y' \times L_2(I\times \omega)}$
 to a consistent pair $(Bu, u|_{I\times\omega})\in \ran \,B_\omega$, for some $u\in X$, this ``closeness'' is generally not known beforehand.

 To perform such a recovery we formulate in the next section a family of {\em regularizations} of the ill-posed problem
 \eqref{11} involving a parameter $\eps \ge 0$, taking data errors and model bias into account.  We then show, first on the continuous
 {\em  infinite-dimensional} level, that for each $\eps \ge 0$
 there exists a unique regularized solution $u_\eps\in X$. Letting this precede an actual {\em discrete} scheme,
   will be important for a number of issues, such as
 the design of efficient iterative solvers,
 the derivation of {\em a posteriori} error bounds, as well as disentangling regularization and discretization
 in favor of an overall good balance of uncertainties. Aside from the question what a preferable choice of $\eps$ would be in that latter respect, a central issue will be to assess the quality of a regularized
 state $u_\eps$ and of its approximation $u^{\delta,\delta}_\eps$ from a given finite-dimensional trial space $X^\delta\subset X$ provided
 by our numerical scheme.

 To that end, recall that generally (for inconsistent data) the idealized assimilation problem \eqref{11} has no  solution.
 So whatever state $u\in X$
 may be used to ``explain'' the data, should be viewed as a {\em candidate} or {\em reference state} that is connected with the recovery task
 through the {\em consistency error}
 \be
\label{econs}
e_\cons(u) := {\sqrt{\|B u-g\|^2_{Y'}+\|\Gamma_\omega u-f\|^2_{L_2(I \times \omega)}}}.
\ee

At the heart of our analysis is then an a priori estimate of the type
 \be
 \label{type}
 \|u-u_\eps^{\delta,\delta}\|_{X_\eta} \lesssim  {e_\cons(u)}+ {e^\delta_\ap(u)} +\eps\|\gamma_0 u\|_{L_2(\Omega)},
 \ee
where $e^\delta_\ap(u)$ denotes the error of the best approximation to $u$ from $X^\delta$ in $X$, thereby {implicitly quantifying} the regularity of the state $u$. Recall that, as always, the constant in this estimate absorbed by the $\lesssim$-symbol may depend on $\eta>0$, but neither on $u$ {nor on $\eps$}.

It is important to note that \eqref{type} is valid for {\em any} $u\in X$, not making use of the assumption that $(\new{B} u,\Gamma_\omega {u})$
be close to $(g,f)$. It is of evident value, of course, for states $u$ with small or at least moderate consistency error. This suggests singling out
a particular state
$$
 {u_0 := \argmin_{u\in X} e_\cons(u)}
$$
that minimizes the consistency error.
As in the case of consistent data, we will see that $u_0$ is unique, and, and as is suggested by its notation, it will turn out to be the limit for $\eps\downarrow 0$ of the regularized solutions $u_\eps$ that will be defined later.
One reason for not confining the error estimates -- or perhaps better termed {\em distance estimates} -- to the specific state ${u_0}$
is the potential significant model bias. In fact, we view it as a strength to keep \eqref{type} general
since this covers automatically various somewhat specialized scenarios.
For instance, if the data were exact,
 i.e., ${e_\cons(u_0) =0}$, \eqref{type},  {for $u=u_0$,} shows the dependence
of the error just on $\eps$ and the choice of the discretization. Moreover, if the model is exact (or the model bias is negligible
compared with data accuracy) it will later be seen how to get a ``nearly-computable'' bound for ${e_\cons(u_0)}$ and hence
an idea of the model bias (due to $g$) and measurement errors in $f$.
Another case of interest is $u=u_\eps$ because this is the ``compromise-solution''
suggested by the chosen regularization and targeted by
the numerical scheme.

 Finally, while in principle, $\eps$ can be chosen as small as we wish (even zero), it will be seen to benefit solving
the discrete problems by choosing $\eps$ as large as possible so as to remain just dominated, ideally, by
$e_\cons(u_0)$, in practice, by the announced a posteriori bounds.


\section{Regularized least squares} \label{Sregularized}
\newcommand{\cD}{\mathcal{D}}
\newcommand{\cC}{\mathcal{C}}
Knowing that the data \new{assimilation} problem is ill-posed and taking the preceding considerations into account,  we consider for
some parameter $\eps\geq 0$ the regularized
{\em least squares problem} of finding the minimizer $u_\eps$ over $X$ of
\be
\label{Geps}
G_\eps\colon w \mapsto \|Bw-g\|_{Y'}^2+\|\Gamma_\omega w-f\|_{L_2(I \times \omega)}^2+\eps^2\|\gamma_0 w\|_{L_2(\Omega)}^2,\footnotemark
\ee
\footnotetext{We could have included additional weights in front of the first terms that could reflect a priori knowledge on model- or data-fidelity. Since this would not affect the subsequent developments we disregard this option for simplicity of exposition.}%
where, as before,  $\Gamma_\omega$ is the restriction of a function on $I \times \Omega$ to a function on $I \times \omega$. The resulting Euler-Lagrange equations read as
\be \label{EL}
\langle B u_\eps \!-\!g,B w\rangle_{Y'}\!+\!
\langle \Gamma_\omega u_\eps\! -\!f,\Gamma_\omega w\rangle_{L_2(I \times \omega)}\!+\!\eps^2 \langle \gamma_0 u_\eps,\gamma_0 w \rangle_{L_2(\Omega)}=0 \quad(w \in X).
\ee
Since $\Gamma_\omega \in \cL(X,L_2(I \times \omega))$, and on account of \eqref{boundedinvertible}, for $w \in X$
it holds that
\be \label{equivalence}
\eps^2 \|w\|_X^2 \lesssim \|Bw\|_{Y'}^2+\|\Gamma_\omega w\|_{L_2(I \times \omega)}^2+\eps^2\|\gamma_0 w\|_{L_2(\Omega)}^2
 \lesssim \max(1,\eps^2) \|w\|_X^2.
 \ee
By the Lax-Milgram Lemma, we thus know that for $\eps>0$  the minimizer  $u_\eps$ exists uniquely, and satisfies
\be \label{basic_stab}
\|u_\eps\|_X \lesssim \max(\eps^{-1},1) \big(\|g\|_{Y'}+\|f\|_{L_2(I \times \omega)}\big).
\ee

Selecting any reference state $u \in X$,  similarly to \eqref{basic_stab} one can show that for $\eps>0$
$$
\|u-u_\eps\|_X \lesssim  \max(\eps^{-1},1)\big(\eps \|\gamma_0 u\|_{L_2(\Omega)}+  {e_\cons(u)}\big),
$$
 {see \eqref{econs}.}
This result is by no means satisfactory.
With the aid of \eqref{carleman}, 
 much better bounds will be established for $\|u-u_\eps\|_{X_\eta}$.

\begin{remark}
\label{rem:G0}
 Also for $\eps=0$, the minimizer $u_0$ of $G_0(\cdot)= {e_\cons(\cdot)^2}$ over $X$ exists uniquely.
Indeed, suppose there are two minimizers. Then,
by \eqref{EL},  their difference $e_0$ satisfies
$$
\langle B e_0,B w\rangle_{Y'}+
\langle \Gamma_\omega e_0,\Gamma_\omega w\rangle_{L_2(I \times \omega)}=0 \quad(w \in X),
$$
and so $\|\Gamma_\omega e_0\|^2_{L_2(I \times \omega)}+\|Be_0\|^2_{Y'}=0$. As we have seen, using \eqref{embedding} and \eqref{carleman} this implies $e_0=0$.
 \end{remark}

A frequently used tool reads as follows.
\begin{lemma}
\label{lem_basic}  For any $w \in X$
one has
$$
\|u-w\|_{X_\eta} \lesssim \sqrt{G_{{0}}(w)}+{e_\cons(u)}. 
$$
\end{lemma}

\begin{proof}
We infer from \eqref{carleman}  and a triangle-inequality for the norm $\sqrt{\|\,\|_{L_2(I \times\omega)}^2+\|\,\|_{Y'}^2}$ that
$$
\|u-w\|_{X_\eta} \lesssim \sqrt{\|\Gamma_\omega(u-w)\|_{L_2(I\times \omega)}^2+\|B(u-w)\|_{Y'}^2} \leq \sqrt{G_0(w)}+e_\cons(u)
$$
which confirms the claim. \hfill$ \qedhere$
\end{proof}

When taking as reference state $u=u_\eps$, we obtain the following \emph{a posteriori bound}.
\begin{proposition}
\label{prop1}  For $\eps \geq 0$ and $w \in X$,
one has
$$
\|u_\eps-w\|_{X_\eta} \lesssim \sqrt{G_{\eps}(w)}.
$$
\end{proposition}

\begin{proof} The proof follows from Lemma~\ref{lem_basic} and
$$
\sqrt{G_{0}(w)}+e_\cons(\new{u_\eps})\leq \sqrt{G_{\eps}(w)}+\sqrt{G_{\eps}(u_\eps)} \leq 2\sqrt{G_{\eps}(w)}.\qedhere
$$
\end{proof}

The same arguments, used to show for $\eps\geq 0$ existence and uniqueness of the minimizer $u_\eps$ of $G_\eps$ over $X$, show for any closed subspace $X^\delta\subset X$
uniqueness of the minimizer $u^\delta_\eps$ of $G_\eps$ over $X^\delta$.
 An \emph{a priori bound} for  $\|u-u_\eps^\delta\|_{X_\eta}$ for an arbitrary reference state $u \in X$ is given in the next proposition.

\begin{proposition}
\label{prop2} It holds that
\begin{align*}
\|u-u_\eps^\delta\|_{X_\eta} &\lesssim  {e_\cons(u)}
+ {e_{\ap}^\delta(u)} 
+ \eps\|\gamma_0 u\|_{L_2(\Omega)},
\end{align*}
where
$$
 e^\delta_\ap(u)   :=\min_{w \in X^\delta} \|u-w\|_X
$$
 denotes the corresponding \emph{approximation error} of the state $u$.
\end{proposition}

\begin{proof} 
Let $P_{X^\delta}$ denote the $X$-orthogonal projector onto $X^\delta$, then using
$B \in \cL(X,Y')$ and \eqref{embedding}, we infer that
\begin{align*}
 {\sqrt{G_0(u^\delta_\eps)} \leq} &\sqrt{G_\eps(u^\delta_\eps)}  \leq  \sqrt{G_\eps(P_{X^\delta} u)} \\
  {\leq} & \|B(u-  P_{X^\delta} u)\|_{Y'} +
\|Bu-g\|_{Y'}+\|f- \Gamma_\omega u\|_{L_2(I\times \omega)}+
\\
& \|\Gamma_\omega(u- P_{X^\delta} u)\|_{L_2(I\times \omega)} + \eps\|\gamma_0 (u- P_{X^\delta} u)\|_{L_2(\Omega)} + \eps\|\gamma_0 u\|_{L_2(\Omega)}\\
 \lesssim & {e_{\ap}^\delta(u)} + {e_\cons(u)}   + \eps \|\gamma_0 u\|_{L_2(\Omega)},
\end{align*}
which together with Lemma~\ref{lem_basic} completes the proof.
\end{proof}

 At this point we note that because of the presence of the dual norm $\|\cdot\|_{Y'}$ in $G_\eps$, neither $u_\eps^\delta$ nor the a posteriori bound for $\|u_\eps-w\|_{X_\eta}$ from Proposition~\ref{prop1} for e.g.~$w=u_\eps^\delta$ can be computed. Both problems are going to be tackled in the next two subsections.



\begin{remark}
\label{rem:stab}  {Although the upper bound from Proposition~\ref{prop2} is minimal for $\eps=0$,} a reason for nevertheless taking $\eps>0$, say of the order of the expected magnitude of  $ {e_\cons(u) +e_{\ap}^\delta(u)}$,
is to enhance the numerical stability of solving the Euler-Lagrange equations.
\end{remark}

\new{\begin{remark} Notice that even when $e_\cons(u)=0$ and $\eps=0$, Proposition~\ref{prop2} does not show that $u_0^\delta$ is a quasi-best approximation to $u$ from $X^\delta$. Indeed the norm $\|\cdot\|_X$ used to define $e^\delta_\ap(u)$ differs from the norm $\|\cdot\|_{X_\eta}$ in which $u-u_0^\delta$ is measured.
\end{remark}}

We conclude this section with a few comments on the behavior of $u_\eps$ when $\eps$ tends to zero.  First, note that the consistency error of $u_\eps$ approaches the minimal consistency error  {$e_\cons(u_0)$}
when $\eps \to 0$ because
$$
e_\cons(u_\eps) \le \sqrt{G_\eps(u_\eps)} \le \sqrt{G_\eps({u_0})} \, {\leq} \, {e_\cons(u_0)} + \eps \|\gamma_0  {u_0}\|_{L_2(\Omega)}.
$$
In particular, a first trivial consequence of Proposition~\ref{prop2} is that, for consistent and exact data, i.e., $ {e_\cons(u_0)} =0$,  $u_\eps$ tends to the state $ {u_0}$ \new{in $X_\eta$ for any $\eta>0$. Even without the assumption $e_\cons(u_0)=0$, a stronger result is derived in the following remark.}

\begin{remark}
\label{rem:Gamma}
One has
$$
\| {u_0} - u_\eps\|_{X}\to 0,\quad \eps \to 0.
$$
\end{remark}
\begin{proof}
We remark first that $\sqrt{{G_\eps}}$ $\,\Gamma$-converges to $\sqrt{{G_0}} =:F$. In fact,
let $(\eps_n)_{n\in \N}$ tend to zero. The functionals $F_n:= \sqrt{{G_{\eps_n}}}: X\to \R_+$ are uniformly coercive (in the sense of optimization,
meaning that $F_n(w)\to \infty$ for $\|w\|_X\to \infty$). Let $(w_n)_{n\in\N}$ be any sequence in $X$ with limit $w\in X$. Then
\begin{align*}
F(w)-F_{{n}}(w_n) &\le F(w)-F(w_n) {\leq \sqrt{\|B(w-w_n)\|_{Y'}^2+\|\gamma_\omega(w-w_n)\|_{L_2(I \times\omega)}^2}}\\
& \lesssim \|w-w_n\|_X,\quad n\in \N,
\end{align*}
so that $F(w) \le \liminf_{n\to\infty}F_n(w_n)$. 
Moreover, for any $w\in X$ there exists a sequence
$(w_n)_{n\in\N}$ in $X$ such that $F(w) \ge \limsup_{n\to\infty} F_n(w_n)$, as can be seen by simply taking $w_n=w$. Thus, by the main Theorem of $\Gamma$-convergence,
minimizers of $F_n$ converge to the minimizer of $F$.
\end{proof}

Thus, trying to solve the  regularized problem, with $\eps$ as small as possible  incidentally favors $ {u_0}$ as a target state.
Thus, it is of interest to estimate $e_\cons(u_0)$ (see  {Corollary~\ref{corol}} later below)
since a relatively large
$e_\cons(u_0)$ weakens the relevance of $u_0$, favoring correspondingly larger regularization parameters.


\subsection{Discretizing the dual norm}\label{ssec:dualnorm}

Minimizing $G_\eps$ over $X^\delta$ does not correspond to a practical method because the dual norm $\|\cdot\|_{Y'}$ cannot be evaluated.
Therefore, given a family of finite dimensional subspaces $(X^\delta)_{\delta \in \Delta}$ of $X$, the idea is to find a family
$(Y^\delta)_{\delta \in \Delta}$  of finite dimensional subspaces of $Y$, ideally with $\dim Y^\delta \lesssim \dim X^\delta$,
such that $\|B w\|_{Y'}$ can be controlled for $w\in X^\delta$ by  the computable quantity $\|B w\|_{{Y^\delta}'}$.
This is ensured whenever
\be \label{inf-sup}
\inf_{\delta \in \Delta} \inf_{ {\{w \in X^\delta\colon B w \neq 0\}}} \ \sup_{0 \neq \mu \in Y^\delta} \frac{(B w)(\mu)}{\|B w\|_{Y'} \|\mu\|_Y}>0
\ee
is valid.

In the subsequent discussion we make heavy use of
the Riesz isometry $R\in \Lis(Y,Y')$,  defined by
$$
(Rv)(w):=\langle v,w\rangle_Y:=\int_I \int_\Omega \nabla_{\bf x} v \cdot \nabla_{\bf x} \new{w} \,d{\bf x}\, dt, \quad(v,w \in Y). 
$$
Introducing  auxiliary variables for $\mu_\eps =R^{-1}(g-B u_\eps )\in Y$,
 $\theta_\eps =   f- \Gamma_\omega u_\eps \in L_2(I\times \omega)$, and
$\nu_\eps =  {-}\gamma_0 u_\eps \in L_2(\Omega)$ gives rise to a mixed formulation {of the problem of finding the minimizer $u_\eps$ over $X$ of $G_\eps$ defined in} \eqref{Geps}  in terms of
 the saddle point system
\be \label{saddle}
S_\eps (\mu_\eps, \theta_\eps,\nu_\eps,u_\eps) :=
\left[
\begin{array}{@{}cccc@{}} R & 0 & 0 & B\\
0 & I & 0 & \Gamma_\omega\\
0 & 0 & I & \eps \gamma_0\\
B' & \Gamma_\omega' & \eps \gamma_0' & 0
\end{array}
\right]
\left[
\begin{array}{@{}c@{}}
\mu_\eps \\ \theta_\eps \\ \nu_\eps \\ u_\eps
\end{array}
\right]
=
\left[
\begin{array}{@{}c@{}}
g \\ f \\ 0 \\0
\end{array}
\right].
\ee
(see  \cite[Sect.~2.2]{45.44}).
(Equivalently,  \eqref{saddle} characterizes the critical point
of the   Lagrangian obtained when inserting in $G_\eps$ these variables and appending corresponding constraints by
Lagrange multipliers.)

%

\begin{remark}
\label{rem:condensed}
Eliminating the second and third variable from \eqref{saddle}, one arrives at the equivalent more compact  formulation
$$
\left[
\begin{array}{@{}cc@{}} R & B \\ B' & -(\Gamma_\omega' \Gamma_\omega+\eps^2\gamma_0'\gamma_0)
\end{array}
\right]
\left[
\begin{array}{@{}c@{}}  {\mu_\eps} \\  {u_\eps}
\end{array}
\right]
=
\left[
\begin{array}{@{}c@{}} g\\ -\Gamma_\omega' f,
\end{array}
\right]
$$
It serves in Section \ref{Sprac1} as the starting point for a numerical scheme.
\end{remark}

\begin{theorem} \label{thm1} Let \eqref{inf-sup} be valid. For $u_\eps^{\delta,\delta}$ denoting the (unique) minimizer over $X^\delta$ of
$$
G^\delta_\eps := w \mapsto \|Bw-g\|_{{Y^\delta}'}^2+ \|\Gamma_\omega w-f\|_{L_2(I \times \omega)}^2+\eps^2\|\gamma_0 w\|_{L_2(\Omega)}^2,
$$
one has
$$
\|u-u^{\delta,\delta}_\eps\|_{X_\eta} \lesssim  {e_\cons(u) + e^\delta_\ap(u)} +\eps\|\gamma_0 u\|_{L_2(\Omega)}.
$$
\new{(We recall that, as always, the constant absorbed by the $\lesssim$-symbol may (actually will) depend on $\omega$ and $\eta$, but \emph{not} on $\eps \geq 0$ or $\delta \in \Delta$.)}
\end{theorem}

\begin{proof}
 {Denoting the block-diagonal operator comprized of the leading $3\times 3$ block in $S_\eps$ by $D$, the operator $S_\eps$ can be rewritten as
$$
S_\eps = \left[
\begin{array}{@{}cc@{}}
D & C_\eps\\
C_\eps' & 0
\end{array}
\right],
$$
where $C_\eps \in \cL\big(X,Y' \times L_2(I\times \omega)\times L_2(\Omega)\big)$ is defined by
$$
(C_\eps w)(\mu,\theta,\nu) := (B w)(\mu) + \langle \Gamma_\omega w,\theta\rangle_{L_2(I\times \omega)} + \eps \langle
\gamma_0 w, \nu\rangle_{L_2(\Omega)}.
$$
With the usual identification of $L_2(I\times \omega)$ and $L_2(\Omega)$ with their duals, $D$ is just the isometric Riesz isomorphism between $Y \times L_2(I\times \omega)\times L_2(\Omega)$ and its dual.
Equipping $X$ with the  ($\eps$-dependent)  ``energy''-norm
$$
\nrm w \nrm_{\eps}:=\sqrt{\|B w\|_{Y'}^2 +\|\Gamma_\omega w\|_{L_2(I \times \omega)}^2+\eps^2\|\gamma_0 w\|_{L_2(\Omega)}^2},
$$
one verifies that $\|\cC_\eps w\|_{Y' \times L_2(I\times \omega)\times L_2(\Omega)} = \nrm w \nrm_{\eps}$, so that in particular $C_\eps$ satisfies an `inf-sup' condition.}
Consequently, the operator $S_\eps$ on the left hand side of \eqref{saddle} is a boundedly invertible mapping from $Y \times L_2(I \times \omega) \times L_2(\Omega) \times (X,\nrm\,\nrm_\eps)$ to its dual (uniformly in $\eps$).

Analogously  {to the continuous case}, the minimizer $u^{\delta,\delta}_\eps$ of $G^\delta_\eps$ equals the fourth component of the solution $(\mu^{\delta,\delta}_\eps,\theta^{\delta,\delta}_\eps,\nu^{\delta,\delta}_\eps,u^{\delta,\delta}_\eps)$ of the Galerkin discretization of \eqref{saddle} with trial space $Y^\delta \times L_2(I \times \omega) \times L_2(\Omega) \times X^{\delta}$.
Thanks to \eqref{inf-sup}, for $w \in X^\delta$ we have
$$
\sup_{0 \neq (\tilde{\mu},\tilde{\theta},\tilde{\nu}) \in Y^\delta \times L_2(I \times \omega) \times L_2(\Omega)} \frac{(B w)(\tilde{\mu})+\langle \Gamma_\omega w,\tilde{\theta}\rangle_{L_2(I \times \omega)}+\eps\langle \gamma_0 w,\tilde \nu\rangle_{L_2(\Omega)}}{\sqrt{\|\tilde \mu\|_{Y}^2 +\|\tilde \theta\|_{L_2(I \times \omega)}^2+\|\tilde \nu\|_{L_2(\Omega)}^2}} \eqsim\nrm w\nrm_\eps,
$$
 {so that the so-called Ladyzhenskaya-Babu\v{s}ka-Brezzi condition is satisfied.
Consequently, the discretization of the saddle-point system is uniformly stable, and so we have}
\be \label{40}
\begin{split}
 \|\mu_\eps-&\mu^{\delta,\delta}_\eps\|_Y+\|\theta_\eps-\theta^{\delta,\delta}_\eps\|_{L_2(I \times \omega) }+\|\nu_\eps-\nu^{\delta,\delta}_\eps\|_{L_2(\Omega)} +\nrm u_\eps-u^{\delta,\delta}_\eps\nrm_\eps \\
&\lesssim \min_{(\tilde{\mu},\tilde{u}) \in Y^\delta \times X^\delta} \|\mu_\eps-\tilde \mu\|_Y+\nrm u_\eps-\tilde{u} \nrm_\eps
\\
& \leq  \|\mu_\eps\|_Y+\min_{\tilde{u} \in  X^\delta} \nrm u_\eps-\tilde{u} \nrm_\eps
 =  \|g - B u_\eps\|_{Y'}+\min_{\tilde{u} \in  X^\delta} \nrm u_\eps-\tilde{u} \nrm_\eps.
\end{split}
\ee

From \eqref{carleman}, we have
$$
\|u-u_\eps^{\delta,\delta}\|_{X_\eta}\lesssim
\nrm u-u_\eps^{\delta,\delta}\nrm_{\eps}
\leq
\nrm u-u_\eps\nrm_{\eps}
+
\nrm u_\eps-u_\eps^{\delta,\delta}\nrm_{\eps},
$$
where, by \eqref{40},
\begin{align*}
\nrm u_\eps-u_\eps^{\delta,\delta}\nrm_{\eps}
& \lesssim \|g - B u_\eps\|_{Y'}+\nrm u-u_\eps \nrm_\eps +\min_{\tilde{u} \in  X^\delta} \nrm u-\tilde u \nrm_\eps,\\
& \lesssim \|g - B u_\eps\|_{Y'}+\nrm u-u_\eps \nrm_\eps+ {e^\delta_\ap(u)}  ,
\end{align*}
where we have used $\nrm\,\nrm_{\eps} \lesssim \|\,\|_X$.
By applying a triangle-inequality for the norm $\sqrt{\|\cdot\|_{Y'}^2+\|\cdot\|_{L_2(I \times \omega)}^2+\eps^2 \|\cdot\|_{L_2(\Omega)}^2}$, one infers that
\begin{align*}
\|g - B u_\eps\|_{Y'}+\nrm u-u_\eps \nrm_\eps & \leq \sqrt{G_0(u_\eps)}+\sqrt{G_\eps(u)}+\sqrt{G_\eps(u_\eps)}  \leq 3 \sqrt{G_\eps(u)} \\
&\leq 3 (e_\cons(u) +\eps\|\gamma_0 u\|_{L_2(\Omega)}).
\end{align*}
By combining the estimates from the last three displayed formulas the proof is complete.
\end{proof}

 {\begin{remark} \label{rem:Gamma2}
Let $(X^\delta)_{\delta \in \Delta}=(X^{\delta_n})_{n \in \N}$ be such that $\overline{\cup_n X^{\delta_n}}=X$ and $X^{\delta_n} \subset X^{\delta_{n+1}}$ ($\forall n$). Let $(Y^{\delta_n})_{n \in \N}$ be a corresponding sequence such that \eqref{inf-sup} is valid, $\overline{\cup_n Y^{\delta_n}}=Y$ and $Y^{\delta_n} \subset Y^{\delta_{n+1}}$ ($\forall n$),
and let $(\eps_n)_{n \in \N}$ be such that $\lim_{n \rightarrow \infty} \eps_n=0$.
Then
$$
\lim_{n \rightarrow \infty} G_0(u_{\eps_n}^{\delta_n,\delta_n})=\lim_{n \rightarrow \infty} G_{\eps_n}(u_{\eps_n}^{\delta_n,\delta_n})=G_0(u_0)=e_\cons(u_0).
$$
\end{remark}

\begin{proof}
For convenience writing $(\delta,\eps)=(\delta_n,\eps_n)$, for $\xi \in \{0,\eps\}$ we write
$$
G_0(u_0)-G_{\xi}(u_{\eps}^{\delta,\delta})=G_0(u_0)-G_{\xi}(u_{\eps})+G_{\xi}(u_{\eps})-G_{\xi}(u_{\eps}^{\delta,\delta}).
$$
Since  $\lim_{n \rightarrow \infty} \|u_0-u_{\eps}\|_X =0$  {by Remark~\ref{rem:Gamma}}, and hence $\lim_{n \rightarrow \infty} \|\mu_0-\mu_{\eps}\|_Y =
\lim_{n \rightarrow \infty} \|B (u_0-u_{\eps})\|_{Y'}=0$,  we have $\lim_{n \rightarrow \infty} G_{\xi}(u_{\eps})=G_0(u_0)$.
 As shown in \eqref{40}, it holds that
\begin{align*}
|G_{\xi}&(u_{\eps})-G_{\xi}(u_{\eps}^{\delta,\delta})| \leq \nrm u_{\eps}-u_{\eps}^{\delta,\delta}\nrm_\eps
\lesssim \min_{(\tilde{\mu},\tilde{u}) \in Y^\delta\times X^\delta} \|\mu_\eps -\tilde{\mu}\|_Y+\nrm u_\eps-\tilde{u}\nrm_\eps
\\
&\lesssim \min_{(\tilde{\mu},\tilde{u}) \in Y^\delta\times X^\delta} \|\mu_\eps -\tilde{\mu}\|_Y+\|u_\eps-\tilde{u}\|_X
\\ & \leq  \|\mu_0-\mu_\eps \|_Y+\|u_0-u_\eps\|_X+
\min_{(\tilde{\mu},\tilde{u}) \in Y^\delta\times X^\delta} \|\mu_0 -\tilde{\mu}\|_Y+\|u_0-\tilde{u}\|_X  \rightarrow 0
\end{align*}
for $n \rightarrow \infty$.
\end{proof}}

The above results hinge on the validity of \eqref{inf-sup}. When $A(t) \equiv A$ is a spatial (second order elliptic) differential operator with \emph{constant coefficients} on a \emph{convex} polytopal domain $\Omega$, and $X^\delta$ is a \emph{lowest order} finite element space w.r.t.~\emph{quasi-uniform} prismatic elements, we will be able to verify
in \S\ref{Sstandard} the inf-sup condition \eqref{inf-sup}.

Since we are able to show \eqref{inf-sup} only under such restrictive conditions on $\Omega$ and the trial spaces $X^\delta$,
 we will consider in Sect.~\ref{SFOSLS} a \emph{First Order System Least Squares} formulation of the data assimilation problem, for which a corresponding inf-sup condition will be shown in more general situations in \S\ref{SFOSLS2}.

Stability of the discretization, and hence  \eqref{inf-sup}, is in particular intimately
 connected with  {\em a posteriori accuracy control}. 
   A well-known tool for establishing \eqref{inf-sup} is
 the identification of suitable Fortin operators which also serve to define appropriate notions of {\em data oscillation} as
 discussed next.


\subsection{Fortin operators, a posteriori error estimation and data-oscillation} \label{Sapost}
It is well-known that existence of uniformly bounded Fortin interpolators is a \emph{sufficient} condition for the inf-sup condition \eqref{inf-sup} to hold.
 {In the next theorem it is shown that existence of such interpolators is also a necessary condition, and \new{quantitative} statements are provided.
\begin{theorem} Let
\be \label{fortin}
Q^\delta \in \cL(Y,Y) \text{ with } \ran Q^\delta \subset Y^\delta \text{ and } (B X^\delta)\big((\identity - Q^\delta)Y\big)=0.
\ee
Then $\gamma^\delta := \inf_{\{w \in X^\delta\colon B w \neq 0\}} \sup_{0 \neq \mu \in Y^\delta} \frac{(B w)(\mu)}{\|B w\|_{Y'} \|\mu\|_Y} \geq \|Q^\delta\|_{\cL(Y,Y)}^{-1}$.

Conversely, when $\gamma^\delta>0$, then there exists a $Q^\delta$ as in \eqref{fortin}, which is a projector, and
$\|Q^\delta\|_{\cL(Y,Y)} \new{\leq 2+}1/\gamma^\delta$.
\end{theorem}

\begin{proof} If a $Q^\delta$ as in \eqref{fortin} exists, then for $w \in X^\delta$ it holds that
$$
\|B w\|_{Y'}=\sup_{0 \neq \mu \in Y}\frac{(Bw)(\mu)}{\|\mu\|_Y}=\sup_{0 \neq \mu \in Y}\frac{(Bw)(Q^\delta \mu)}{\|\mu\|_Y}
\leq \|Q^\delta\|_{\cL(Y,Y)} \sup_{0 \neq \mu^\delta \in Y^\delta}\frac{(Bw)(\mu)}{\|\mu\|_Y},
$$
or $\gamma^\delta \geq \|Q^\delta\|_{\cL(Y,Y)}^{-1}$.

Now let $\gamma^\delta>0$. Equipping $X^\delta/\!\ker B$ with $\|B\cdot\|_{(Y^\delta)'}$\new{, given} $\mu \in Y$ consider the problem: find
$(\mu^\delta,[w^\delta]) \in Y^\delta \times X^\delta/\!\ker B$ that solves
\be \label{100}
\left(\left[\begin{array}{@{}cc@{}} R & B \\ B' & 0 \end{array}\right]
\left[\begin{array}{@{}c@{}} \mu^\delta-\mu\\ {[w^\delta]} \end{array}\right]\right)
\left[\begin{array}{@{}c@{}} \tilde{\mu}^\delta\\ {[\tilde{w}^\delta]} \end{array}\right]
=0 \quad ((\tilde{\mu}^\delta,[\tilde{w}^\delta]) \in Y^\delta \times X^\delta/\!\ker B).
\ee
\new{One verifies that  $Q^\delta:=\mu \mapsto \mu^\delta$ is a projector and satisfies \eqref{fortin}, and so what remains is to bound its norm.}

Denoting by $I^\delta_Y\colon Y^\delta \rightarrow Y$ and  $I^\delta_X\colon X^\delta/\!\ker B \rightarrow X/\!\ker B$ the trivial embeddings,
in operator language the above system reads as
$$
\left[\begin{array}{@{}cc@{}} (I^\delta_Y)' R I^\delta_Y & (I^\delta_Y)' B I^\delta_X\\
(I^\delta_X)' B' I^\delta_Y& 0 \end{array}\right]
\left[\begin{array}{@{}c@{}} \mu^\delta\\ {[w^\delta]} \end{array}\right]
=
\left[\begin{array}{@{}c@{}} \new{(I_Y^\delta)'R \mu}\\ (I^\delta_X)' B' \mu \end{array}\right].
$$
One verifies that $\new{B^\delta:=}(I^\delta_Y)' B I^\delta_X\colon X^\delta/\!\ker B  \rightarrow (Y^\delta)'$ is an isometry, and \new{furthermore that} $\new{R^\delta:=} (I^\delta_Y)' R I^\delta_Y\colon Y^\delta \rightarrow (Y^\delta)'$ is an isometric isomorphism. Therefore, \new{the Schur complement $S^\delta:={B^\delta}'{R^\delta}^{-1}B^\delta$ is an isometric isomorphism.
From
$$
\mu^\delta={R^\delta}^{-1}\big[(I_Y^\delta)'R \mu+B^\delta {S^\delta}^{-1} \big((I^\delta_X)' B' \mu-{B^\delta}' {R^\delta}^{-1}(I_Y^\delta)'R \mu\big)\big],
$$
$\|(I_Y^\delta)'R\|_{\cL(Y,{Y^\delta}')}\leq 1$, and
\begin{align*}
\|(I^\delta_X)' B'\|_{\cL(Y,(X^\delta/\!\ker B)')}&=\|B I^\delta\|_{\cL(X^\delta/\!\ker B,Y')} \\
&=
 \sup_{\{w \in X^\delta\colon B w \neq 0\}} \inf_{0 \neq \mu \in Y^\delta} \frac{\|B w\|_{Y'} \|\mu\|_Y}{(B w)(\mu)}
= 1/\gamma^\delta,
\end{align*}
we conclude that $\|\mu^\delta\|_Y \leq (2+1/\gamma^\delta)\|\mu\|_Y$ which completes the proof.}
\end{proof}}

 \begin{lemma} \label{lem10} Let $(X^\delta,Y^\delta)_{\delta \in \Delta} \subset X \times Y$ be such that \eqref{inf-sup} is satisfied, and let $(Q^\delta)_{\delta \in \Delta}$ be a corresponding family of  {uniformly bounded} Fortin interpolators as in \eqref{fortin}. Then with
$$
 e^\delta_\osc(g):=\|(\identity-{Q^\delta}')g \|_{Y'}\footnotemark
$$
\footnotetext{A similar data-oscillation term appears in \cite{35.93556} dealing with the derivation of a posteriori error estimators for minimal residual methods w.r.t.~a dual norm (as our norm on $Y'$).}
one has for any $\eps \geq 0$ and $w \in X^\delta$,
$$
\sqrt{G_\eps(w)} \lesssim \sqrt{G^\delta_\eps(w)}+ e^\delta_\osc(g).
$$
\end{lemma}

\begin{proof} \new{Thanks to $(\identity -{Q^\delta}') B X^\delta=0$,} the proof follows from
\begin{align*}
\|B w-g\|_{Y'} & \leq \|{Q^\delta}'(B w- g)\|_{Y'}+\|(\identity -{Q^\delta}')g\|_{Y'}\\
& \leq \|Q^\delta\|_{\cL(Y,Y)}\|B w- g\|_{{Y^\delta}'}+ e^\delta_\osc(g). \qedhere
\end{align*}
\end{proof}

Together Proposition~\ref{prop1} and Lemma~\ref{lem10} show the following a posteriori error bound.
\begin{corollary} \label{corol1} In the situation of Lemma~\ref{lem10}, one has for $\eps \geq 0$ and $w \in X^\delta$,
$$
\|u_\eps-w\|_{X_\eta} \lesssim \sqrt{G_\eps^\delta(w)}  + e^\delta_\osc(g).
$$
\end{corollary}

Lemma~\ref{lem10} can also be used to compute an a posteriori upper bound, modulo data-oscillation, for the minimal consistency error.

\begin{corollary} \label{corol}
Adhering to the setting  in Lemma~\ref{lem10}, one has for 
any $w \in X^\delta$
$$
e_\cons(u_0) \lesssim \sqrt{G_0^\delta({w})}  + e^\delta_\osc(g).
$$
\end{corollary}

\begin{proof} The proof follows from
$e_\cons(u_0)  = \sqrt{G_0(u_0)} \leq \sqrt{G_0({w})}$
and an application of Lemma~\ref{lem10}.
\end{proof}

In view of Proposition~\ref{prop2} or Theorem \ref{thm1}, this upper bound on $e_\cons(u_0)$  narrows the range for appropriate regularization parameters balancing accuracy of the state estimator
and the condition of corresponding discrete systems.

In the light of  the a priori error bound from Theorem~\ref{thm1} the above observations hint at further desirable properties of the family $(Y^\delta)_{\delta \in \Delta}$
associated with given trial spaces $(X^\delta)_{\delta \in \Delta}$. Namely, they should permit the construction
of  {uniformly bounded}  Fortin interpolators $Q^\delta$, as in \eqref{fortin}, for which \emph{in addition},
\be
\label{oscillation}
 e^\delta_\osc(g) = {\mathcal O}( {e_\ap^\delta(u)}), \text{ or even }  e^\delta_\osc(g) = o( {e_\ap^\delta(u)})
\ee
hold for sufficiently smooth ${g}$.  {For the model case mentioned at the end of \S\ref{ssec:dualnorm}, in \S\ref{Sstandard} we will construct $(Y^\delta)_{\delta \in \Delta}$
such that both \eqref{fortin} and \eqref{oscillation} are valid.}

\subsection{Comparisons with the Forward Problem}\label{ssec:compare}
To show that the solution of the least squares problem
$$
\argmin_{w \in X^\delta} \|B w-h\|^2_{{Y^\delta}'}+\|\gamma_0 w-z_0\|^2_{L_2(\Omega)}
$$
is a quasi-\new{best} approximation from $X^\delta$ to the solution of the \emph{initial-value problem} \eqref{initial}, the corresponding inf-sup condition reads as
\be \label{temp}
\inf_{\delta \in \Delta} \inf_{0 \neq w \in X^\delta} \sup_{0 \neq (\mu,z) \in Y^\delta\times L_2(\Omega)} \frac{(B w)(\mu)+\langle \gamma_0 w, z\rangle_{L_2(\Omega)}}{\|w\|_X( \|\mu\|_Y+\|z\|_{L_2(\Omega)})}>0.
\ee

The inf-sup condition \eqref{inf-sup} which is relevant for our data-assimilation problem implies \eqref{temp}. The converse is true when $\gamma_0 w=0$ for all $w \in X^\delta$.
If there is no reason to assume that the target solution $u$ of our data-assimilation problem vanishes at $t=0$, then however this is not a relevant case.

 {As shown in \cite{11,249.99} for the symmetric case $A(t)'=A(t)$ ($t \in I$ a.e.), and without this restriction in \cite{249.992},} sufficient conditions for \eqref{temp}
 are  $X^\delta \subset Y^\delta$ and
$$
\inf_{\delta \in \Delta} \inf_{0 \neq w \in X^\delta} \sup_{0 \neq \mu \in Y^\delta} \frac{(\partial_t w)(\mu)}{\|\partial_t w\|_{Y'} \|\mu\|_Y}>0
$$
This latter  inf-sup condition can be realized in far more general discretization settings than we are able to show \eqref{inf-sup}.

Even for the initial-value problem, a benefit of having \eqref{inf-sup},  {i.e.~\eqref{fortin},} is that it gives rise to the efficient and, up to a data-oscillation term, reliable a posteriori error bound
\begin{align*}
&\sqrt{\|B w-h\|^2_{{Y^\delta}'}+\|\gamma_0 w -z_0\|^2_{L_2(\Omega)}} \lesssim \\
&\hspace*{5em} \|z-w\|_X \lesssim \sqrt{\|B w-h\|^2_{{Y^\delta}'}+\|\gamma_0 w -z_0\|^2_{L_2(\Omega)}} + \osc^\delta(h).
\end{align*}
where $z$ is the solution of \eqref{initial}, and $w$ is any element of $X^\delta$.


\subsection{Numerical Solution of the Discrete Problem}  \label{Sprac1}
As in Remark \ref{rem:condensed},
by eliminating the second and third variable from the Galerkin discretization of \eqref{saddle} with trial space $Y^\delta \times L_2(I \times \omega) \times L_2(\Omega) \times X^\delta$,  the minimizer $u^{\delta,\delta}_\eps$ of $G_\eps^\delta$ over $X^\delta$ can be found as the second component of the solution $(\mu_\eps^{\delta,\delta},u_\eps^{\delta,\delta}) \in Y^\delta \times X^\delta$ of
$$
\left(\left[
\begin{array}{@{}cc@{}} R & B \\ B' & -(\Gamma_\omega' \Gamma_\omega+\eps^2\gamma_0'\gamma_0)
\end{array}
\right]
\left[
\begin{array}{@{}c@{}} \mu_\eps^{\delta,\delta} \\ u_\eps^{\delta,\delta}
\end{array}
\right]
-
\left[
\begin{array}{@{}c@{}} g \\ -\Gamma_\omega' f
\end{array}
\right]\right)
\left[
\begin{array}{@{}c@{}} \tilde{\mu} \\ \tilde{u}
\end{array}\right]
=0 \quad((\tilde \mu,\tilde u) \in Y^\delta \times X^\delta).
$$

To solve this linear system, we need to select bases. Let $\Phi^{Y^\delta}=\{\phi_1^{Y^\delta},\phi_2^{Y^\delta},\ldots\}$ and
$\Phi^{X^\delta}=\{\phi_1^{X^\delta},\phi_2^{X^\delta},\ldots\}$ denote ordered bases, formally viewed as column vectors, for $Y^\delta$ and $X^\delta$, respectively.
Writing $\mu_\eps^{\delta,\delta}=(\bm{\mu}_\eps^{\delta,\delta})^\top \Phi^{Y^\delta}$,
$u_\eps^{\delta,\delta}=(\bm{u}_\eps^{\delta,\delta})^\top \Phi^{X^\delta}$, and defining the vectors
$\bm{g}^\delta:=g(\Phi^{Y^\delta})$,
$\bm{f}^\delta_\omega:=f(\Phi^{X^\delta}|_{I \times \omega})$, introducing the matrix
$\bm{R}^\delta:=(R \Phi^{Y^\delta})(\Phi^{Y^\delta})= {[\langle \phi^{Y^\delta}_j,\phi^{Y^\delta}_i\rangle_Y}]_{ij}$, and similarly, matrices
$\bm{B}^\delta:=(B \Phi^{X^\delta})(\Phi^{Y^\delta})$,\linebreak
$\bm{M}_{\Gamma_\omega}^\delta:=\langle \Gamma_\omega \Phi^{X^\delta},\Gamma_\omega \Phi^{X^\delta}\rangle_{L_2(I \times \omega)}$,
and $\bm{M}_{\gamma_0}^\delta:=\langle \gamma_0 \Phi^{X^\delta},\gamma_0 \Phi^{X^\delta}\rangle_{L_2(\Omega)}$,
one finds the pair $(\bm{\mu}_\eps^{\delta,\delta},\bm{u}_\eps^{\delta,\delta})$ as the solution of
\be \label{DS}
\left[
\begin{array}{@{}cc@{}} \bm{R}^\delta & \bm{B}^\delta \\ {\bm{B}^\delta}^\top&
- (\bm{M}_{\Gamma_\omega}^\delta+\eps^2 \bm{M}_{\gamma_0}^\delta)
\end{array}
\right]
\left[
\begin{array}{@{}c@{}} \bm{\mu}_\eps^{\delta,\delta} \\ \bm{u}_\eps^{\delta,\delta}
\end{array}
\right]
=
\left[
\begin{array}{@{}c@{}}  \bm{g}^\delta \\ -\bm{f}^\delta_\omega
\end{array}
\right].
\ee

\begin{remark}
\label{rem:apostnum}
Using that $\|Bu_\eps^{\delta,\delta}-g\|_{{Y^\delta}'}=\|\mu_\eps^{\delta,\delta}\|_Y$, one verifies that for any $\tilde{\eps} \geq 0$,
 the a posteriori estimate from {Corollary~\ref{corol1}} for the deviation of $u_\eps^{\delta,\delta}$ from $u_{\tilde\eps}$ can be evaluated according to
%
\begin{align*}
G^\delta_{{\tilde \eps}}(u_\eps^{\delta,\delta})=&\langle \bm{R}^\delta \bm{\mu}_\eps^{\delta,\delta},\bm{\mu}_\eps^{\delta,\delta}\rangle+\langle \bm{M}_{\Gamma_\omega}^\delta \bm{u}_\eps^{\delta,\delta}, \bm{u}_\eps^{\delta,\delta}\rangle\\
&-2\langle \bm{u}_\eps^{\delta,\delta}, \bm{f}^\delta_\omega\rangle +\|f\|_{L_2(I \times \omega)}^2+ {\tilde{\eps}^2}\langle \bm{M}_{\gamma_0}^\delta \bm{u}_\eps^{\delta,\delta}, \bm{u}_\eps^{\delta,\delta}\rangle.
\end{align*}
This will later be used in the numerical experiments.
\end{remark}


For spatial domains with dimension $d>1$, the realization of any reasonable accuracy gives rise to system sizes that
require resorting to an {\em iterative solver}.
When employing a discretization based on a partition of the  time-space cylinder into ``time slabs'', the availability of
a uniformly spectrally equivalent preconditioner  $\bm{K}_Y^\delta \eqsim (\bm{R}^\delta)^{-1}$ that can be applied at linear cost,
 is actually a mild assumption.

All properties we have derived for the solution of \eqref{DS} remain valid when we replace $\bm{R}^\delta$ in this system by $(\bm{K}_Y^\delta)^{-1}$, because this replacement amounts to  replacing the $Y$-norm on $Y^\delta$ by an equivalent norm.
Therefore, despite this replacement, we continue to denote the solution vector and corresponding function in $X^\delta$ by $\bm{u}_\eps^{\delta,\delta}$ and $u_\eps^{\delta,\delta}=(\bm{u}_\eps^{\delta,\delta})^\top \Phi^{X^\delta}$, respectively.

 To approximate $\bm{u}_\eps^{\delta,\delta}$ we apply Preconditioned Conjugate Gradients to the Schur complement equation
\be
\label{Schur2ndorder}
\underbrace{({\bm{B}^\delta}^\top  {\bm{K}_Y^\delta} \bm{B}^\delta+\bm{M}_{\Gamma_\omega}^\delta+\eps^2 \bm{M}_{\gamma_0}^\delta)}_{\bm{G}_\eps^\delta:=} \bm{u}_\eps^{\delta,\delta}=
\underbrace{\bm{f}^\delta_\omega+{\bm{B}^\delta}^\top  {\bm{K}_Y^\delta} \bm{g}^\delta}_{\bm{h}^\delta:=}.
\ee
 We use a preconditioner $\bm{K}_X^\delta$ that is the representation of a uniformly boundedly invertible operator ${X^\delta}'\rightarrow X^\delta$, with $X^\delta$ and ${X^\delta}'$ being equipped with $\Phi^{X^\delta}$ and the corresponding dual  {basis}. 
Again, under the time-slab restriction, such preconditioners ${\bm{K}_X^\delta}$ of wavelet-in-time, multigrid-in-space type, that can be applied at linear cost, have been constructed in \cite{12.5,249.991}.
Assuming \eqref{inf-sup} (even \eqref{temp} suffices), it follows from \eqref{equivalence}  that
$\lambda_{\max}(\bm{K}_X^\delta \bm{G}_\eps^\delta) \lesssim \max(1,\eps^2)$ and $\lambda_{\min}(\bm{K}_X^\delta \bm{G}_\eps^\delta) \gtrsim \eps^2$.
Consequently, the number of iterations that is sufficient to reduce an initial algebraic error by a factor $\rho$
in the $\|{(\bm{G}_\eps^\delta})^{\frac12}\cdot\|$-norm\footnote{A reduction of the desired factor $\rho$ can be achieved by applying a nested iteration approach.} can be bounded by $\lesssim \eps^{-1} \log \rho^{-1}$.

 To derive a stopping criterion for the iteration,
for $\tilde{\bm{u}}_\eps^{\delta,\delta} \approx \bm{u}_\eps^{\delta,\delta}$  let $\bm{e}:= \bm{u}_\eps^{\delta,\delta}-\tilde{\bm{u}}_\eps^{\delta,\delta}$, $\bm{r}:=\bm{h}^\delta- \bm{G}_\eps^\delta \tilde{\bm{u}}_\eps^{\delta,\delta}$,
$\tilde{u}_\eps^{\delta,\delta}:=(\tilde{\bm{u}}_\eps^{\delta,\delta})^\top \Phi^{X^\delta}$,
and the {\em algebraic error} $e:=u_\eps^{\delta,\delta}-\tilde{u}_\eps^{\delta,\delta}$. Then, from \eqref{inf-sup} we have that \begin{align*}
  G_0^\delta(e)
& \leq
\|Be\|_{{Y^\delta}'}^2+\|\Gamma_\omega e\|_{L_2(I \times \omega)}^2+\eps^2\|\gamma_0 e\|_{L_2(\Omega)}^2\\
&=
\big\langle ({\bm{B}^\delta}^\top {\bm{R}^\delta}^{-1} \bm{B}^\delta +\bm{M}_{\Gamma_\omega}^\delta+\eps^2 \bm{M}_{\gamma_0}^\delta )\bm{e},\bm{e}\big\rangle\\
&\eqsim
\big\langle ({\bm{B}^\delta}^\top {\bm{K}_Y^\delta} \bm{B}^\delta +\bm{M}_{\Gamma_\omega}^\delta+\eps^2 \bm{M}_{\gamma_0}^\delta )\bm{e},\bm{e}\big\rangle=
 \langle
\bm{G}_{ {\eps}}^\delta \bm{e},\bm{e}\rangle.
\end{align*}
Moreover, for $\bm{e}\neq \bm{0}$, we have\footnotemark
\be \label{two-sided}
{\max(1,\eps^2)^{-1}} \lesssim \lambda_{\max}({\bm{K}_X^\delta} \bm{G}_{ {\eps}}^\delta)^{-1}  \leq \frac{\langle \bm{G}_{ {\eps}}^\delta \bm{e},\bm{e}\rangle}{\langle \bm{r}, {\bm{K}_X^\delta}  \bm{r}\rangle} \leq \lambda_{\min}({\bm{K}_X^\delta} \bm{G}_{ {\eps}}^\delta)^{-1} \lesssim \eps^{-2}.
\ee
\footnotetext{Instead of the possibly very pessimistic upper bound in \eqref{two-sided}, that moreover requires estimating $\lambda_{\min}( {\bm{K}_X^\delta}{\bm{G}_\eps^\delta})$, one may consult \cite{75.56,203.6,15.91} for methods to accurately estimate $\langle \bm{G}_\eps^\delta \bm{e},\bm{e}\rangle$ using data that is obtained in the PCG iteration.}

Taking $u_0$ as the reference state,  the iteration should ideally be stopped as soon as the algebraic error is dominated by $\|u_0- u_\eps^{\delta,\delta}\|_{X_\eta}$.
Ignoring data-oscillation, as an indication that $\tilde{u}_\eps^{\delta,\delta}$
is indeed close enough to ${u}_\eps^{\delta,\delta}$ we accept that the
respective upper bounds from  Corollary \ref{corol1} are close enough, i.e.,
 $\sqrt{G_0^\delta(\tilde{u}_\eps^{\delta,\delta}})$
satisfies
$\sqrt{G_0^\delta(\tilde{u}_\eps^{\delta,\delta})} \lesssim \sqrt{G_0^\delta(u_\eps^{\delta,\delta})}$.
Using $\sqrt{G_0^\delta(\tilde{u}_\eps^{\delta,\delta})} \leq \sqrt{G_0^\delta(u_\eps^{\delta,\delta})}+
\sqrt{G_0^\delta(e)}$, and the above bound for $G_0^\delta(e)$, we conclude that for the latter to hold true it suffices when for a sufficiently small constant $\mu>0$,
$$
\langle \bm{r}, {\bm{K}_X^\delta}  \bm{r}\rangle \leq \mu \eps^2 G_0^\delta(\tilde{u}_\eps^{\delta,\delta}).
$$
Since we expect  \eqref{two-sided} to be pessimistic, we simply take $\mu=1$ and thus will stop the iterative solver as soon as
$\langle \bm{r}, {\bm{K}_X^\delta}  \bm{r}\rangle \leq \eps^2 G_0^\delta(\tilde{u}_\eps^{\delta,\delta})$.


\section{First order system least squares (FOSLS) formulation} \label{SFOSLS}
In view of the difficulty to demonstrate the inf-sup condition \eqref{inf-sup} in general settings for the second order weak formulation of the data assimilation problem,  we consider in this section  a regularized FOSLS formulation. Its analysis builds to a large extent on the
concepts used in Section  \ref{Sregularized}.

For $\vec{b} \in L_\infty(I \times \Omega)^d$, $c \in L_\infty(I \times \Omega)$, and uniformly positive definite $K=K^\top \in L_\infty(I \times \Omega)^{d\times d}$, we consider $a(t;\theta,\zeta)$ as in \eqref{1}-\eqref{2} of the form
\be \label{special_a}
a(t;\theta,\zeta)=\int_\Omega K\nabla \theta \cdot \nabla \zeta\,d{\bf x}+(\vec{b} \cdot \nabla \theta + c \theta) \zeta \,d{\bf x}.
\ee
Adhering to the definitions of the spaces $X,Y$  from the previous sections, we abbreviate $Z:=L_2(I;L_2(\Omega)^d)$ and
consider the operator $C \in \cL(X \times Z,Y')$, given by
\be
\label{Cwq}
C(w,\vec{q})(v):=\int_I \int_\Omega
\partial_t w v + \vec{q}\cdot\nabla_{\bf x} v+ (\vec{b}\cdot \nabla_{\bf x} w+c w)v\,d{\bf x}  dt.
\ee
Moreover, we introduce the corresponding least squares functional
  $H_\eps\colon X \times Z \rightarrow \R$,
defined by
\begin{align*}
&H_\eps(w,\vec{q}):=\\
&\|C(w,\vec{q})-g\|_{Y'}^2+
 \| \vec{q}-K \nabla_{\bf x} w\|_{Z}^2
+
\|\Gamma_\omega w-f\|_{L_2(I \times \omega)}^2+ \eps^2\|\gamma_0 w\|_{L_2(\Omega)}^2.
\end{align*}
The following simple observations allow us to tie the analysis of the corresponding minimization problem to the
the concepts developed in the previous section.
\begin{remark}
\label{rem:simple}
One has for any $w\in X$  
$$
C(w,K\nabla_{\bf x} w)(v) = (Bw)(v),\quad v\in Y,
$$
and more generally, for any $(w,\vec{q},\ell) \in X \times Z \times Y'$,
$$
C(w,\vec{q}) - \ell (v) = (Bw)(v) -\ell (v) + \int_{I\times\Omega}(\vec{q} - K\nabla_{\bf x}w)\cdot \nabla_{\bf x} v dxdt.
$$
Hence
\be \label{42}
\begin{split}
\|B w -\ell\|_{Y'} &\leq \|C(w,\vec{q})-\ell\|_{Y'} +\|v \mapsto \int_I \int_\Omega (\vec{q}-K \nabla_{\bf x} w)\cdot \nabla_{\bf x} v\,d{\bf x}\,dt \|_{Y'}\\
&\leq \|C(w,\vec{q})-\ell\|_{Y'} +\|\vec{q}-K \nabla_{\bf x} w \|_{Z},
\end{split}
\ee
which, with $\ell = g$ in particular implies that
\be
\label{GH}
 H_\eps(w,K\nabla_{\bf x} w)= G_\eps(w)\le 2 H_\eps(w,\vec{q}),\quad \text{for any}\,\, \vec{q}\in Z.
\ee
\end{remark}
\medskip

Using \eqref{42}, 
 one infers from \eqref{equivalence}  that
$$
\eps^2 \lesssim \frac{
\|C(w,\vec{q})\|_{Y'}^2+
 \|\vec{q}-K \nabla_{\bf x} w\|_{Z}^2
+
\|\Gamma_\omega w\|_{L_2(I \times \omega)}^2+ \eps^2\|\gamma_0 w\|_{L_2(\Omega)}^2
}{\|w\|_X^2+\|\vec{q}\|_Z^2}
 \lesssim \max(1,\eps^2).
$$
By an application of the Lax-Milgram Lemma, we conclude that for $\eps>0$ the minimizer $(\bar{u}_\eps,\vec{p}_\eps)$ over $X \times Z$ of $H_\eps$ exists uniquely, and satisfies
$$
\|\bar{u}_\eps\|_X+\|\vec{p}_\eps\|_Z \lesssim  \max(\eps^{-1},1) (\| g\|_{Y'}+\|f\|_{L_2(I \times \omega)}),
$$
as well as,  for any reference state $u \in X$,
  $$
  \|u-\bar{u}_\eps\|_X+\|K \nabla u-\vec{p}_\eps\|_Z
\lesssim \max(\eps^{-1},1)\big(\eps \|\gamma_0 u\|_{L_2(\Omega)} + e_{\cons} {(u)}\big).
$$
Again, using \eqref{carleman}, much better bounds will be established for $\|u-\bar{u}_\eps\|_{X_\eta}$.

\begin{remark}
\label{rem:1order0}
 Also for $\eps=0$, the minimizer $(\bar{u}_0,\vec{p}_0)$ exists uniquely.
Indeed, let there be two minimizers of $H_0$ over $X \times Z$. Then their difference $(e_0,\vec{e}_0)$ is a homogeneous solution
of the corresponding Euler-Lagrange equations which, in turn, implies
$
\|C(e_0,\vec{e}_0)\|_{Y'}^2+\|\vec{e}_0-K\nabla_{\bf x} e_0\|^2_Z+\|\Gamma_\omega e_0\|_{L_2(I \times \omega)}^2=0$, and so
$\|B e_0\|_{Y'}^2+\|\Gamma_\omega e_0\|_{L_2(I \times \omega)}^2=0$, which as we have seen, implies that $e_0=0$, and so $\vec{e}_0=0$.
\end{remark}

\begin{proposition} \label{lem_basic3}
For any $w \in X$, $\vec{q} \in Z$ one has
$$
\|u-w\|_{X_\eta} \lesssim \sqrt{H_{0}(w,\vec{q})}+e_{\cons} {(u)}.
$$
 In particular, for $\eps\geq 0$, we have the \emph{ a posteriori bound}
$$
 \|\bar{u}_\eps-w\|_{X_\eta} \lesssim \sqrt{H_\eps(w,\vec{q})}.
$$
\end{proposition}

\begin{proof}
Lemma~\ref{lem_basic} gives $\|u-w\|_{X_\eta} \lesssim \sqrt{G_{0}(w)} + e_{\cons}{(u)}$,
and $G_{0}(w) \leq 2H_{0}(w,\vec{q})$ by \eqref{GH}.
The second result follows from
\begin{align*}
& \sqrt{H_0(w,\vec{q})}+e_{\cons}(\bar{u}_\eps)
\leq
\sqrt{H_\eps(w,\vec{q})}+\sqrt{G_\eps(\bar{u}_\eps)}\\
&\quad \leq
\sqrt{H_\eps(w,\vec{q})}+\sqrt{2} \sqrt{H_\eps(\bar{u}_\eps,\vec{p}_\eps)}
\leq (1+\sqrt{2})\sqrt{H_\eps(w,\vec{q})}. \qedhere
\end{align*}
\end{proof}

The same arguments used to show for $\eps\geq 0$ existence and uniqueness of the minimizer $(\bar{u}_\eps,\vec{p}_\eps)$ of $H_\eps$ over $X \times Z$ show for any closed subspace $X^\delta \times Z^\delta \subset X \times Z$
uniqueness of the minimizer $(u^\delta_\eps,\vec{p}_\eps^\delta)$ of $H_\eps$ over $X^\delta \times Z^\delta$.
 An \emph{a priori bound} for  $\|u-\bar{u}_\eps^\delta\|_{X_\eta}$ for an arbitrary reference state $u \in X$ is given in the next proposition.

\begin{proposition}
\label{cor:4.2}
It holds that
$$
\|u-\bar{u}_\eps^\delta\|_{X_\eta} \lesssim e_{\cons} {(u)}+\bar{e}_{\ap}^\delta {(u)}+\eps \|\gamma_0 u\|_{L_2(\Omega)},
$$
where $\bar{e}_{\ap}^\delta(u):=\min_{(w,\vec{q}) \in X^\delta \times Z^\delta} \|u-w\|_X+\|K\nabla_{\bf x} u-\vec{q}\|_Z$.
\end{proposition}

\begin{proof}
Let $P_{X^\delta}$ and  $P_{Z^\delta}$ denote the $X$- respectively $Z$-orthogonal   projector onto $X^\delta$ respectively $Z^\delta$. Then
we  have
\begin{align*}
\sqrt{H_\eps(P_{X^\delta} u,P_{Z^\delta} K \nabla_{\bf x} u)}&\leq
\|C(P_{X^\delta} u,P_{Z^\delta} K\nabla_{\bf x} u)-C(u,K\nabla_{\bf x} u)+Bu-g\|_{Y'}\\
&\quad+
 \|P_{Z^\delta} K \nabla_{\bf x}u-K \nabla_{\bf x} u+K \nabla_{\bf x}u-K \nabla_{\bf x}P_{X^\delta} u \|_{Z}\\
&\quad+
 \|\Gamma_\omega(P_{X^\delta} u-u)+\Gamma_\omega u-f\|_{L_2(I \times \omega)}+ \eps\|\gamma_0 P_{X^\delta} u\|_{L_2(\Omega)} \\
&\leq \|C(P_{X^\delta} u-u,P_{Z^\delta} K\nabla_{\bf x} u-K\nabla_{\bf x} u)\|_{Y'} + \|Bu-g\|_{Y'} \\
&\quad  + \|P_{Z^\delta} K \nabla_{\bf x}u-K \nabla_{\bf x} u\|_Z + \|K\nabla_{\bf x}(u- P_{X^\delta} u) \|_{Z} \\
& \quad  + \|\Gamma_\omega(P_{X^\delta} u-u)\|_{L_2(I \times \omega)}+\|\Gamma_\omega u-f\|_{L_2(I \times \omega)}\\
&\quad + \eps \|\gamma_0(u-P_{X^\delta} u)\|_{L_2(\Omega)} + \eps \|\gamma_0u \|_{L_2(\Omega)}\\
&\lesssim \eps{\|\gamma_0u\|_{L_2(\Omega)}}+e_{\cons}(u)+\bar{e}_{\ap}^\delta(u),
\end{align*}
where we have used $C \in \cL(X \times Z,Y')$, $K\nabla_{\bf x} \in \cL(Z,X)$, and \eqref{embedding}.
 Since by Proposition~\ref{lem_basic3}, $\|u-\bar{u}_\eps^\delta\|_{X_\eta} \lesssim \sqrt{H_\eps(\bar{u}_\eps^\delta,\vec{p}_\eps^\delta)} +e_{\cons} {(u)}$ and since, by definition,  $H_\eps(\bar{u}_\eps^\delta,\vec{p}_\eps^\delta) \leq H_\eps(P_{X^\delta} u,P_{Z^\delta} K \nabla_{\bf x} u)$,
  the proof is completed.
\end{proof}

 Since the definition of $H_\eps$ incorporates the dual norm $\|\cdot\|_{Y'}$ neither its minimizer $(\bar{u}_\eps^\delta,\vec{p}_\eps^\delta)$ over $X^\delta \times Z^\delta$ can be computed, nor the a posteriori error bound from Proposition~\ref{lem_basic3} can be evaluated. In the next subsection both problems will be tackled by discretizing this dual norm.

\begin{remark} \label{rem:KarkulikFuehrer} Our FOSLS formulation of the data-assimilation problem has been based on the fact that a well-posed FOSLS formulation of the initial-value problem \eqref{initial}, with $(A(t) \theta)(\zeta)=a(t;\theta,\zeta)$ of the form \eqref{special_a}, is given by
$$
 {\argmin_{(w,\vec{q}) \in X \times Z}
\|C(w,\vec{q})-g\|_{Y'}^2+
 \| \vec{q}-K \nabla_{\bf x} w\|_{Z}^2
+ \|\gamma_0 w-z_0\|_{L_2(\Omega)}^2},
$$
see \cite[Lem.~2.3 \new{and Rem.~2.4}]{243.85}. \new{Notice that with well-posedness we mean that $(w,\vec{q}) \mapsto(C(w,\vec{q}), \vec{q}-K \nabla_{\bf x} w, \gamma_0 w) \in \Lis(X \times Z,Y' \times Z \times L_2(\Omega))$.}
In the recent work \cite{75.257} it was shown that an alternative well-posed FOSLS formulation for this problem\footnote{The result given in \cite{75.257} for the heat equation immediately generalizes to the more general parabolic problem under consideration,  {see \cite{75.28}}. \new{Surjectivity of $(w,\vec{q}) \mapsto(C(w,\vec{q}), \vec{q}-K \nabla_{\bf x} w, \gamma_0 w)$ has also been shown in the latter work.}} is given by
$$
   {  \qquad\qquad\quad\argmin\limits_{\mathclap{\{ {(w,\vec{q}) \in X \times Z\colon}{ \partial_t w-\divv_x \vec{q} \in L_2(I;L_2(\Omega))}\}}}
\,\,\,\|C(w,\vec{q})-g\|_{L_2(I; L_2(\Omega))}^2 +
 \| \vec{q}-K \nabla_{\bf x} w\|_{Z}^2
+ \|\gamma_0 w-z_0\|_{L_2(\Omega)}^2.}
$$
Applying the latter formulation to the data-assimilation setting would offer the important advantage that there is no need to discretize the dual norm $\|\,\|_{Y'}$.
On the other hand, error estimates for such a formulation would be based on the estimate
$\|w\|_{X_\eta} \lesssim \| {\Gamma_\omega}w\|_{L_2(I \times \omega)}+\|Bw\|_{L_2(I; L_2(\Omega))}$,
which is a weaker version of the Carleman estimate
$\|w\|_{X_\eta} \lesssim \| {\Gamma_\omega}w\|_{L_2(I \times \omega)}+\|Bw\|_{Y'}$.
Furthermore, in view of an iterative solution process, a likely non-trivial issue is the development of optimal preconditioners for the space $\{(w,\vec{q}) \in X \times Z\colon \partial_t w-\divv_x \vec{q} \in L_2(I;L_2(\Omega))\}$ equipped with the graph norm.
\end{remark}

\subsection{Discretizing the dual norm}
Given a family of finite dimensional subspaces $(X^\delta \times Z^\delta)_{\delta \in \Delta}$ of $X \times Z$, for each $\delta \in \Delta$ we seek a finite dimensional subspace $\bar{Y}^\delta \subset Y$, with $\dim \bar{Y}^\delta \lesssim \dim X^\delta+\dim Z^\delta$, such that
 {in analogy to \eqref{inf-sup}}
\be \label{inf-sup2}
\inf_{\delta \in \Delta} \inf_{0 \neq (w,\vec{q}) \in X^\delta\times Z^\delta} \sup_{0 \neq \mu \in \bar{Y}^\delta} \frac{C(w,\vec{q})(\mu)}{\|C(w,\vec{q})\|_{Y'} \|\mu\|_Y}>0.
\ee

\begin{theorem} \label{thm10} Let \eqref{inf-sup2} be valid. For $(\bar{u}_\eps^{\delta,\delta},\vec{p}_\eps^{\delta,\delta})$ denoting the (unique) minimizer over $X^\delta \times Z^\delta$ of
\begin{align*}
&H^\delta_\eps  := (w,\vec{q}) \mapsto\\
&\|C(w,\vec{q})-g\|_{{\bar{Y}^\delta}'}^2+
 \|\vec{q}-K \nabla_{\bf x} w\|_{Z}^2
+
\|\Gamma_\omega w-f\|_{L_2(I \times \omega)}^2+ \eps^2\|\gamma_0 w\|_{L_2(\Omega)}^2,
\end{align*}
it holds that
$$
\|u-\bar{u}_\eps^{\delta,\delta}\|_{X_\eta} \lesssim e_{\cons} {(u)}+\bar{e}_\ap^\delta {(u)}+\eps\|\gamma_0 u\|_X.
$$
\end{theorem}

\begin{proof} Equipping $X \times Z$ with ``energy''-norm
$$
\nrmm (w,\vec{q})\nrmm_\eps:= \sqrt{
\|C(w,\vec{q})\|_{Y'}^2+\|\vec{q}-K \nabla_{\bf x} u\|_Z^2+\|\Gamma_\omega w\|_{L_2(I \times \omega)}^2+ \eps^2 \|\gamma_0 w\|_{L_2(\Omega)}^2,
}
$$
analogously to the proof of Theorem~\ref{thm1}, in particular following
the same reasoning that leads to \eqref{40}, one concludes that
\begin{align}
\label{analog}
\nrmm (\bar{u}_\eps,\vec{p}_\eps)-(\bar{u}^{\delta,\delta}_\eps,\vec{p}^{\delta,\delta}_\eps)\nrmm_\eps
&\lesssim \|g -  C(\bar{u}_\eps,\vec{p}_\eps)\|_{Y'} \nonumber\\
&\qquad\quad +\min_{(w,\vec{q}) \in X^\delta \times Z^\delta} \nrmm (\bar{u}_\eps,\vec{p}_\eps) -(w,\vec{q})\nrmm_\eps .
\end{align}

From \eqref{carleman}, the triangle-inequality, and \eqref{42} we have
\begin{align*}
\|u-\bar{u}_\eps^{\delta,\delta}\|_{X_\eta} &\lesssim
\nrm u-\bar{u}_\eps\nrm_\eps
+
\nrm \bar{u}_\eps-\bar{u}_\eps^{\delta,\delta}\nrm_\eps\\
& \leq \sqrt{2} (\nrmm (u,K \nabla_{\bf x} u) -(\bar{u}_\eps,\vec{p}_\eps)\nrmm_\eps+
\nrmm (\bar{u}_\eps,\vec{p}_\eps)-(\bar{u}^{\delta,\delta}_\eps,\vec{p}^{\delta,\delta}_\eps)\nrmm_\eps )
\end{align*}
From \eqref{analog} one infers
\begin{align*}
 \nrmm (\bar{u}_\eps,\vec{p}_\eps)-(\bar{u}^{\delta,\delta}_\eps,&\vec{p}^{\delta,\delta}_\eps)\nrmm_\eps
\\
&\lesssim
\|g -  C(\bar{u}_\eps,\vec{p}_\eps)\|_{Y'} +
\nrmm (u,K \nabla_{\bf x} u) -(\bar{u}_\eps,\vec{p}_\eps)\nrmm_\eps+\bar{e}_\ap(u),
\end{align*}
 where we have used that $\nrmm\,\nrmm_\eps \lesssim \|\,\|_{X \times Z}$.
An application of a triangle-inequality for the norm $\sqrt{\|\,\|_{Y'}^2+\|\,\|_Z^2+\|\,\|_{L_2(I\times\omega)}^2+\eps^2\|\,\|_{L_2(\Omega)}^2}$ gives
\begin{align*}
\|g - & C(\bar{u}_\eps,\vec{p}_\eps)\|_{Y'} +
\nrmm (u,K \nabla_{\bf x} u) -(\bar{u}_\eps,\vec{p}_\eps)\nrmm_\eps
\\
&\leq \sqrt{H_0(\bar{u}_\eps,\vec{p}_\eps)}+\sqrt{H_\eps(u,K \nabla_{\bf x} u)}+\sqrt{H_\eps(\bar{u}_\eps,\vec{p}_\eps)}
\\
&\leq 3\sqrt{H_\eps(u,K \nabla_{\bf x} u)}
=3\sqrt{G_\eps(u)} \leq 3 (e_\cons(u)+\eps\|\gamma_0 u\|_{L_2(\Omega)}).
\end{align*}
By combining these estimates from the last three displayed fomulas the proof is completed.
\end{proof}

Similar to Sect.~\ref{Sapost}, a necessary and sufficient condition for \eqref{inf-sup2} to hold is the existence of a family of
 uniformly bounded Fortin interpolators, i.e.,
\be \label{fortin2}
\bar{Q}^\delta \in \cL(Y, {\bar{Y}}^\delta), \quad (C (X^\delta \times Z^\delta))\big((\identity - \bar{Q}^\delta)Y\big)=0,\quad
\sup_{\delta \in \Delta} \|\bar{Q}^\delta\|_{\cL(Y,Y)}<\infty.
\ee

Similar to Corollary~\ref{corol1}, we have the following a posteriori error bound.

\begin{proposition} \label{prop11} Let $(X^\delta,Z^\delta,  {\bar{Y}^\delta})_{\delta \in \Delta} \subset X \times Z \times Y$ be such that \eqref{inf-sup2} is satisfied, and let $(\bar{Q}^\delta)_{\delta \in \Delta}$ be a corresponding family of Fortin interpolators as in \eqref{fortin2}. Then with
$$
\bar{e}_\osc^\delta {(g)}:=\|(\identity-{\bar{Q}^\delta}')g \|_{Y'},
$$
for any $(w,\vec{q}) \in X^\delta \times Z^\delta$ it holds that
$$
\|\bar{u}_\eps-w\|_{X_\eta} \lesssim \sqrt{H_\eps^\delta(w,\vec{q})}+ {\bar{e}^\delta_\osc(g)}.
$$
\end{proposition}

\begin{proof} From $\|C(w,\vec{q})-g\|_{Y'} \leq \|{\bar{Q}^\delta}\|_{\cL(Y,Y)}\|C(w,\vec{q})-g\|_{{\bar{Y}^\delta}'}+\bar{e}^\delta_\osc(g)$
and Proposition~\ref{lem_basic3} the proof follows.
\end{proof}

Bearing the a priori error bound from Theorem~\ref{thm10} in mind,
 this result  shows that a desirable additional property of the sequence of spaces $(\bar{Y}^\delta)_{\delta \in \Delta},$
  associated with a given sequence
 of trial spaces  $(X^\delta \times Z^\delta)_{\delta \in \Delta}$,  gives rise to Fortin interpolators $\bar{Q}^\delta$, as in \eqref{fortin2},
 warranting   for sufficiently smooth $g$
$$
\bar{e}_\osc^\delta {(g)} = {\mathcal O}(\bar{e}_\ap^\delta {(u)}), \text{ or even } \bar{e}_\osc^\delta {(g)} = o(\bar{e}_\ap^\delta) {(u)}.
$$

We conclude by remarking that, in analogy to the second order formulation,  condition \eqref{inf-sup2}
is sufficient for the well-posedness of the corresponding forward problem, and  gives in addition an a posteriori error bound.

\begin{remark} Concerning the \emph{initial-value problem} \eqref{initial}, if \eqref{inf-sup2} is satisfied, then
$$
 {\argmin_{(w,\vec{q}) \in X^\delta \times Z^\delta}
\|C(w,\vec{q})-h\|_{{\bar{Y}^\delta}'}^2+
 \|\vec{q}-K \nabla_{\bf x} w\|_{Z}^2
+\|\gamma_0 w-z_0\|_{L_2(\Omega)}^2}
$$
is a quasi-best approximation to $(z,K \nabla_{\bf x} z) \in X \times Z$ from $X^\delta \times Z^\delta$, and
for any $(w,\vec{q}) \in X^\delta \times Z^\delta$, it holds
\begin{align*}
&\sqrt{\|C(w,\vec{q})-h\|_{{\bar{Y}^\delta}'}^2+
 \|\vec{q}-K \nabla_{\bf x} w\|_{Z}^2
+\|\gamma_0 w-z_0\|_{L_2(\Omega)}^2}  \\
&\quad\lesssim \|z-w\|_X + \|K \nabla_{\bf x} z-\vec{q}\|_Z \\
&\quad\lesssim
\sqrt{\|C(w,\vec{q})-h\|_{{\bar{Y}^\delta}'}^2+
 \|\vec{q}-K \nabla_{\bf x} w\|_{Z}^2
+\|\gamma_0 w-z_0\|_{L_2(\Omega)}^2}
+ \bar{e}^\delta_\osc(h).
\end{align*}
\end{remark}

\subsection{Numerical Solution of the Discrete Problem} \label{Sprac2}
Recalling the Riesz operator $R \in \Lis(Y,Y')$, $\bar{u}_\eps^{\delta,\delta}$ can be practically computed as the second component of the solution
$(\lambda_\eps^{\delta,\delta},\bar{u}_\eps^{\delta,\delta},\vec{p}_\eps^{\delta,\delta}) \in \bar{Y}^\delta\times X^\delta \times Z^\delta$ of the linear system
$$
\left(\left[
\begin{array}{@{}ccc@{}}
R &
C_u &
C_{\vec{p}} \\
C_u' & - (\nabla_{\bf x}' K^2 \nabla_{\bf x} \!+\Gamma_\omega' \Gamma_\omega\!+\!\eps^2\gamma_0'\gamma_0)
&  \nabla_{\bf x}' K  \\
C_{\vec{p}}' & K \nabla_{\bf x}  &-\identity
\end{array}
\right]
\left[
\begin{array}{@{}c@{}} \lambda_\eps^{\delta,\delta} \\ \bar{u}_\eps^{\delta,\delta} \\ \vec{p}_\eps^{\delta,\delta}
\end{array}
\right]
-
\left[
\begin{array}{@{}c@{}} g \\ - \Gamma_\omega' f \\0
\end{array}
\right]
\right)
\left[
\begin{array}{@{}c@{}} \tilde{\lambda} \\ \tilde{u} \\ \tilde{\vec{p}}
\end{array}
\right]=0,
$$
($(\tilde{\lambda},\tilde{u},\tilde{\vec{p}}) \in \bar{Y}^\delta\times X^\delta \times Z^\delta$),
where for $C(\cdot,\cdot)$ defined by \eqref{Cwq},
$(C_u w)(v):=C(w,0)(v)$ and $(C_{\vec{p}} q)(v):=C(0,q)(v)$.

With ordered bases $\Phi^{\bar{Y}^\delta}$, $\Phi^{X^\delta}$, and $\Phi^{Z^\delta}$ for $\bar{Y}^\delta$, $X^\delta$, and $Z^\delta$, and the previously used or otherwise obvious notations
$\bm{\lambda}_\eps^{\delta,\delta}$,  $\bar{\bm{u}}_\eps^{\delta,\delta}$, $\vec{\bm{p}}_\eps^{\delta,\delta}$,
${\bar{\bm{g}}^\delta=g(\Phi^{\bar{Y}^\delta})}$, ${\bm{f}}^\delta_\omega$,
$\bm{R}^\delta$, $\bm{C}^\delta_u$, $\bm{C}^\delta_{\vec{p}}$, $\bm{M}_{\Gamma_\omega}^\delta$, and $\bm{M}_{\gamma_0}^\delta$,
and $\bm{J}^\delta:=\langle K \Phi^{Z^\delta}, \nabla_x \Phi^{X^\delta}\rangle_{L_2(\Omega)^d}$,
$\bm{L}^\delta:=\langle K \nabla_{\bf x} \Phi^{X^\delta}, K \nabla_{\bf x} \Phi^{X^\delta}\rangle_{L_2(\Omega)^d}$,
and $\bm{N}^\delta:=\langle \Phi^{Z^\delta}, \Phi^{Z^\delta}\rangle_{L_2(\Omega)^d}$, one finds $(\bm{\lambda}_\eps^{\delta,\delta},\bar{\bm{u}}_\eps^{\delta,\delta},\vec{\bm{p}}_\eps^{\delta,\delta})$ as the solution of
\be \label{DS2nd}
\left[\begin{array}{@{}ccc@{}}
\bm{R}^\delta & \bm{C}^\delta_u & \bm{C}^\delta_{\vec{p}}\\
{\bm{C}^\delta_u}^\top & -(\bm{L}^\delta+\bm{M}_{\Gamma_\omega}^\delta+\eps^2 \bm{M}_{\gamma_0}^\delta) & \bm{J}^\delta\\
{\bm{C}^\delta_{\vec{p}}}^\top & {\bm{J}^\delta}^\top & -\bm{N}^\delta
\end{array}
\right]
\left[\begin{array}{@{}c@{}}
\bm{\lambda}_\eps^{\delta,\delta} \\ \bar{\bm{u}}_\eps^{\delta,\delta} \\ \vec{\bm{p}}_\eps^{\delta,\delta}
\end{array}
\right]
=
\left[\begin{array}{@{}c@{}}
{\bar{\bm{g}}}^\delta \\
-{\bm{f}}^\delta_\omega\\
0
\end{array}
\right].
\ee
Similarly as in Sect.~\ref{Sprac1}, one expresses the a posteriori error bound $\sqrt{H_\eps(\bar{u}_\eps^{\delta,\delta},\vec{p}_\eps^{\delta,\delta})}$ (modulo {$\bar{e}^\delta_\osc(g)$}) in terms of the vectors $\bm{\lambda}_\eps^{\delta,\delta}$, $\bar{\bm{u}}_\eps^{\delta,\delta}$, and $\vec{\bm{p}}_\eps^{\delta,\delta}$.

As in Sect.~\ref{Sprac1}, in the above system we replace $\bm{R}^\delta$ by a uniform preconditioner ${(\bm{K}_Y^\delta)^{-1}}$,
whilst keeping the same notation for the resulting solution vector and corresponding function in $\bar{Y}^\delta\times X^\delta \times Z^\delta$, and apply Preconditioned Conjugate Gradients to
the symmetric positive definite Schur complement system
\be \label{Schur1storder}
\underbrace{\left[\begin{array}{@{}cc@{}}
\bm{L}^\delta+\bm{M}_{\Gamma_\omega}^\delta+\eps^2 \bm{M}_{\gamma_0}^\delta +{\bm{C}^\delta_u}^\top {\bm{K}_Y^\delta} \bm{C}^\delta_u &
{\bm{C}^\delta_u}^\top  {\bm{K}_Y^\delta} \bm{C}^\delta_{\vec{p}}-\bm{J}^\delta\\
{\bm{C}^\delta_{\vec{p}}}^\top  {\bm{K}_Y^\delta} \bm{C}^\delta_u-{\bm{J}^\delta}^\top &
\bm{N}^\delta+{\bm{C}^\delta_{\vec{p}}}^\top  {\bm{K}_Y^\delta} \bm{C}^\delta_{\vec{p}}
\end{array}
\right]}_{\bm{H}_\eps^\delta:=}
\left[\begin{array}{@{}c@{}}
\bar{\bm{u}}_\eps^{\delta,\delta} \\ \vec{\bm{p}}_\eps^{\delta,\delta}
\end{array}
\right]
=
\left[\begin{array}{@{}c@{}}
{\bm{f}}^\delta_\omega+\bm{C}_u^\top {\bm{K}_Y^\delta} {\bar{\bm{g}}}^\delta\\
\bm{C}_{\vec{p}}^\top {\bm{K}_Y^\delta} {\bar{\bm{g}}}^\delta
\end{array}
\right].
\ee
 {With $\bm{K}^\delta_X$ from Sect.~\ref{Sprac1}, and ${\bm{K}^\delta_Z}$ being spectrally equivalent to the inverse of the mass matrix of $\Phi^{Z^\delta}$, the eigenvalues of the preconditioned system
$\left[\begin{array}{@{}cc@{}}
\bm{K}^\delta_X & 0 \\ 0 & \bm{K}^\delta_Z
\end{array}
\right] \bm{H}_\eps^\delta$ are bounded from above and below, up to constant factors, by $\max(1,\eps^2)$ and $\eps^2$, respectively}.

For $\tilde{\bm{u}}_\eps^{\delta,\delta} \approx \bm{u}_\eps^{\delta,\delta}$,
$\tilde{\vec{\bm{p}}}_\eps^{\delta,\delta} \approx \vec{\bm{p}}_\eps^{\delta,\delta}$,
with $\bm{e}_u:= \bm{u}_\eps^{\delta,\delta}-\tilde{\bm{u}}_\eps^{\delta,\delta}$,
$\bm{e}_{\vec{p}}:= \vec{\bm{p}}_\eps^{\delta,\delta}-\tilde{\vec{\bm{p}}}_\eps^{\delta,\delta}$,
$e_u:=(\bm{e}_u)^\top \Phi^{X^\delta}$,
$\vec{e}_{\vec{p}}:=(\bm{e}_{\vec{p}})^\top \Phi^{Z^\delta}$, we apply \eqref{carleman} and the arguments from the proof of Proposition~\ref{lem_basic3} to obtain
\begin{align*}
\|e_u\|^2_{X_\eta} & \lesssim \|\Gamma_\omega e_u\|^2_{L_2(I \times \omega)}+\|B e_u\|^2_{Y'}\\
& \leq \|\Gamma_\omega e_u\|^2_{L_2(I \times \omega)}+\|C(e_u,\vec{e}_{\vec{p}})\|^2_{Y'}+\|\vec{e}_{\vec{p}}-K\nabla_x e_u\|^2_X\\
& \eqsim
\|\Gamma_\omega e_u\|^2_{L_2(I \times \omega)}+\|C(e_u,\vec{e}_{\vec{p}})\|^2_{{\bar{Y}^\delta}'}+\|\vec{e}_{\vec{p}}-K\nabla_x e_u\|^2_X\\
& \leq
\|\Gamma_\omega e_u\|^2_{L_2(I \times \omega)}+\|C(e_u,\vec{e}_{\vec{p}})\|^2_{{\bar{Y}^\delta}'}+\|\vec{e}_{\vec{p}}-K\nabla_x e_u\|^2_X+\eps^2\|\gamma_0e_u\|^2_{L_2(\omega)}\\
& \eqsim \left\langle\bm{H}_{ {\eps}}^\delta \left[\begin{array}{@{}c@{}}
\bm{e}_u \\ \vec{\bm{e}}_{\vec{p}}
\end{array}
\right],
\left[\begin{array}{@{}c@{}}
\bm{e}_u \\ \vec{\bm{e}}_{\vec{p}}
\end{array}
\right]
\right\rangle,
\end{align*}
where the last ``$\eqsim$''-symbol reads as an equality for ${(\bm{K}_Y^\delta)^{-1}} =\bm{R}^\delta$.

For the residuals
$\left[\begin{array}{@{}c@{}}
\bm{r}_u \\ \vec{\bm{r}}_{\vec{p}}
\end{array}
\right]:=
\left[\begin{array}{@{}c@{}}
{\bm{f}}^\delta_\omega+\bm{C}_u^\top  {\bm{K}^\delta_Y} {\bar{\bm{g}}}^\delta\\
\bm{C}_{\vec{p}}^\top  {\bm{K}^\delta_Y} {\bar{\bm{g}}}^\delta
\end{array}
\right]-\bm{H}_{ {\eps}}^\delta \left[\begin{array}{@{}c@{}}
\bm{e}_u \\ \vec{\bm{e}}_{\vec{p}}
\end{array}
\right]$ it then holds that
$$
\max(1,\eps^2)^{-1} \left\langle \left[\begin{array}{@{}c@{}}
\bm{r}_u \\ \vec{\bm{r}}_{\vec{p}}
\end{array}
\right],
\left[\begin{array}{@{}c@{}}
 {\bm{K}^\delta_X} \bm{r}_u \\  {\bm{K}^\delta_Z} \vec{\bm{r}}_{\vec{p}}
\end{array}
\right]
\right\rangle
\lesssim
\left\langle\bm{H}_\eps^\delta \left[\begin{array}{@{}c@{}}
\bm{e}_u \\ \vec{\bm{e}}_{\vec{p}}
\end{array}
\right],
\left[\begin{array}{@{}c@{}}
\bm{e}_u \\ \vec{\bm{e}}_{\vec{p}}
\end{array}
\right]
\right\rangle
\lesssim \eps^{-2}
\left\langle \left[\begin{array}{@{}c@{}}
\bm{r}_u \\ \vec{\bm{r}}_{\vec{p}}
\end{array}
\right],
\left[\begin{array}{@{}c@{}}
 {\bm{K}^\delta_X} \bm{r}_u \\  {\bm{K}^\delta_Z} \vec{\bm{r}}_{\vec{p}}
\end{array}
\right]
\right\rangle,
$$
uniformly in $\delta$.

\begin{remark}
\label{rem:stopH}
A reasonable stopping criterion can be determined by the same reasoning
as used in Section \ref{Sprac1}. Ignoring again data oscillation we use the
a posteriori bound from Proposition \ref{prop11} to see whether the pair
$(\tilde{\bm{u}}_\eps^{\delta,\delta},\tilde{\vec{\bm{p}}}_\eps^{\delta,\delta})$
is sufficiently close to $( {\bm{u}}_0^{\delta,\delta}, {\vec{\bm{p}}}_0^{\delta,\delta})$. Specifically, we stop the iteration as soon as
$$
\left\langle \left[\begin{array}{@{}c@{}}
\bm{r}_u \\ \vec{\bm{r}}_{\vec{p}}
\end{array}
\right],
\left[\begin{array}{@{}c@{}}
 {\bm{K}^\delta_X} \bm{r}_u \\  {\bm{K}^\delta_Z} \vec{\bm{r}}_{\vec{p}}
\end{array}
\right]
\right\rangle\le \eps^2 H_0^\delta(\tilde{\bm{u}}_\eps^{\delta,\delta},\tilde{\vec{\bm{p}}}_\eps^{\delta,\delta}).
$$
\end{remark}


\section{Construction of a suitable Fortin interpolator}\label{sec:Fortin}
\newcommand{\cI}{\mathcal{I}}
The spaces $X^\delta$ and $Y^\delta$, or $X^\delta$, $Z^\delta$ and $\bar{Y}^\delta$, that we are going to employ, will be finite element spaces w.r.t.~a partition of the  time-space cylinder into `time slabs' with each time-slab being partitioned into prismatic elements.
As a preparation for the derivation of a suitable Fortin interpolator for both the standard second order formulation from \S\ref{Sregularized} and the first order order formulation from \S\ref{SFOSLS}, we start with constructing certain biorthogonal projectors acting on the spatial domain.

\subsection{Construction of auxiliary biorthogonal projectors} \label{Saux}
Let $(\tria^\delta)_{\delta \in \Delta}$,  $(\tria_S^\delta)_{\delta \in \Delta}$ be a families of conforming, uniformly shape regular partitions of $\overline{\Omega} \subset \R^d$ into, say, closed $d$-simplices,
where $\tria_S^\delta$ is a refinement of $\tria^\delta$  (denoted by $\tria^\delta \prec \tria_S^\delta$)  of some \emph{fixed} maximal depth in the sense that $|T| \gtrsim |T'|$ for
$\tria^\delta_S \ni T \subset T' \in \tria^\delta$. Thus, one still has  $\dim \tria^\delta_S \lesssim \dim \tria^\delta$. On the other hand, setting
$$
\sigma:=\sup_{\delta \in \Delta} \sup_{T' \in \tria^\delta} \sup_{\{T \in \tria_S^\delta\colon T \subset T'\}} {\textstyle \frac{|T|}{|T'|}},
$$
we will assume that this constant is sufficiently small so that the refinement is sufficiently fine.

Thanks to the conformity and the uniform shape regularity, for
$d > 1$ we know that any adjacent $T,T' \in \tria^\delta$ (or $\tria_S^\delta$) with $T \cap T' \neq \emptyset$ have uniformly
comparable sizes. For $d = 1$, we impose this uniform `K-mesh property' explicitly.

Given a conforming partition $\tria$ of $\overline{\Omega}$ into closed $d$-simplices, we define $\mathcal{S}^{-1,q}_{\tria}$ as the space of all  piecewise polynomials of degree $q$ w.r.t.~$\tria$, and for $q \geq 1$, set $\mathcal{S}^{0,q}_{\tria,0}:=\mathcal{S}^{-1,q}_{\tria} \cap H_0^1(\Omega)$. With $\partial \tria$ we denote the mesh skeleton $\cup_{\{T \in \tria\}} \partial T$. Next  we construct projectors
whose range  {is included in a} conforming finite element   {space} of prescribed degree on the refined partition and which vanish on the skeleton of the coarse
partition. Moreover, the range of their adjoints contains all piecewise polynomials of the same degree on the coarse partition, as specified next.

\begin{lemma} \label{lem-biorth} Let $q \geq 1$. Then, \new{for a sufficiently small, but \emph{fixed}} $\sigma$ there exists a family of projectors $(P_q^\delta)_{\delta \in \Delta}$ with
\begin{align} \label{35}
&\ran {P_q^\delta}' \supseteq \mathcal{S}^{-1,q}_{\tria^\delta},\quad  \ran P_q^\delta \subseteq \{w \in \mathcal{S}^{0,q}_{\tria_S^\delta,0}\colon w|_{\partial \tria^\delta}=0\},\\ \label{34}
& \|P_q^\delta w\|_{L_2(T')}  \lesssim \|w\|_{L_2(T')} \quad (T' \in \tria^\delta,\,w \in L_2(\Omega)).
\end{align}
\end{lemma}

\begin{proof}
Let $T' \in \tria^\delta$. Given $p \in {\mathcal P}_q(T')$, let $p_S \in H^1_0(T')$ denote its continuous piecewise polynomial interpolant of degree $q$ w.r.t.~to the partition $\tria_S^\delta|_{T'}$ using the canonical selection of the interpolation points, where on $\partial T '$ the interpolation values are replaced by zeros.

Obviously, $p$ and $p_S$ coincide on each $T \in \tria_S^\delta|_{T'}$ for which $T \cap \partial T'=\emptyset$. Now consider $T \in \tria_S^\delta|_{T'}$ with $T \cap \partial T'\neq\emptyset$.
Equivalence of norms on finite dimensional spaces, and standard homogeneity arguments show that
$$
\|p-p_S\|_{L_2(T)} \eqsim |T|^{\frac12} \|p-p_S\|_{L_\infty(T)} \lesssim  |T|^{\frac12}  \|p\|_{L_\infty(T')} \eqsim  |T|^{\frac12}|T'|^{-\frac12} \|p\|_{L_2(T')}.
$$
Using the uniform shape regularity of $\tria^\delta_S$ and the definition of $\sigma$, we arrive at
$$
\|p-p_S\|_{L_2(T')}^2=\sum_{\{T \in \tria_S^\delta|_{T'} \colon T \cap \partial T'\neq\emptyset\}} \|p-p_S\|_{L_2(T)}^2 \lesssim \sigma^{1/d} \|p\|_{L_2(T')}^2.
$$
From this closeness of $p$ and $p_S$,  one infers that for $\sigma$ sufficiently small,
\be \label{infsup}
\inf_{0 \neq p \in {\mathcal P}_q(T')} \sup_{0 \neq \tilde{p} \in \mathcal{S}^{0,q}_{\tria_S^\delta,0} \cap H^1_0(T')} \frac{\langle p, \tilde p\rangle_{L_2(T')}}{\|p\|_{L_2(T')}\|\tilde p\|_{L_2(T')}} \gtrsim 1,
\ee
which implies that there exists a (uniform) $L_2(T')$-Riesz
collection of functions in $\mathcal{S}^{0,q}_{\tria_S^\delta,0} \cap H^1_0(T')$ that is biorthogonal to the $L_2(T)$-normalized nodal basis for ${\mathcal P}_q(T')$.

Taking ${P_q^\delta}'$ restricted to $T'$ to be the corresponding biorthogonal projector onto ${\mathcal P}_q(T')$, it has all three stated properties.
\end{proof}

As shown in the above lemma, the projectors $P_q^\delta$ exist when $\tria_S^\delta$ is a refinement of $\tria^\delta$ of sufficient \new{\emph{fixed}} depth. \wnew{Hence, the size of the resulting linear systems remains {\em uniformly} proportional to ${\rm dim}\,X^\delta$, with a proportionality factor
depending on $\sigma$.}
In applications, one needs to know which depth suffices.
The usual procedure to construct a partition $\tria^\delta$ of the closure of a (polytopal) domain $\Omega$ is to recursively apply some fixed `affine equivalent' refinement rule to each simplex in an initial (conforming) partition of $\overline{\Omega}$. With this approach, the partition of each $T' \in \tria^\delta$ formed by its
`descendants' of some fixed generation ${\ell} \geq 1$ falls into a fixed finite number of classes $\tria_{{\ell},1}(T'),\ldots,\tria_{{\ell},N({\ell})}(T')$.
By using that the left-hand side of \eqref{infsup} is invariant under affine transformations, fixing a reference $d$-simplex $T'$ and a refinement procedure of the above type, given a degree $q$ and a generation $\ell$,
it suffices to check whether
$$
\alpha(q, {\ell}):= \inf_{1 \leq j \leq N( {\ell})} \inf_{0 \neq p \in {\mathcal P}_q(T')}
\sup_{0 \neq \tilde{p} \in  H^1_0(T') \cap \prod_{T \in \tria_{{\ell},j}(T')} {\mathcal P}_q(T)}
 \frac{\langle p, \tilde p\rangle_{L_2(T')}}{\|p\|_{L_2(T')}\|\tilde p\|_{L_2(T')}} >0,
$$
\wnew{
\begin{remark}
\label{rem:quantify}
For $d \in \{1,2,3\}$, $q \in \{1,2,3,4\}$, and both newest-vertex bisection and red-refinement, we have calculated the minimal ${\ell}$ such that $\alpha(q,{\ell})>0$.
In all cases but one, this minimal ${\ell}$ equals the minimal generation for which $\dim H^1_0(T') \cap \prod_{T \in \tria_{{\ell},j}(T')} \cP_q(T) \geq \dim {\mathcal P}_q(T')$.
Only for $d=3$, $q=4$, and newest vertex bisection, for one of the three classes it was necessary to increase this generation by one in order to ensure uniform  inf-sup stability.
\end{remark}
}

\begin{remark} \label{remmie}
For the construction of the Fortin interpolator in the FOSLS case, it will be sufficient to replace the conditions  \eqref{35}-\eqref{34} on
the projectors from Lemma~\ref{lem-biorth} by the somewhat weaker ones
\begin{align} \label{35relax}
&\ran {P_q^\delta}' \supseteq \mathcal{S}^{0,q}_{\tria^\delta,0}+\mathcal{S}^{-1,q-1}_{\tria^\delta},\quad\ran P_q^\delta \subseteq \mathcal{S}^{0,q}_{\tria_S^\delta,0},
\\ \label{34relax}
& \|\hbar_\delta^{-1} P_q^\delta \hbar_\delta \|_{\cL(L_2(\Omega),L_2(\Omega))}  \lesssim 1,
\end{align}
where $\hbar_\delta$ is the piecewise constant function defined by $\hbar_\delta|_{T'}=\diam T'$ $(T' \in \tria^\delta)$.
Note that because of the uniform `K-mesh property', \eqref{34relax} is implied by local $L_2$-stability of the form
\be \label{36}
\|P_q^\delta w\|_{L_2(T')}  \lesssim \|w\|_{L_2(\{x \in \Omega\colon d(x,T') \lesssim \diam T'\})} \quad (T' \in \tria^\delta,\,w \in L_2(\Omega)),
\ee
which, in particular,  is implied by \eqref{34}.

For $d=1$, the codimension of $\mathcal{S}^{0,q}_{\tria^\delta,0}+\mathcal{S}^{-1,q-1}_{\tria^\delta}$ in $\mathcal{S}^{-1,q}_{\tria^\delta}$ is $1$ when $q=1$, or $0$ when $q>1$.
Since we do not expect that we can benefit from the relaxation of the condition $\ran P_q^\delta \subseteq \{w \in \mathcal{S}^{0,q}_{\tria_S^\delta,0}\colon w|_{\partial \tria^\delta}=0\}$ to $\ran P_q^\delta \subseteq \mathcal{S}^{0,q}_{\tria_S^\delta,0}$,   for $d=1$,
we doubt that the relaxed conditions hold for any less deep refinement $\tria^\delta_S$ of $\tria^\delta$.

For $d>1$ and any fixed degree $q$, however, the aforementioned codimension is $\eqsim \dim \mathcal{S}^{-1,q}_{\tria^\delta}$, and we may hope that a less deep refinement $\tria^\delta_S$ of $\tria^\delta$ suffices to satisfy the relaxed conditions.

So far we have studied this issue in one particular example of $d=2$, $q=2$, and the red-refinement rule.
For this case, we could show the existence of the projectors from Lemma~\ref{lem-biorth} when $\tria^\delta_S$ is created by applying \emph{two} recursive red-refinements to each triangle from $\tria^\delta$.
In the appendix, we show that in order to satisfy the relaxed conditions \eqref{35relax} and  \eqref{36} it suffices to apply \emph{one} red-refinement.
\end{remark}

\begin{remark} When $P_q^\delta$ is an $L_2$-orthogonal projector onto a finite element space, a condition as \eqref{34relax} is known to ensure its $H^1$-stability (see e.g.~\cite{19.26}).
\end{remark}

\begin{remark}
Spaces of type $\mathcal{S}^{0,q}_{\tria^\delta,0}+\mathcal{S}^{-1,q-1}_{\tria^\delta}$, or more precisely $\mathcal{S}^{0,q}_{\tria^\delta,0}+\mathcal{S}^{-1,0}_{\tria^\delta}$, have been used as approximation spaces for the pressure in Stokes solvers to ensure local mass conservation (see e.g.~\cite{38.71}).
\end{remark}

\begin{remark}
 {Also for the construction of the Fortin interpolator for the standard second order formulation, it suffices when $\ran {P_q^\delta}' \supseteq \mathcal{S}^{0,q}_{\tria^\delta,0}+\mathcal{S}^{-1,q-1}_{\tria^\delta}$ instead of $\ran {P_q^\delta}' \supseteq \mathcal{S}^{-1,q}_{\tria^\delta,0}$. The second condition in \eqref{35} however, turns out to be essential.}
\end{remark}

\subsection{Standard, second order formulation} \label{Sstandard}
For this formulation our construction of a suitable Fortin interpolator will be restricted to second order elliptic spatial differential operators with constant coefficients on convex domains,
and \emph{lowest order} finite elements w.r.t.~partitions of the {time-space} cylinder that are \emph{Cartesian products} of a \emph{quasi-uniform} temporal mesh and a \emph{quasi-uniform} conforming, uniformly shape regular spatial mesh into $d$-simplices.

Consider the families of partitions $(\tria^\delta)_{\delta \in \Delta}$ and $(\tria_S^\delta)_{\delta \in \Delta}$ of $\overline{\Omega} \subset \R^d$ introduced in \S\ref{Saux}.
Assuming them to be \emph{quasi-uniform}, we set $h_\delta:=\max_{T' \in \tria^\delta} \diam T'$ (not to be confused with the piecewise constant function $\hbar_\delta$).

Let $(\mathcal{I}^\delta)_{\delta \in \Delta}$ be a family of \emph{quasi-uniform} partitions of $I$ into subintervals, where the lengths of the subintervals in $\mathcal{I}^\delta$  are $\eqsim h_\delta$.
We denote by $\mathcal{S}^{-1,q}_{\mathcal{I}^\delta}$ and $\mathcal{S}^{0,q}_{\mathcal{I}^\delta}$  the space of all piecewise polynomials or continuous piecewise polynomials of degree $q$ w.r.t.~$\mathcal{I}^\delta$, respectively.

\begin{theorem} \label{thm51} Let $\Omega \subset \R^d$ be a convex polytope, $a(t;\theta,\zeta)$ be of the form \eqref{special_a} for \emph{constant} $K$, $\vec{b}$ and $c$,
and let $X^\delta:=\mathcal{S}^{0,1}_{\mathcal{I}^\delta} \otimes \mathcal{S}^{0,1}_{\tria^\delta,0} \subset X$ and $Y^\delta:=\mathcal{S}^{-1,1}_{\mathcal{I}^\delta} \otimes \mathcal{S}^{0,1}_{\tria_S^\delta,0} \subset Y$, where $\tria_S^\delta$ is a sufficiently deep refinement of $\tria^\delta$ such that  a projector $P_1^\delta$ as in Lemma~\ref{lem-biorth} exists. Then, a Fortin interpolator $Q^\delta$ as in \eqref{fortin} exists, and for $g \in F:=L_2(I) \otimes H^1(\Omega)\cap H^{2}(I) \otimes H^{-1}(\Omega)$, it holds that
${e^\delta_\osc(g)} \lesssim h_\delta^2$.
\end{theorem}

\begin{remark} For this $(X^\delta)_{\delta \in \Delta}$, and a sufficiently smooth $u$ we have $e_\ap^\delta {(u)} \lesssim h_\delta$ where in general an approximation error of higher cannot be expected. So indeed, $ {e^\delta_\osc(g)}$ is of higher order as desired, cf.~\eqref{oscillation}.
\end{remark}

\begin{proof}
We are going to construct uniformly bounded $Q_t^\delta \in \cL(L_2(I),L_2(I))$, $Q_{\bf x}^\delta \in \cL(H^1_0(\Omega),H^1_0(\Omega))$ with $\ran Q_t^\delta \subset \mathcal{S}^{-1,1}_{\mathcal{I}^\delta}$,
$\ran Q_{\bf x}^\delta \subset \mathcal{S}^{0,1}_{\tria_S^\delta,0}$ and
\begin{align} \nonumber
\big\langle \mathcal{S}^{0,1}_{\mathcal{I}^\delta}, \ran(\identity- Q_t^\delta)\big\rangle_{L_2(I)}&=0=
\big\langle {\textstyle \frac{d}{d t}} \mathcal{S}^{0,1}_{\mathcal{I}^\delta}, \ran(\identity-Q_t^\delta)\big\rangle_{L_2(I)},\\ \label{14}
\big\langle \mathcal{S}^{-1,0}_{\tria^\delta,0}\!+\!\mathcal{S}^{0,1}_{\tria^\delta,0}, \ran(\identity\!-\! Q_{\bf x}^\delta)\big\rangle_{L_2(\Omega)}&=0=
\big\langle K \nabla_{\bf x} \mathcal{S}^{0,1}_{\tria^\delta,0}, \nabla_{\bf x} \ran(\identity\!-\!Q_{\bf x}^\delta)\big\rangle_{L_2(\Omega)^d}.\!\!\!\!\!\!
\end{align}
Then one verifies that $Q^\delta:=Q_t^\delta \otimes Q_{\bf x}^\delta$ satisfies the conditions in \eqref{fortin}.

A valid choice for $Q_t^\delta$ is given by the $L_2(I)$-orthogonal projector onto $\mathcal{S}^{-1,1}_{\mathcal{I}^\delta}$. It satisfies in addition
\be
\label{h2}
\|(\identity-Q_t^{\delta})'\|_{\cL(H^2(I),L_2(I))} \lesssim h_\delta^2.
\ee

We seek $Q_{\bf x}^\delta$ in the form $Q_{\bf x}^\delta=Q_{\bf x}^{A,\delta}+Q_{\bf x}^{B,\delta}+Q_{\bf x}^{B,\delta} Q_{\bf x}^{A,\delta}$
with $\ran Q_{\bf x}^{A,\delta},  \ran Q_{\bf x}^{B,\delta} \subset \mathcal{S}^{0,1}_{\tria_S^\delta,0}$ such that
\begin{align} \label{16}
&\|Q_{\bf x}^{A,\delta}\|_{\cL(H^1_0(\Omega),H_0^1(\Omega))} \lesssim 1,\quad \|\identity-Q_{\bf x}^{A,\delta}\|_{\cL(H^1_0(\Omega),L_2(\Omega))} \lesssim h_\delta,\\ \label{15}
&\big\langle K \nabla_{\bf x} \mathcal{S}^{0,1}_{\tria^\delta,0}, \nabla_{\bf x} \ran(\identity-Q_{\bf x}^{A,\delta})\big\rangle_{L_2(\Omega)^d}=0,\\ \label{17}
& \|Q_{\bf x}^{B,\delta}\|_{\cL(L_2(\Omega),L_2(\Omega))} \lesssim 1,\quad
\|(\identity - Q_{\bf x}^{B,\delta})'\|_{\cL(H^1(\Omega),L_2(\Omega))} \lesssim h_\delta,\\ \label{18}
&\big\langle \mathcal{S}^{-1,0}_{\tria^\delta,0}\!+\! \mathcal{S}^{0,1}_{\tria^\delta,0}, \ran (\identity\!-\!Q_{\bf x}^{B,\delta})\big\rangle_{L_2(\Omega)}=0=
\big\langle K \nabla_{\bf x} \mathcal{S}^{0,1}_{\tria^\delta,0}, \nabla_{\bf x} \ran Q_{\bf x}^{B,\delta}\big\rangle_{L_2(\Omega)^d}.\!\!\!\!\!
\end{align}

One easily verifies that
$$
\identity- Q_{\bf x}^\delta=(\identity- Q_{\bf x}^{B,\delta})(\identity- Q_{\bf x}^{A,\delta}),
$$
which, together with the first relation in \eqref{18},  yields the first relation in \eqref{14}. Moreover, from \eqref{15} and the second relation
in \eqref{18} one deduces the second relation \eqref{14}.

Similarly, observing that $Q_{\bf x}^\delta=Q_{\bf x}^{A,\delta}+Q_{\bf x}^{B,\delta}(\identity-Q_{\bf x}^{A,\delta})$ in combination with
\eqref{16}, $\|Q_{\bf x}^{B,\delta}\|_{\cL(L_2(\Omega),L_2(\Omega))} \lesssim 1$, and the inverse inequality $\|\,\|_{H^1(\Omega)} \lesssim h_\delta^{-1} \|\,\|_{L_2(\Omega)}$ on $\mathcal{S}^{0,1}_{\tria_S^\delta,0} \supset \ran Q_{\bf x}^{B,\delta}$, one infers that  $\|Q_{\bf x}^{\delta}\|_{\cL(H^1_0(\Omega),H_0^1(\Omega))} \lesssim 1$. Thus, all claimed properties of $Q_{\bf x}^\delta$ have been verified.

Before turning to the construction of $Q_{\bf x}^{A,\delta}$ and $Q_{\bf x}^{B,\delta}$, we estimate ${e^\delta_\osc(g)}$. For $g \in F$, we have
$$
\|(\identity-Q^\delta)' g\|_{Y'} \leq \|(\identity-Q^\delta)'\|_{\cL(F,Y')} \|g\|_F\\
 =\|\identity-Q^\delta\|_{\cL(Y,F')} \|g\|_F.
$$
Writing
$$
\identity-Q^\delta=(\identity \otimes (\identity-Q_{\bf x}^\delta))(Q_t^\delta \otimes \identity)+(\identity-Q_t^\delta) \otimes \identity,
$$
from $L_2(I) \otimes H^1(\Omega)' \hookrightarrow F'$, $H^{2}(I)' \otimes H_0^{1}(\Omega)\hookrightarrow F'$, and $\|Q^\delta_t\|_{\cL(L_2(I),L_2(I))} \lesssim 1$ we infer
\begin{align*}
&\|\identity-Q^\delta\|_{\cL(Y,F')}  \\
&\lesssim \|\identity \otimes (\identity-Q_{\bf x}^\delta)\|_{\cL(Y,L_2(I) \otimes H^1(\Omega)')}+
\|(\identity-Q_t^\delta)\otimes \identity\|_{\cL(Y,H^{2}(I)' \otimes H_0^{1}(\Omega))}\\
& = \|\identity-Q_{\bf x}^\delta\|_{\cL(H^1_0(\Omega),H^1(\Omega)')}
+\|\identity-Q_t^\delta\|_{\cL(L_2(I),H^{2}(I)')}\\
& \leq \|\identity-Q_{\bf x}^{B,\delta}\|_{\cL(L_2(\Omega),H^1(\Omega)')} \|\identity-Q_{\bf x}^{A,\delta}\|_{\cL(H^1_0(\Omega),L_2(\Omega))}
+\|(\identity-Q_t^\delta)'\|_{\cL(H^{2}(I),L_2(I))}\\
&\lesssim h_\delta h_\delta+ h_\delta^{2},
\end{align*}
where we have used \eqref{h2}, \eqref{16}, and \eqref{17}.

We now identify the operators  $Q_{\bf x}^{A,\delta}, Q_{\bf x}^{B,\delta}$. For the operator $Q_{\bf x}^{A,\delta}$, we take the `Galerkin' projector onto $\mathcal{S}^{0,1}_{\tria^\delta,0}$, i.e. the orthogonal projector w.r.t.~$\langle K \nabla_{\bf x}\cdot,\nabla_{\bf x}\cdot\rangle_{L_2(\Omega)^d}$.
It satisfies \eqref{15}, and $\|Q_{\bf x}^{A,\delta}\|_{\cL(H^1_0(\Omega),H_0^1(\Omega))}=1$.

{Thanks to $\Omega$ being a convex polytope, the homogeneous Dirichlet problem with operator $-\divv K \nabla$ is $H^2$-regular.
Indeed, by making a linear coordinate transformation that transforms the convex polytope into another convex polytope (\cite{14.4}), the operator reads as $-\triangle$ for which this regularity result is well-known.
Consequently} the usual Aubin-Nitsche duality argument shows that
$$
\|(\identity -Q_{\bf x}^{A,\delta})v\|_{L_2(\Omega)} \lesssim h_\delta \|\nabla(\identity -Q_{\bf x}^{A,\delta})v\|_{L_2(\Omega)^d} \leq h_\delta \|\nabla v\|_{L_2(\Omega)^d}
$$
holds for $v \in H^1_0(\Omega)$. This verifies the validity of \eqref{16}.

Next, we take $Q_{\bf x}^{B,\delta}=P_1^\delta$ as constructed in Lemma~\ref{lem-biorth}. It satisfies $\ran P_1^\delta \subset \mathcal{S}^{0,1}_{\tria_S^\delta,0}$,
$\|P_1^\delta\|_{\cL(L_2(\Omega),L_2(\Omega))} \lesssim 1$, and $\ran {P_1^\delta}' \supseteq \mathcal{S}^{-1,1}_{\tria^\delta}$.
The last property shows the first condition in \eqref{18}.   Using  the uniform boundedness, one concludes that
$$
\|(\identity-P_1^\delta)' w\|_{L_2(\Omega)} \lesssim \inf_{v \in \mathcal{S}^{-1,1}_{\tria^\delta}}\|w-v\|_{L_2(\Omega)} \lesssim h_\delta |w|_{H^1(\Omega)},
$$
which is the second condition in \eqref{17}.
The second condition in \eqref{18} follows by an element-wise integration-by-parts from
$w|_{\partial\tria^\delta}=0$ for any $w \in \ran P_1^\delta$, and the fact that $\mathcal{S}^{0,1}_{\tria^\delta,0}$ is a space of continuous piecewise \emph{linears}.\footnote{This argument is the sole reason why this theorem is restricted to lowest order trial spaces $X^\delta$.}
\end{proof}

\subsection{FOSLS formulation} \label{SFOSLS2}\mbox{}
We construct a suitable Fortin interpolator for the FOSLS formulation of our data assimilation problem.
In contrast to the standard second order formulation, we allow now non-convex domains $\Omega$, higher order finite element spaces w.r.t.~possibly
non-quasi-uniform partitions into prismatic elements. However, the time-space cylinder must be  partitioned into time slabs.

\begin{theorem} \label{thm3} As in Theorem~\ref{thm51}, let $a(t;\theta,\zeta)$ be of the form \eqref{special_a} for \emph{constant} $K$, $\vec{b}$ and $c$.
For $(I^\delta=(([t_i^\delta,t^\delta_{i+1}])_i)_{\delta \in \Delta}$ being a family of partitions of $I$, we consider $(X^\delta)_{\delta \in \Delta}$, $(Z^\delta)_{\delta \in \Delta}$, and $(Y^\delta)_{\delta \in \Delta}$ that satisfy
\begin{align} \label{spaceX}
X^\delta & \subseteq \{w \in C(I;H^1_0(\Omega)) \colon w|_{(t_i^\delta,t_{i+1}^\delta)} \in {\mathcal P}_q(t_i^\delta,t_{i+1}^\delta) \otimes {\mathcal S}^{0,q}_{\tria^{\delta_i},0}\},\\ \nonumber
Y \supseteq \bar{Y}^\delta & \supseteq\{v \in Y \colon v|_{(t_i^\delta,t_{i+1}^\delta)} \in {\mathcal P}_q(t_i^\delta,t_{i+1}^\delta) \otimes {\mathcal S}^{0,q}_{\tria_S^{\delta_i},0}\},\footnotemark\\ \label{spaceZ}
Z^\delta & \subseteq \{\vec{q} \in L_2(I;H(\divv;\Omega)) \colon \vec{q}|_{(t_i^\delta,t_{i+1}^\delta)} \in {\mathcal P}_{q-1}(t_i^\delta,t_{i+1}^\delta) \otimes {\mathcal Z}^q_{\tria^{\delta_i}}\},
\end{align}
\footnotetext{If it were not for guaranteeing an oscillation error of higher order, then the polynomial degree in the time direction could be reduced to $q-1$.\label{foot5}}%
where $\divv {\mathcal Z}^q_{\tria^{\delta_i}} \subset {\mathcal S}_{\tria^{\delta_i}}^{-1,q-1}$, and where for each $i$, $\tria^{\delta_i}$ is some partition from $(\tria^\delta)_{\delta \in \Delta}$ with corresponding refinement $\tria_S^{\delta_i} \in (\tria_S^\delta)_{\delta \in \Delta}$.

Then for $\tria_S^\delta$ being a sufficiently deep refinement of $\tria^\delta$ such that  a projector $P_q^\delta$ as in Remark~\ref{remmie} exists, a Fortin interpolator $\bar{Q}^\delta$ as in \eqref{fortin2}
exists, and
\begin{align*}
(\bar{e}^\delta_\osc(g))^2 \lesssim \sum_i \sum_{T' \in \tria^{\delta_i} }\Big\{&
\inf_{p \in {\mathcal P}_q(t_i^\delta,t_{i+1}^\delta) \otimes L_2(T')} \|g- p\|^2_{L_2((t_i^\delta,t_{i+1}^\delta) \times T')} \\&+
(\diam T')^2 \inf_{p \in L_2(t_i^\delta,t_{i+1}^\delta) \otimes {\mathcal P}_{q-1}(T')} \|g- p\|^2_{L_2((t_i^\delta,t_{i+1}^\delta)\times T')}\Big\}.
\end{align*}
\end{theorem}

\begin{remark} In view of balancing the approximation rates for smooth functions by $X^\delta$ in $X$ and $Z^\delta$ in $Z$,
for $X^\delta$ and $Z^\delta$ being the spaces on the right-hand side of \eqref{spaceX} or \eqref{spaceZ} (in the latter case, possibly with $\vec{q} \in L_2(I;H(\divv;\Omega))$ reading as $\vec{q} \in C(I;H(\divv;\Omega))$),
 a natural choice for ${\mathcal Z}^q_{\tria^\delta} $
is the Raviart-Thomas space of index $q$ or the Brezzi-Douglas-Marini finite element space of index $\min(1,q-1)$ w.r.t.~$\tria^\delta$.

Notice that with these definitions of $X^\delta$ and $Z^\delta$, for sufficiently smooth $g$ the local oscillation error is of higher order than the expected local approximation error by $X^\delta$ in $X$ and $Z^\delta$ in $Z$.
\end{remark}

\begin{proof} For $(w,\vec{q}) \in X^\delta \times Z^\delta$, $v \in Y$, taking ${\mathcal Z}^q_{\tria^{\delta_i}} \subset H(\divv;\Omega)$ into account, integration-by-parts shows
$$
C(w,\vec{q})(v)=\int_I \int _\Omega \big({\textstyle \frac{\partial w}{\partial t}}-\divv_{\bf x} \vec{q}+\vec{b}\cdot \nabla_{\bf x} w+c w \big)v \,dt\,d{\bf x}.
$$
Let $(\bar{Q}_{\bf x}^\delta)_{\delta \in \Delta}$ denote a family of uniformly bounded operators $\bar{Q}_{\bf x}^\delta \in \cL(H^1_0(\Omega),H^1_0(\Omega))$ with the properties
\be \label{33}
\ran \bar{Q}_{\bf x}^\delta \subset \mathcal{S}^{0,q}_{\tria_S^\delta,0},\quad
\langle \mathcal{S}^{0,q}_{\tria^\delta,0}+\mathcal{S}^{-1,q-1}_{\tria^\delta}, \ran(\identity-\bar{Q}_{\bf x}^\delta)\rangle_{L_2(\Omega)}=0.
\ee
Moreover, let $Q^i_q$ be the $L_2(I)$-orthogonal projector onto ${\mathcal P}_q(t_i^\delta,t_{i+1}^\delta)$. Then, the operator
$\bar{Q}^\delta$ defined by
$$
(\bar{Q}^\delta v)|_{(t_i^\delta,t_{i+1}^\delta)\times \Omega}=(Q^i_q \otimes \bar{Q}_{\bf x}^{\delta_i}) v|_{(t_i^\delta,t_{i+1}^\delta)\times \Omega},
$$
satisfies the conditions of \eqref{fortin2}.

We again seek $\bar{Q}_{\bf x}^\delta$ of the form $\bar{Q}_{\bf x}^\delta=\bar{Q}_{\bf x}^{A,\delta}+\bar{Q}_{\bf x}^{B,\delta}+\bar{Q}_{\bf x}^{B,\delta} \bar{Q}_{\bf x}^{A,\delta}$
where
$$
\ran \bar{Q}_{\bf x}^{A,\delta}, \ran \bar{Q}_{\bf x}^{B,\delta} \subset \mathcal{S}^{0,q}_{\tria_S^\delta,0},\quad \big\langle \mathcal{S}^{0,q}_{\tria^\delta,0}+\mathcal{S}^{-1,q-1}_{\tria^\delta}, \ran (\identity-\bar{Q}_{\bf x}^{B,\delta})\big\rangle_{L_2(\Omega)}=0.
$$
Then from $\identity- \bar{Q}_{\bf x}^\delta=(\identity- \bar{Q}_{\bf x}^{B,\delta})(\identity- \bar{Q}_{\bf x}^{A,\delta})$,
we infer that \eqref{33} is satisfied.

We take $\bar{Q}_{\bf x}^{A,\delta}$ to be the Scott-Zhang quasi-interpolator onto $\mathcal{S}^{0,q}_{\tria_S^\delta,0}$, and $\bar{Q}_{\bf x}^{B,\delta}=P_q^\delta$ from Remark~\ref{remmie}. Writing $\bar{Q}_{\bf x}^\delta=\bar{Q}_{\bf x}^{A,\delta}+P_q^\delta(\identity-\bar{Q}_{\bf x}^{A,\delta})$,
the uniform boundedness of $\bar{Q}_{\bf x}^{A,\delta} \in \cL(H^1_0(\Omega),H^1_0(\Omega))$,
$\hbar_\delta^{-1}(\identity-\bar{Q}_{\bf x}^{A,\delta}) \in \cL(H^1_0(\Omega),L_2(\Omega))$,
as well as $\hbar_\delta^{-1} P_q^\delta \hbar_\delta \in \cL(L_2(\Omega),L_2(\Omega))$,
and $\|\cdot\|_{H^1(\Omega)} \lesssim \|h_\delta^{-1} \cdot\|_{L_2(\Omega)}$ on $\mathcal{S}^{0,q}_{\tria_S^\delta,0}$, imply the uniform boundedness of $\bar{Q}_{\bf x}^\delta \in \cL(H^1_0(\Omega),H^1_0(\Omega))$.

For any $p_i \in {\mathcal P}_q(t_i^\delta,t_{i+1}^\delta) \otimes L_2(\Omega)$, $\tilde{p}_i \in L_2(t_i^\delta,t_{i+1}^\delta) \otimes {\mathcal S}_{\tria^{\delta_i}}^{-1,q-1}$, and $y \in Y$, we have
\begin{align*}
 \big\langle g,(\identity - \bar{Q}^\delta)y \big\rangle_{L_2(I \times \Omega)}
=
\sum_i \sum_{T' \in \tria^{\delta_i}}\big\langle ((\identity -Q_q^i)\otimes \identity) (g- p_i),y \big\rangle_{L_2((t_i^\delta,t_{i+1}^\delta) \times T')}\\
+
\sum_i \sum_{T' \in \tria^{\delta_i}} \big \langle (\identity \otimes (\identity-P^\delta_q)')(g-\tilde{p}_i),Q_q^i \otimes (\identity-\bar{Q}_{\bf x}^{A,\delta_i}) y \big\rangle_{L_2((t_i^\delta,t_{i+1}^\delta) \times T')},
\end{align*}
since ${\mathcal P}_q(t_i^\delta,t_{i+1}^\delta)$ is reproduced by $Q_q^i$, and ${\mathcal S}_{\tria_S^{\delta_i}}^{-1,q-1}$ by ${P^\delta_q}'$.
The first double sum is bounded by a constant multiple of
$\sqrt{\sum_i \sum_{T' \in \tria^{\delta_i} }
\|g- p_i\|^2_{L_2((t_i^\delta,t_{i+1}^\delta) \times T')}}\,\|y\|_{L_2(I \times \Omega)}$.
On account of $\|\hbar_\delta {P_q^\delta}' \hbar_\delta^{-1}\|_{\cL(L_2(\Omega),L_2(\Omega))} \lesssim 1$ and
$$
\|(\identity-\bar{Q}_{\bf x}^{A,\delta_i}) v\|_{L_2(T')} \leq (\diam T')|v|_{H^1(\cup_{\{T'' \in \tria^{\delta_i}\colon T'' \cap T' \neq \emptyset\}}T'')},
$$
one infers that the second double sum can be bounded by a constant multiple of
$\sqrt{\sum_i \sum_{T' \in \tria^{\delta_i}}
(\diam T')^2 \|g- {\tilde p}_i\|^2_{L_2((t_i^\delta,t_{i+1}^\delta) \times T')}}\,\|y\|_{Y},
$
which completes the proof.
\end{proof}


\section{Numerical experiments} \label{sec:numexp}
 In this section we investigate our two formulations for solving the data assimilation problem numerically.
As underlying parabolic equation we select a simple heat equation posed on a spatial
domain $\Omega \subset \R^d$, and we take $T=1$, i.e.~$I = [0,1]$.

We use \emph{NGSolve}, \cite{247.06,247.065}, to assemble the system matrices  {and for spatial multigrid}.
We employ a preconditioned conjugate gradient scheme for solving the corresponding Schur
complement systems \eqref{Schur2ndorder} from \S\ref{Sprac1} and \eqref{Schur1storder} from \S\ref{Sprac2}.


\subsection{Unit interval}
We start with the simplest possible situation where $d=1$, and $\Omega := [0,1]$.
We subdivide $I$ and $\Omega$ into $1/h_\delta \in \N$ equal subintervals yielding
$\mathcal I^\delta$ and $\tria^\delta$ respectively.
We then select our discrete function spaces as tensor-product spaces of the form
\be \label{numspaces}
  X^\delta := \mathcal S^{0, 1}_{\mathcal I^\delta} \otimes \mathcal S^{0,1}_{\tria^\delta, 0}, ~~
  Y^\delta_\ell := \bar Y^\delta_\ell := \mathcal S^{-1, 1}_{\mathcal I^\delta} \otimes \mathcal S^{0,1}_{\tria^\delta_{\ell}, 0}, ~~
  Z^\delta := \mathcal S^{-1, 0}_{\mathcal I^\delta} \otimes \mathcal S^{0,1}_{\tria^\delta}
\ee
with $\tria^\delta_{\ell}$ constructed from $\tria^\delta$ by
 {recursively bisecting every subinterval
$\ell$ times.}

As follows from Sect.~\ref{sec:Fortin}, in our current setting, for both second order and FOSLS formulation, for
$\ell \geq 2$ uniformly bounded Fortin interpolators exist, i.e., \eqref{fortin} or \eqref{fortin2} are satisfied,
so that the minimizers $u_\eps^{\delta,\delta} \in X^\delta$ and $(\bar{u}_\eps^{\delta,\delta},\vec{p}_\eps^{\delta,\delta})\in X^\delta \times Z^\delta$ of $G_\eps^\delta$ or $H_\eps^\delta$ exist uniquely, and satisfy the a priori bounds from Theorem~\ref{thm1} or Theorem~\ref{thm10}, as well as the a posteriori bounds from Corollary~\ref{corol1} and Proposition~\ref{prop11}.
Moreover, these Fortin interpolators can be selected such that for sufficiently smooth datum $g$ the order of the data-oscillation term $e^\delta_\osc(g)$ or $\bar{e}^\delta_\osc(g)$, that are present in the a posteriori bounds, exceeds the generally best possible approximation order that can be expected.
Consequently, for $\ell=2$ the expressions
$\sqrt{G_{ {0}}^\delta(u_\eps^{\delta,\delta})}$ or $\sqrt{H_{{0}}^\delta(\bar u_{\eps}^{\delta,\delta}, \vec p_{\eps}^{\delta,\delta})}$
are, modulo a constant factor and oscillation terms of higher order, upper bounds for the $X_\eta$-norm of $e_\eps^\delta := u_0 - u_\eps^{\delta,\delta}$
or $\bar e_\eps^\delta := u_0 - \bar u_\eps^{\delta,\delta}$, respectively.

We will use this fact to explore in subsequent experiments also whether it would actually be harmful in practice to take $\ell {<} 2$
(resulting in lower computational cost). Note that the choice of the refinement level $\ell$ in $Y^\delta_\ell$ or $\bar{Y}^\delta_\ell$
affects, on the one hand, the quality of the numerical solution $u^{\delta,\delta}_\eps$ and, on the other hand, the reliability
of the a posteriori error bound. We will denote below by $\ell$ the refinement level used to compute $u^{\delta,\delta}_\eps$, and
by $L$ the refinement level  in $Y^\delta_L$ or $\bar{Y}^\delta_L$ used to compute the a posteriori error bounds.
Since these `reliable' a posteriori error bounds with $L=2$ apply to any function from $X^\delta$ (taking for the second argument of $H_\eps^\delta$ any argument from $Z^\delta$),
we have also used them, in particular, to assess the quality of the numerical approximations based on taking $Y^\delta_0$ or $\bar{Y}^\delta_0$ instead of $Y^\delta_2$ or $\bar{Y}^\delta_2$.

Equipping $\mathcal S^{-1,1}_{\mathcal I^\delta}$ with basis $\Phi_t^\delta$, and $\mathcal S^{0,1}_{\tria^\delta_{\ell}, 0}$ with $\Phi_{\bf x}^\delta$, the representation of the Riesz isometry $Y^\delta \rightarrow {Y^\delta}'$ reads as
${\bm R}^\delta=\langle \Phi_t^\delta, \Phi_t^\delta\rangle_{L_2(I)} \otimes \langle \nabla \Phi_{\bf x}^\delta, \nabla \Phi_{\bf x}^\delta\rangle_{L_2(\Omega)^d}$. Taking $\Phi_t^\delta$ to be $L_2(I)$-orthogonal, the first factor is diagonal and can be inverted directly. With ${\bf MG}^\delta_{\bf x} \eqsim \langle \nabla \Phi_{\bf x}^\delta, \nabla \Phi_{\bf x}^\delta\rangle_{L_2(\Omega)^d}^{-1}$ a symmetric spatial multigrid solver,
we define ${\bm K}^\delta_Y :=\langle \Phi_t^\delta, \Phi_t^\delta\rangle_{L_2(I)}^{-1}  \otimes {\bf MG}^\delta_{\bf x} \eqsim ({\bm R}^\delta)^{-1}$,
which can be applied at linear cost.
As explained in \S\ref{Sprac1} and \S\ref{Sprac2}, all considerations concerning the discrete approximations $u_\eps^{\delta,\delta}$ or $(\bar{u}^{\delta,\delta}_\eps,\vec{p}^{\delta,\delta}_\eps)$ remain valid when ${\bm R}^\delta$ in the matrix vector systems \eqref{DS} or \eqref{DS2nd}, that define these approximations, is replaced by $({\bm K}_Y^\delta)^{-1}$, and despite this replacement we continue to denote them by $u_\eps^{\delta,\delta}$ and $(\bar{u}^{\delta,\delta}_\eps,\vec{p}^{\delta,\delta}_\eps)$.

Equipping $X^\delta$ and $Z^\delta$ with similar tensor product bases, for the efficient iterative solution of the Schur complements \eqref{Schur2ndorder} or \eqref{Schur1storder} that define $u_\eps^{\delta,\delta}$ or $(\bar{u}^{\delta,\delta}_\eps,\vec{p}^{\delta,\delta}_\eps)$, for $W \in \{X,Z\}$
we use a preconditioner $\bm{K}_W^\delta$ that can be applied at linear cost and that is uniformly spectrally equivalent to the inverse of the representation of the Riesz isometry $W^\delta \rightarrow {W^\delta}'$. The construction of $\bm{K}_Z^\delta$ does not pose any difficulties, and for the construction of $\bm{K}_X^\delta$, that builds on a symmetric spatial multigrid solver that is robust for diffusion-reaction problems and the use of a wavelet basis in time that is stable in $L_2(I)$ and $H^1(I)$, we refer to \cite{249.991}.

\subsubsection{Consistent data} \label{sec:1d-consistent}
As a first test we prescribe the solution
\be \label{numtest}
u(t,x)=(t^3+1)\sin(\pi x),
\ee
take $\omega=[\frac14,\frac34]$ and use
data $(g,f)$ that are \emph{consistent} with $u$.
We computed $u_\eps^{\delta,\delta}$ and $(\bar u_{\eps}^{\delta,\delta}, \vec p_{\eps}^{\delta,\delta})$ for $\eps=h_\delta$, being the largest value of $\eps$ (up to a constant factor) for which we expect that the regularization doesn't spoil the order of convergence. Indeed, we expect that $e^\delta_\ap(u) \eqsim \bar{e}^\delta_\ap(u) \eqsim h_\delta$.
For both the second order and the FOSLS formulation, Figure~\ref{fig:1d-data} depicts the a posteriori error estimators $\sqrt{G_{{0}}^\delta(u_\eps^{\delta,\delta})}$ and $\sqrt{H_{{0}}^\delta(\bar u_{\eps}^{\delta,\delta}, \vec p_{\eps}^{\delta,\delta})}$ for $(\ell,L) \in \{(2,2), (0,2), (0,0)\}$
as a function of
$\dim X^\delta \approx h_\delta^{-2}$. The two formulations show very similar performance.
Moreover, the observed convergence rate $1/2$ is the best possible given our
discretization of piecewise linears on uniform meshes.
Concerning the choices for $\ell$ and $L$,
the results for $L=2$, that give reliable a posteriori error bounds, indicate that there is hardly any difference in the numerical approximations for test spaces $Y^\delta_2$ or $Y^\delta_0$,
respectively, $\bar{Y}^\delta_2$ or $\bar{Y}^\delta_0$, i.e., for $\ell=2$ or $\ell=0$,  so that we will take $\ell=0$ in the sequel.
 For the second order formulation, the value of the a posteriori estimator evaluated for $L=0$ is significantly smaller than that for $L=2$, but it shows qualitatively the same behaviour.
In view of  this observation, we  will also use $L=0$ in what follows.
\begin{figure}
  \begin{center}
    \includegraphics[width=\linewidth]{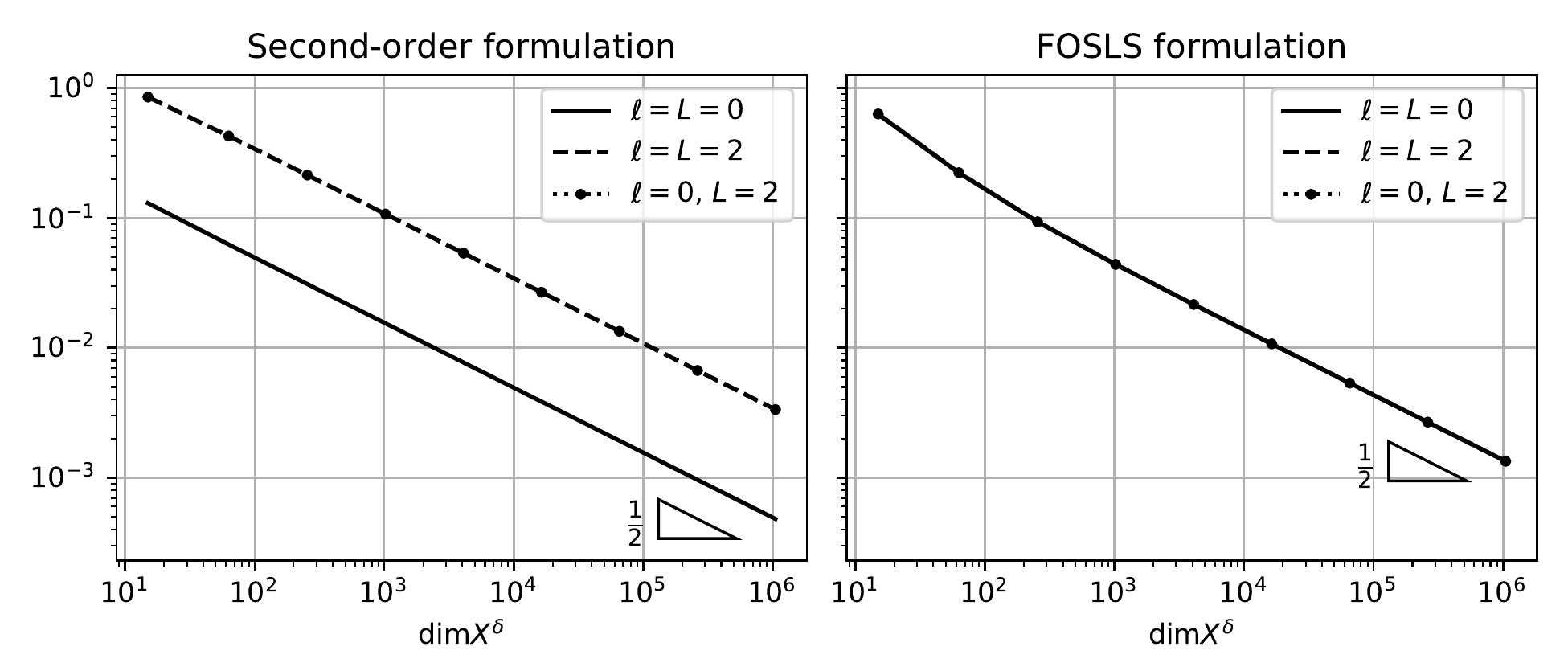}
  \end{center}
  \vspace{-1em}
  \caption{A posteriori error estimators for the unit interval problem with consistent data.}
\label{fig:1d-data}
\end{figure}

\subsubsection{System conditioning}
To see how the choice of $\eps$ affects the condition number of the preconditioned
systems~\eqref{Schur2ndorder} and~\eqref{Schur1storder}, we computed these
condition numbers for various $\eps$ and decreasing mesh sizes.
The results depicted in Figure~\ref{fig:1d-conditioning} illustrate that for constant $\eps >0$, the condition numbers are 
uniformly bounded. We show the values for $\ell=2$; for $\ell=0$, the values are very similar.
It also reveals that the growth in terms of $\eps$ is far more modest than the upper bound $\eqsim \eps^{-2}$ on these condition numbers that we found in Sect.~\ref{Sprac1} and~\ref{Sprac2}.
\begin{figure}
  \begin{center}
    \includegraphics[width=\linewidth]{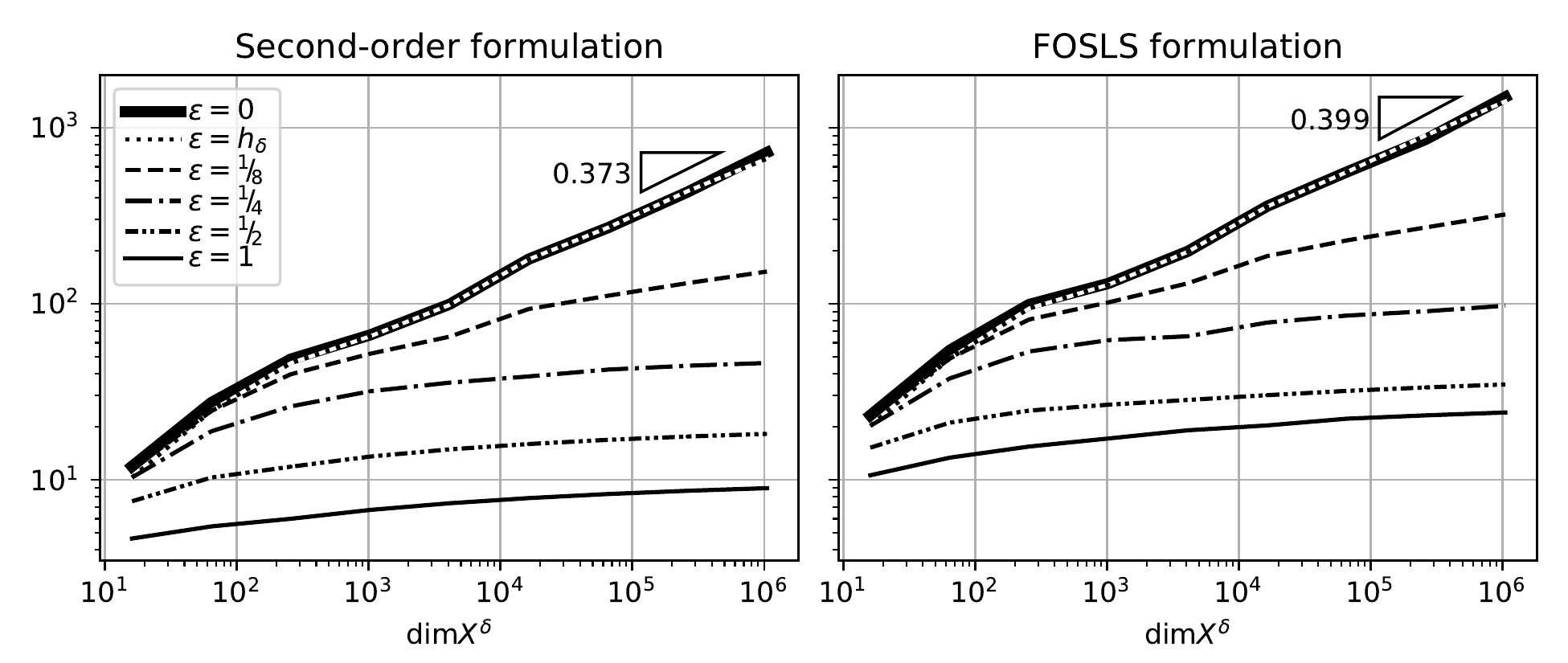}
  \end{center}
  \vspace{-1em}
  \caption{Condition numbers of the preconditioned system for a number of regularization parameters.}
\label{fig:1d-conditioning}
\end{figure}

\subsubsection{Inconsistent data}
\label{S1d-inconsistent}
In case of inconsistent data, there exists no state that exactly explains the data, and $e_\cons(u_0)>0$.
In this case, it does not make sense to approximate $u_0$ within a tolerance that is significantly smaller than $e_\cons(u_0)$.
Considering  for the second order formulation the a priori estimate
$$
\|u_0-u^{\delta,\delta}_\eps\|_{X_\eta} \lesssim  e_\cons(u_0) + e^\delta_\ap(u_0) +\eps\|\gamma_0 u_0\|_{L_2(\Omega)}
$$
from Theorem \ref{thm1} and taking the fact into account that choosing $\eps$ small has an only moderate effect on the conditioning of the preconditioned linear system, in the following we take $\eps$ of the order of the best possible approximation error that can be expected, so that $\eps\|\gamma_0 u_0\|_{L_2(\Omega)} \lesssim  e^\delta_\ap(u_0)$.
Then ideally we would like to stop refining our mesh as soon as $e^\delta_\ap(u_0)  \approx  e_\cons(u_0)$.
In order to achieve this we use the a posteriori error estimator.
From Corollary~\ref{corol} we know that
$$
e_\cons(u_0) \lesssim \sqrt{G_0^\delta(u^{\delta,\delta}_\eps)}  + e^\delta_\osc(g),
$$
where, following the reasoning from the proof of Proposition~\ref{prop2},
\begin{align*}
\sqrt{G_0^\delta(u^{\delta,\delta}_\eps)}&\leq
\sqrt{G_\eps^\delta(u^{\delta,\delta}_\eps)} \leq
\sqrt{G_\eps^\delta(P_{X^\delta}u_0)}
\leq \sqrt{G_\eps(P_{X^\delta}u_0)} \\
&\lesssim e^\delta_\ap(u_0)+ e_\cons(u_0)+\eps\|\gamma_0 u\|_{L_2(\Omega)}.
\end{align*}
We selected $(Y^\delta)_{\delta \in \Delta}$ such that, in any case for sufficiently smooth $g$, the order of $e^\delta_\osc(g)$ is equal or higher than the generally best possible order of the approximation error, so that $e^\delta_\osc(g) \lesssim  e^\delta_\ap(u_0)$. In view of our earlier assumption on $\eps$, we conclude that
$$
e_\cons(u_0) \lesssim \sqrt{G_0^\delta(u^{\delta,\delta}_\eps)} \lesssim e_\cons(u_0)+e^\delta_\ap(u_0).
$$

Exploiting a common uniform or adaptive refinement strategy, it can be expected that $e^\delta_\ap(u_0)$ decays with a certain algebraic rate $\rho<1$. Unless $e_\cons(u_0)$ is very large, it can therefore be expected that in the early stage of the iteration the a posteriori error estimator $\sqrt{G_0^\delta(u^{\delta,\delta}_\eps)}$ decays with this rate, whose value therefore can be monitored. By contrast, as soon as $e^\delta_\ap(u_0)$ has been reduced to $C e_\cons(u_0)$ for some constant $C>0$, the reduction of $\sqrt{G_0^\delta(u^{\delta,\delta}_\eps)}$ in the next step cannot be expected to be better than $\frac{1+C \rho}{1+C}$.
Taking $C=1/3$, our strategy will therefore be to stop the iteration as soon as the observed reduction of $\sqrt{G_0^\delta(u^{\delta,\delta}_\eps)}$  {is worse than} $\frac{1+C \rho}{1+C}$.

We have implemented this strategy, and a similar one for the FOSLS formulation, where we apply the discrete spaces as in \eqref{numspaces}, take $\eps=h_\delta$, and
again consider the unit interval problem \eqref{numtest} but now perturb the measured state $f=u|_{I \times \omega}$ by adding $\lambda \mathbb{1}$ to it for various values of $\lambda$.

{From the results in Figure~\ref{fig:1d-data-inconsistent}} we see that the error estimators decrease at first, but then stagnate in the aforementioned sense, at which point we exit the refinement loop (indicated by a $\times$-sign). Further refinement (indicated by the thin dashed lines) is not very useful, and the error estimators stabilize to a value just below $\lambda |\omega|^{1/2}$, being the $L_2(I \times \Omega)$-norm of the perturbation we added to the consistent $f$. Knowing that the error estimator converges to $e_\cons(u_0)$ (see Remark~\ref{rem:Gamma2}), we conclude that $(0, \mathbb{1}) \in Y'\times L_2(I \times \omega)$ is close to being orthogonal to $\ran B_\omega$.
We note that selecting $\ell=L=2$ produces very similar results.
\begin{figure}
  \begin{center}
    \includegraphics[width=\linewidth]{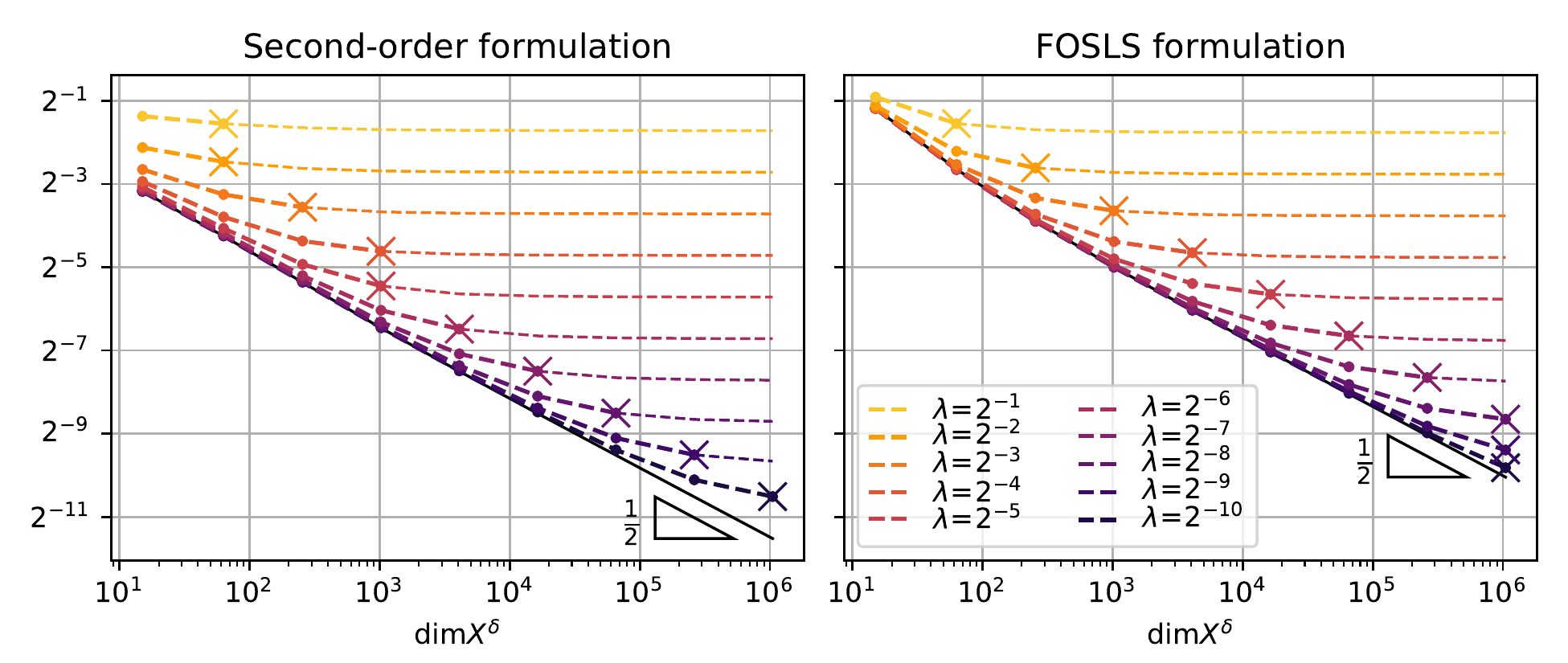}
  \end{center}
  \vspace{-1em}
  \caption{A posteriori error estimators for the unit interval problem with inconsistent data of varying amounts.}
\label{fig:1d-data-inconsistent}
\end{figure}

\subsection{Unit square}
\label{S2d-square}
 {We choose $\Omega := (0,1)^2$.}
We again subdivide $I$ into $1/h_\delta \in \N$ equal subintervals yielding $\mathcal I^\delta$, and
$\Omega$ first into $1/h_\delta \times 1/h_\delta$ squares and then into $2/h_\delta^2$
triangles by connecting the lower left and the upper right corner in each square yielding $\tria^\delta$. For a polynomial degree $q$, we take $\eps = h_\delta^q$.
 {Following the discussion in Sect.~\ref{sec:1d-consistent}, we select $\ell=L=0$ and take our discrete spaces as}
\[
  X^\delta_q := \mathcal S^{0, q}_{\mathcal I^\delta} \otimes \mathcal S^{0,q}_{\tria^\delta, 0}, ~~
  Y^\delta_q := \mathcal S^{-1, q}_{\mathcal I^\delta} \otimes \mathcal S^{0,q}_{\tria^\delta, 0}, ~~
  \bar Y^\delta_q := \mathcal S^{-1, q-1}_{\mathcal I^\delta} \otimes \mathcal S^{0,q}_{\tria^\delta, 0}, ~~
  Z^\delta_q := \mathcal S^{-1, q-1}_{\mathcal I^\delta} \otimes \mathcal Z^q_{\tria^\delta}
\]
where $\mathcal Z^q_{\tria^\delta}$ is the BDM space of index $\min(1, q-1)$.
Note that the degree $q-1$ in the temporal direction of $\bar{Y}_q^\delta$ guarantees an oscillation error of the same order as the approximation error, cf. Footnote~\ref{foot5}.

We define the preconditioners $\bm{K}^\delta_Y$, $\bm{K}^\delta_Z$, and $\bm{K}^\delta_X$ similar as in the 1D case.

\subsubsection{Consistent data}
We start with $\omega := [\tfrac{1}{4}, \tfrac{3}{4}]^2$ with the  prescribed solution
$u(t,x,y) := (t^3 + 1) \sin(\pi x) \sin(\pi y)$ and consistent data $(g, f)$.
Figure~\ref{fig:2d-unit-square} shows for both formulations and $q \in \{1,2\}$
the error estimators as a function of $\dim X^\delta \eqsim h_\delta^{-3}$.
The choice of preconditioners allows to reach the desired tolerance $\langle {\bf r}, {\bf K}_X^\delta {\bf r} \rangle \leq \eps^2 G_0^\delta(\tilde u_\eps^{\delta,\delta}) = 5.937 \cdot 10^{-13}$ for a system with $268\,434\,945$ unknowns in only 96 iterations.
The two formulations again exhibit similar performance, and the observed rate $q/3$
is the best possible, in line with Theorem~\ref{thm3}. Moreover, we see that
while theory is incomplete for the second order formulation in practice it works well also for piecewise quadratics.
\begin{figure}
\begin{center}
  \includegraphics[width=\linewidth]{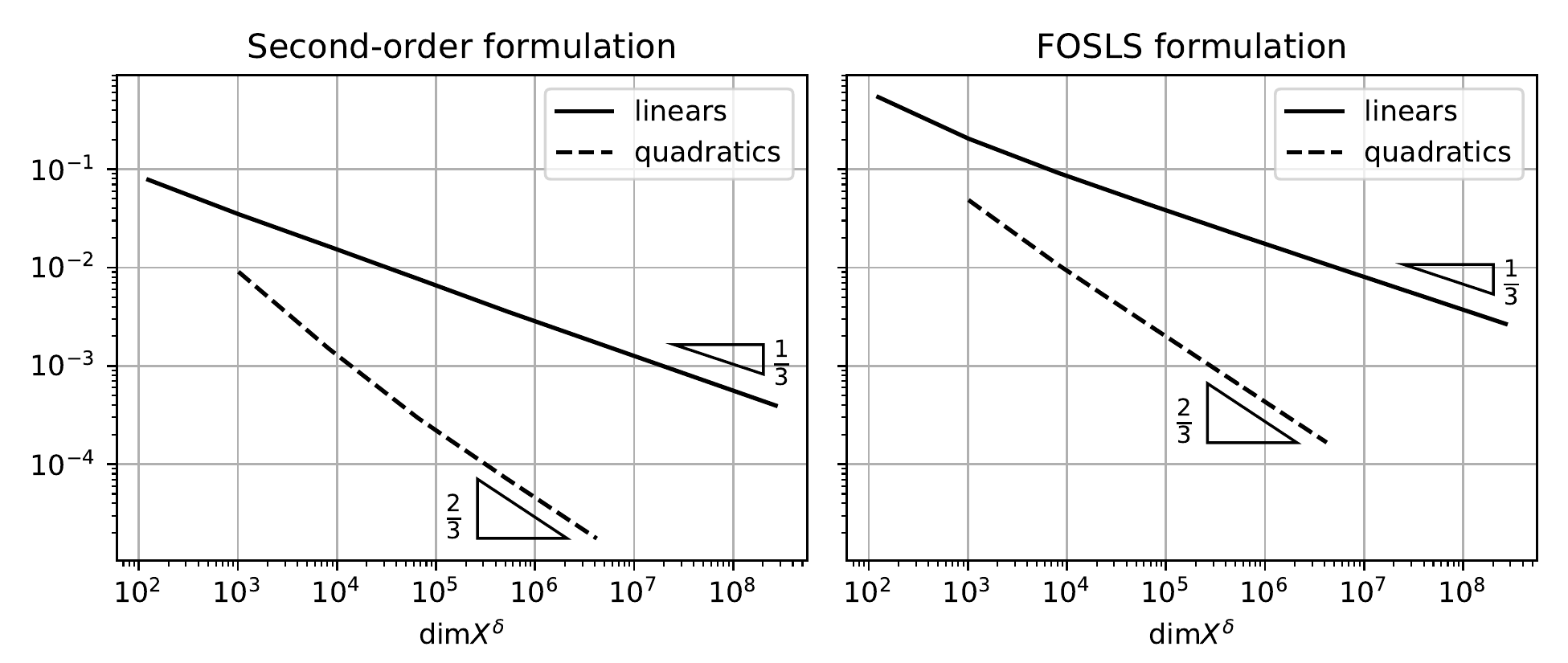}
\end{center}
\vspace{-1em}
  \caption{A posteriori error estimators for the consistent unit square problem with $\omega := [\tfrac{1}{4},\tfrac{3}{4}]^2$ using piecewise linears and quadratics.}
\label{fig:2d-unit-square}
\end{figure}

Thanks to $X_\eta \hookrightarrow C([\eta,1],L_2(\Omega))$, the time-slice errors $\|e_\eps^\delta(t)\|_{L_2(\Omega)}$
or $\|\bar e_\eps^\delta(t)\|_{L_2(\Omega)}$ are bounded by multiples of
$\|e_\eps^\delta\|_{X_\eta}$ or $\|\bar e_\eps^\delta\|_{X_\eta}$, respectively.
Figure~\ref{fig:2d-unit-timeslices} shows these time-slice errors for both formulations
using piecewise linears, i.e.,~$q=1$. We see that for both formulations, the
time-slice errors converge with the better rate $2/3$, and that these errors deteriorate
for $t \searrow 0$.
\begin{figure}
  \begin{center}
    \includegraphics[width=\linewidth]{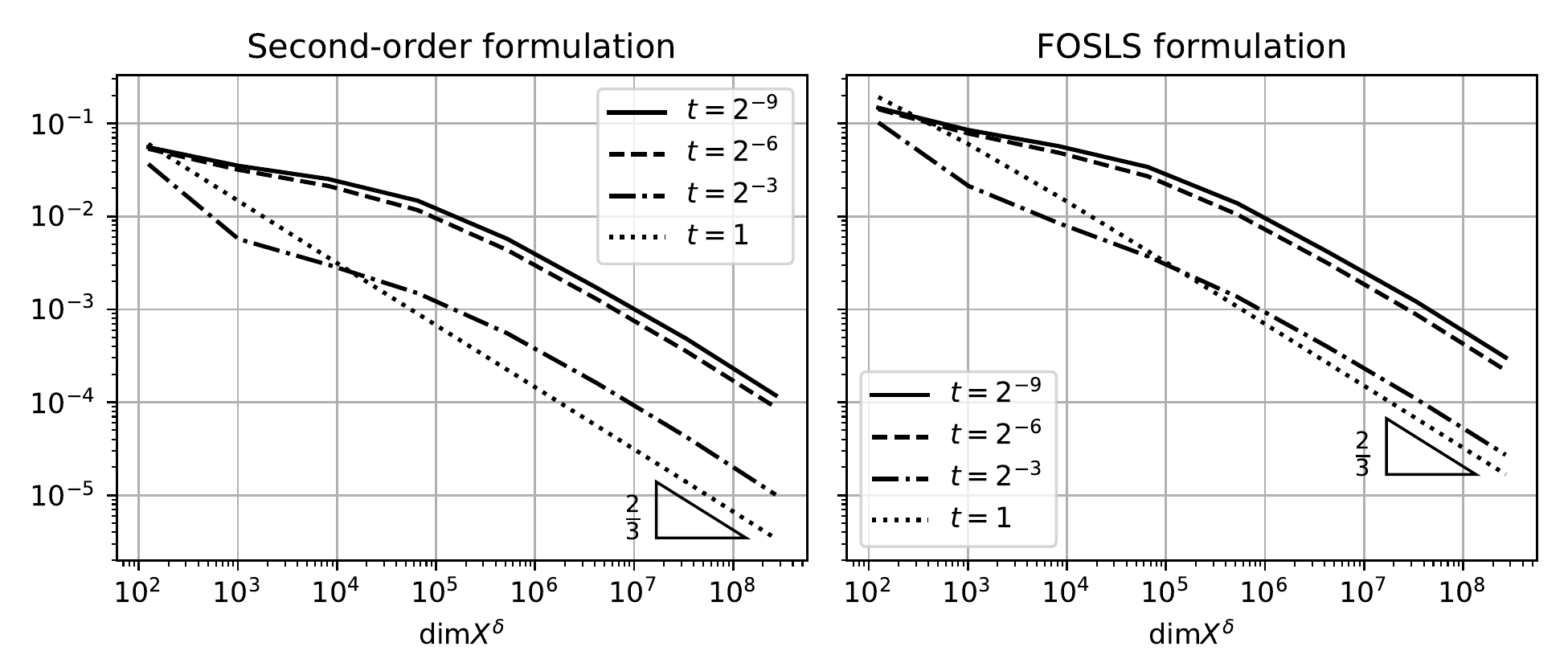}
  \end{center}
  \vspace{-1em}
  \caption{Time-slice errors for the consistent unit square problem with $\omega := [\tfrac{1}{4},\tfrac{3}{4}]^2$ using piecewise linears.}
\label{fig:2d-unit-timeslices}
\end{figure}

This deterioration becomes much stronger when $\diam \omega \to 0$: taking for
example $\omega := [\tfrac{7}{16},\tfrac{9}{16}]^2$, Figure~\ref{fig:2d-omega-smaller}
shows that while the error estimators remain nearly unchanged, the time-slice
errors fan-out an order of magnitude more than in the case of $\omega := [\tfrac{1}{4},\tfrac{3}{4}]^2$.
\begin{figure}
  \begin{center}
    \includegraphics[width=\linewidth]{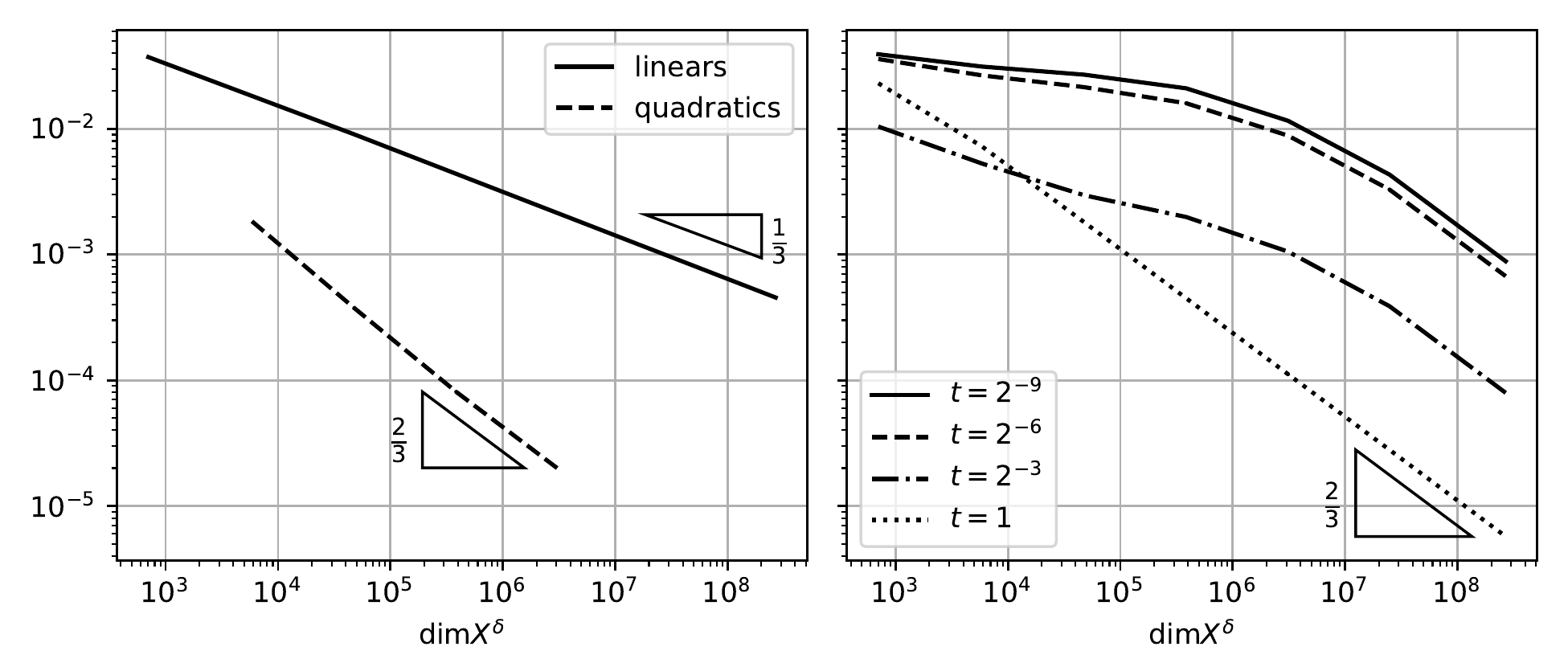}
  \end{center}
  \vspace{-1em}
  \caption{Second order formulation for the consistent unit square problem with $\omega := [\tfrac{7}{16},\tfrac{9}{16}]^2$. Left: error estimators; right: time-slice errors.}
\label{fig:2d-omega-smaller}
\end{figure}

\subsubsection{Inconsistent data}
 {Finally, we return to the case of inconsistent observational data. Again taking
$u(t,x,y) := (t^3 + 1) \sin(\pi x) \sin (\pi y)$ and $\omega := [\tfrac{1}{4}, \tfrac{3}{4}]^2$,
we select consistent forcing data $g := B u$ but \emph{perturbed} observational data
$f := u|_{I \times \omega} + \lambda \mathbb{1}$. Running the strategy outlined
in Sect.~\ref{S1d-inconsistent} with $C = 1/3$, with uniform refinements and choosing
$\eps = h_\delta$, yields the results of Figure~\ref{fig:2d-data-inconsistent}.
We see a situation very similar to the unit interval case: the error estimators
decrease at first and then stagnate, at which point we exit the refinement loop.
Error estimators again stabilize at around $\lambda |\omega|^{1/2}$.
\begin{figure}
  \begin{center}
    \includegraphics[width=\linewidth]{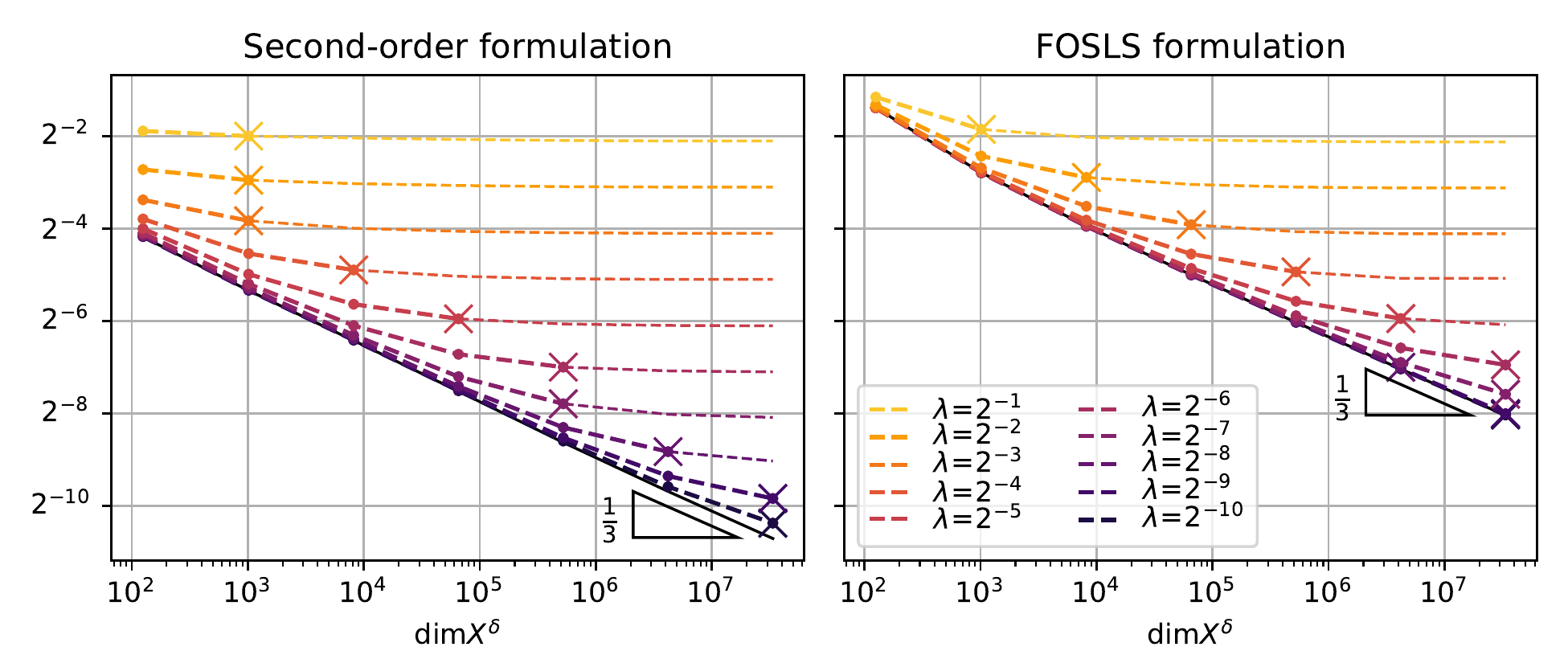}
  \end{center}
  \vspace{-1em}
  \caption{A posteriori error estimators for the unit square problem with inconsistent data of varying amounts.}
\label{fig:2d-data-inconsistent}
\end{figure}}


\section{Concluding Remarks} \label{sec:outlook}

We have seen that basing data assimilation for parabolic problems on infinite-dimensional stable time-space formulations
and related regularized least squares functionals has a number of conceptual advantages:
one obtains improved a  {priori} error estimates as well as a posteriori error bounds. Among other things the latter ones
are important for determining suitable stopping criteria for iterative solvers. Moreover, the design of corresponding preconditioners
is based on the infinite-dimensional variational formulation. We have shown that for each fixed regularization parameter
$\eps$ the preconditioner is optimal relative to the condition of the regularized problem so that the numerical complexity
remains under control.  Moreover,
the regularization parameter is disentangled from the discretizations which offers possibilities of optimizing its choice.

Furthermore, it will be interesting to relate the present results to the recent
state estimation concepts in \cite{22.56,45.48,192.5} providing error bounds in the full energy norm
$\|\cdot\|_X$ at the expense of certain stability factors reflecting a geometric relation between $X$ and
a certain space of functionals providing the data which, in turn, quantifies
the ``visibility'' of the true states by the sensors. A further important issue is to explore the use of  the obtained ``static'' methods
for ``dynamic data assimilation''. In this context the underlying stable variational formulations are expected to
be crucial for the use of certified reduced models.

 A price for building on the above  ``natural'' variational formulations -- in the sense that no {\em excess regularity} is implied --
is to properly discretize dual norms.
As pointed out earlier in Remark~\ref{rem:KarkulikFuehrer}, this is avoided in
 \cite{75.257} by replacing the term $\|C(w,\vec{q})-\tilde{g}\|_{Y'}^2$ in $H_\eps(w;\vec{q})$ (for $K= \identity$) by
 the $L_2$-residual     $\|C(w,\vec{q})-\tilde{g}\|_{L_2(I;L_2(\Omega))}^2$.
%
Being reduced to using then a somewhat weaker version of the Carleman estimate,
we would obtain a statement similar to that in Corollary \ref{cor:4.2}, but with an approximation
error $\bar{e}_{\ap}^\delta {(u)}$ measured in a somewhat stronger norm
\begin{align*}
&\min_{\{(w,\vec{q}) \in X^\delta \times Z^\delta\colon \partial_t w-\divv_{\bf x} \vec{q} \in L_2(I;L_2(\Omega))\}}  \|u-w\|_X+\|\nabla_{\bf x} u-\vec{q}\|_Z\\
&\qquad\qquad\qquad\qquad\qquad\qquad\qquad\qquad+
\|\partial_t u -\triangle u-(\partial_t w-\divv_{\bf x} \vec{q})\|_{L_2(I;L_2(\Omega))}.
\end{align*}

Finally,
optimal preconditioning in the space $\{(w,\vec{q}) \in X^\delta \times Z^\delta\colon \partial_t w-\divv_{\bf x} \vec{q} \in L_2(I;L_2(\Omega))\}$ equipped with the graph norm,  seems to be a challenge.


On the other hand, we also have the standard, second order formulation whose implementation is cheaper, and
at least in the above experiments performs well also in cases beyond the regime so far covered by theory.


\appendix
\section{Construction of the biorthogonal projector as in Remark~\ref{remmie} for $d=q=2$ and one red-refinement}
A basis for ${\mathcal S}_{\tria^\delta,0}^{0,2}+{\mathcal S}_{\tria^\delta}^{-1,1}$ is given by the sum of  the union over $T' \in \tria^\delta$ of the usual nodal basis for $P_1(T')$, and the union over the internal edges of $\tria^\delta$ of the continuous piecewise quadratic bubble associated to that edge, whose support extends to the two neighbouring triangles in $\tria^\delta$.
Indeed, one easily verifies that this set of functions is linearly independent, and that each function from either ${\mathcal S}_{\tria^\delta,0}^{0,2}$ or ${\mathcal S}_{\tria^\delta}^{-1,1}$ is in its span.

We consider the restriction of this basis to one $T' \in \tria^\delta$, and subsequently transfer it into a collection of functions on a
`reference triangle' $\hat{T}$ with $|\hat{T}|=1$ by an affine transformation. We denote the resulting functions as indicated in Figure~\ref{fig1}.
\begin{figure}
\begin{center}
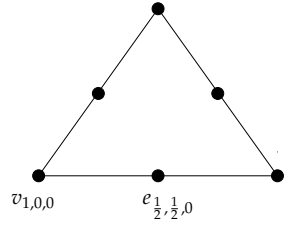
\end{center}
\caption{Notation of the basis functions at the `primal side' with the indices of the
missing basis functions obtained by permuting the barycentric coordinates.}
\label{fig1}
\end{figure}
At the `dual side', we consider the nodal basis of the continuous piecewise quadratics w.r.t.~the red-refinement of $\hat{T}$, where we omit the basis functions associated to the vertices of $\hat{T}$.
We denote these basis functions as indicated in Figure~\ref{fig2}.
\begin{figure}
\begin{center}
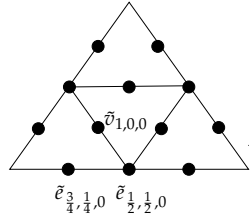
\end{center}
\caption{Notation at the `dual side'.}
\label{fig2}
\end{figure}

We now apply the following transformations:
\begin{enumerate}
\item On the primal side, we redefine
$$
e_{\frac12,\frac12,0}\leftarrow e_{\frac12,\frac12,0}-{\textstyle \frac{7}{10}}(v_{1,0,0}+v_{0,1,0})+{\textstyle \frac{7}{30}}v_{0,0,1},
$$
and update $e_{\frac12,0\frac12}$ and $e_{0,\frac12,\frac12}$ analogously.
As a consequence, we obtain $\Span\{e_{\frac12,\frac12,0},e_{\frac12,0,\frac12},e_{0,\frac12,\frac12}\} \perp \Span\{\tilde{v}_{1,0,0},\tilde{v}_{0,1,0},\tilde{v}_{0,0,1}\}$.
\item On the dual side, we redefine
$$
 \tilde{e}_{\frac12,\frac12,0} \leftarrow {\textstyle \frac{1}{102}}\tilde{e}_{\frac12,\frac12,0}-{\textstyle \frac{7}{2312}}(\tilde{e}_{\frac34,\frac14,0}+\tilde{e}_{\frac14,\frac34,0}),
 $$
and update $\tilde{e}_{\frac12,0\frac12}$ and $\tilde{e}_{0,\frac12,\frac12}$ analogously.
Consequently, $\{e_{\frac12,\frac12,0},e_{\frac12,0,\frac12},e_{0,\frac12,\frac12}\}$ and $\{\tilde{e}_{\frac12,\frac12,0},\tilde{e}_{\frac12,0,\frac12},\tilde{e}_{0,\frac12,\frac12}\}$ became biorthogonal.
The functions $e_{\pi(\frac34,\frac14,0)}$ for any permutation $\pi$, will not play any  role anymore, and will be ignored.

\item On  the dual side, we redefine
$$
\left[\begin{array}{@{}c@{}} \tilde{v}_{1,0,0} \\ \tilde{v}_{0,1,0}\\\tilde{v}_{0,0,1}\end{array}\right] \leftarrow
  12\left[\begin{array}{@{}rrr@{}} 3 & -1 & -1\\ -1 & 3 & -1\\ -1 & -1 & 3\\ \end{array}\right]\left[\begin{array}{@{}c@{}} \tilde{v}_{1,0,0} \\ \tilde{v}_{0,1,0}\\\tilde{v}_{0,0,1}\end{array}\right].
$$
Consequently, $\{v_{1,0,0},v_{0,1,0},v_{0,0,1}\}$ and $\{\tilde{v}_{1,0,0},\tilde{v}_{0,1,0},\tilde{v}_{0,0,1}\}$ became biorthogonal.
\end{enumerate}
After these 3 steps,
the $6 \times 6$ `local generalized mass matrix'
that contains the $L_2(\hat{T})$-inner products between all primal functions, grouped into $v$- and $e$-functions, and all (remaining) dual functions, grouped into $\tilde v$- and $\tilde e$-functions, has the $2 \times 2$ block structure
  $\left[\begin{array}{@{}cc@{}} \identity & \tfrac{9}{32} \identity - \tfrac{31}{32} \1 \\ 0 & \identity \end{array}\right]$,
with $\1$ the $3 \times 3$ all-ones matrix (and with the $\tilde e$-functions ordered as the `opposite' $v$-functions).
The invertibility of this matrix confirms that both collections of 6 primal and 6 dual functions are linearly independent.

We use these primal and dual functions on the reference triangle $\hat{T}$ to construct collections of primal and dual functions on $\Omega$ by the usual lifting by means of an affine bijection between $\hat{T}$ and any $T' \in \tria^\delta$. When doing so, we connect the functions of $e$ or $\tilde e$-type continuously over `their' edges, and omit them on edges on $\partial\Omega$.

Each function of $v$ or $\tilde v$-type is supported on one $T' \in \tria^\delta$, and we multiply them by the factor  $|T'|^{-\frac12}$. The functions of $e$ or $\tilde e$-type are supported on two adjacent $T',T'' \in \tria^\delta$, and we multiply them by the factor  $(|T'|+|T''|)^{-\frac12}$.

By their construction, the resulting primal and dual collections, denoted by $\Phi^\delta$ and $\tilde{\Phi}^\delta$, are uniformly $L_2(\Omega)$-Riesz systems, with mass matrices whose extremal eigenvalues are inside the interval spanned by the extremal eigenvalues of the corresponding primal or dual mass matrices on the reference triangle.

Furthermore, $\Span\Phi^\delta={\mathcal S}_{\tria^\delta,0}^{0,2}+{\mathcal S}_{\tria^\delta}^{-1,1}$,
and $\Span \tilde{\Phi}^\delta \subset {\mathcal S}_{\tria_S^\delta,0}^{0,2}$, with $\tria_S^\delta$ being constructed from $\tria^\delta$ by one uniform red-refinement.

The generalized mass matrix, i.e., the matrix with the $L_2(\Omega)$-inner products between all primal functions, grouped into $v$- and $e$-functions, and all dual functions, grouped into $\tilde v$- and $\tilde e$-functions, has the $2 \times 2$ block structure $\left[\begin{array}{@{}cc@{}} \identity & * \\ 0 & \identity \end{array}\right]$.
The uniform $L_2(\Omega)$-Riesz basis property of both $\Phi^\delta$ and $\tilde{\Phi}^\delta$ shows that the spectral norm of the non-zero off-diagonal block is uniformly bounded.
By now redefining $\Phi^\delta \leftarrow \left[\begin{array}{@{}cc@{}} \identity & -* \\ 0 & \identity \end{array}\right] \Phi^\delta$, we obtain  primal and dual uniformly $L_2(\Omega)$-Riesz systems that are \emph{biorthogonal}, where
$\Span \Phi^\delta={\mathcal S}_{\tria^\delta,0}^{0,2}+{\mathcal S}_{\tria^\delta}^{-1,1}$ and
$\Span \tilde{\Phi}^\delta \subset {\mathcal S}_{\tria_S^\delta,0}^{0,2}$.

In view of the supports of the dual functions, and those of the primal functions before the last transformation, we infer that the support of a function in $\tilde{\Phi}^\delta$ is contained in either one $T' \in \tria^\delta$ ($\tilde v$-type), or in the union of two triangles from $\tria^\delta$ that share an edge ($\tilde e$-type),
and that the support of a function in $\Phi^\delta$ is contained in either the union of two triangles from $\tria^\delta$ that share an edge ($e$-type), or in the union of $T' \in \tria^\delta$ and those at most three $T''  \in \tria^\delta$ that share an edge with $T'$.
We conclude that the biorthogonal projector $P_2^\delta\colon u \mapsto \langle u, \Phi^\delta\rangle_{L_2(\Omega)} \tilde{\Phi}^\delta$ satisfies both conditions \eqref{35relax} and \eqref{36}.



\newcommand{\etalchar}[1]{$^{#1}$}

\end{document}

%% file: numberingprimal.pdf_t
\begin{picture}(0,0)%
\includegraphics{numberingprimal.pdf}%
\end{picture}%
\setlength{\unitlength}{1973sp}%
\begingroup\makeatletter\ifx\SetFigFont\undefined%
\gdef\SetFigFont#1#2#3#4#5{%
  \reset@font\fontsize{#1}{#2pt}%
  \fontfamily{#3}\fontseries{#4}\fontshape{#5}%
  \selectfont}%
\fi\endgroup%
\begin{picture}(3443,2639)(2941,-3817)
\put(4606,-3721){\makebox(0,0)[lb]{\smash{{\SetFigFont{8}{9.6}{\rmdefault}{\mddefault}{\updefault}{\color[rgb]{0,0,0}$e_{\frac12,\frac12,0}$}%
}}}}
\put(2956,-3721){\makebox(0,0)[lb]{\smash{{\SetFigFont{8}{9.6}{\rmdefault}{\mddefault}{\updefault}{\color[rgb]{0,0,0}$v_{1,0,0}$}%
}}}}
\end{picture}%

%% file: numberingdual.pdf_t
\begin{picture}(0,0)%
\includegraphics{numberingdual.pdf}%
\end{picture}%
\setlength{\unitlength}{1973sp}%
\begingroup\makeatletter\ifx\SetFigFont\undefined%
\gdef\SetFigFont#1#2#3#4#5{%
  \reset@font\fontsize{#1}{#2pt}%
  \fontfamily{#3}\fontseries{#4}\fontshape{#5}%
  \selectfont}%
\fi\endgroup%
\begin{picture}(3024,2583)(3289,-3832)
\put(4486,-2806){\makebox(0,0)[lb]{\smash{{\SetFigFont{8}{9.6}{\rmdefault}{\mddefault}{\updefault}{\color[rgb]{0,0,0}$\tilde{v}_{1,0,0}$}%
}}}}
\put(4636,-3736){\makebox(0,0)[lb]{\smash{{\SetFigFont{8}{9.6}{\rmdefault}{\mddefault}{\updefault}{\color[rgb]{0,0,0}$\tilde{e}_{\frac12,\frac12,0}$}%
}}}}
\put(3871,-3736){\makebox(0,0)[lb]{\smash{{\SetFigFont{8}{9.6}{\rmdefault}{\mddefault}{\updefault}{\color[rgb]{0,0,0}$\tilde{e}_{\frac34,\frac14,0}$}%
}}}}
\end{picture}%